\documentclass{article}
\usepackage[left=1in,right=1in,top=1in,bottom=1in]{geometry}
\usepackage{amsfonts,amsmath,mathrsfs,bbm,amsthm, amssymb, enumitem,comment,apptools}
\usepackage{xcolor}
\usepackage{authblk}
\usepackage{appendix}
\usepackage{enumitem}
\usepackage{graphicx}
\usepackage[hidelinks]{hyperref}
\usepackage{url}
\DeclareMathOperator{\dist}{dist}
\DeclareMathOperator{\id}{id}

\DeclareMathOperator{\dom}{dom}

\newcommand{\C}{\mathbb C}

\newcommand{\A}{\mathscr{A}}

\newcommand{\M}{\mathbb{M}}
\newcommand{\N}{\mathbb N}

\newcommand{\X}{\mathfrak{X}}
\DeclareMathOperator{\diam}{diam}

\renewcommand{\P}{\mathbb{P}}

\newcommand{\E}{\mathbb{E}}

\DeclareMathOperator{\Aut}{Aut}

\newcommand{\<}{\langle}
\renewcommand{\>}{\rangle}
\newcommand{\aff}{\ \widetilde{\in}\:}

\let\emptyset\varnothing

\newtheorem{thm}{Theorem}[section]
\newtheorem{thmx}{Theorem}

\newtheorem{prop}[thm]{Proposition}
\newtheorem{lem}[thm]{Lemma}
\newtheorem{cor}[thm]{Corollary}
\newtheorem*{thm*}{Theorem}
\theoremstyle{definition}
\newtheorem{define}[thm]{Definition}
\newtheorem{exmp}[thm]{Example}
\newtheorem{rmk}[thm]{Remark}

\newtheorem{note}[thm]{Notation}

\begin{document}
\title{\huge\textbf{Ergodic Quantum Processes\\ on Finite von Neumann Algebras}}
\author{
Brent Nelson\thanks{Department of Mathematics, Michigan State University \hfill \url{brent@math.msu.edu}} \ 
and 
Eric B. Roon\thanks{Department of Mathematics, The University of Arizona \hfill \url{ebroon@math.arizona.edu}}
}
\date{}                     
\setcounter{Maxaffil}{0}
\renewcommand\Affilfont{\itshape\small}
\maketitle
\begin{abstract}
\noindent Let $(M,\tau)$ be a tracial von Neumann algebra with a separable predual and let $(\Omega, \mathbb{P})$ be a probability space. A bounded positive random linear operator on $L^1(M,\tau)$ is a map $\gamma\colon \Omega \times L^1(M,\tau)\to L^1(M,\tau)$ so that $\tau(\gamma_\omega(x)a)$ is measurable for all $x\in L^1(M,\tau)$ and $a\in M$, and $x\mapsto \gamma_\omega(x)$ is bounded, positive, and linear almost surely. Given an ergodic $T\in \text{Aut}(\Omega, \mathbb{P})$, we study quantum processes of the form $\gamma_{T^n\omega}\circ \gamma_{T^{n-1}\omega}\circ \cdots \circ \gamma_{T^m\omega}$ for $m,n\in \mathbb{Z}$. Using the Hennion metric introduced in \cite{jeff}, we show that under reasonable assumptions such processes collapse to replacement channels exponentially fast almost surely. Of particular interest is the case when $\gamma_\omega$ is the predual of a normal positive linear map on $M$. As an example application, we study the clustering properties of normal states that are generated by such random linear operators. These results offer an infinite dimensional generalization of the theorems in \cite{jeff}.
\end{abstract}

\section*{Introduction}
In quantum information theory, changes in a quantum system are modelled by quantum channels. In the Schr\"odinger picture, quantum channels take the form of trace-preserving completely positive maps on states. Dually, in the Heisenberg picture of quantum mechanics, observable quantities are represented by operators in a C*-algebra (or von Neumann algebra), and quantum channels are then unital completely positive (normal) maps on that algebra. Physically, a quantum channel represents the dynamics of observables when the quantum system is \emph{weakly coupled} to an environment (or reservoir) into which information about the evolving system can escape \cite{BreuerPetruccione}. Asymptotic properties of these dynamics play an important role throughout the mathematical physics literature. 
	
 This dynamical interpretation of quantum channels (ucp maps) is not the only way of incorporating disorder into a quantum mechanical model. Recent years have seen an advance in quantum mechanical models that incorporate \emph{large disorder}, by randomizing the quantum channel that evolves the system (see the non-exhaustive list \cite{BougronJoyePillet, GGJN18, jeff, pathirana2023law}). It is natural, then, to analyze the asymptotic evolution of the system up to probability. That is, given a family of \emph{random} quantum channels $\{\phi_n\colon n\in \mathbb{N}\}$, one seeks to analyze the asymptotic behavior of
    \[
        \phi_n\circ \cdots \circ\phi_1,
    \] 
or in the dual Schr\"{o}dinger picture one would instead consider $\phi_1^*\circ\cdots \circ \phi_n^*$. In the work \cite{jeff}, the authors use ergodicity to obtain almost sure asymptotic estimates for non-independent homogeneously distributed quantum channels. More precisely, given an ergodic transformation $T$ on a probability space $(\Omega,\P)$ and a random variable $\gamma_\omega$ valued in quantum channels on states, Movassagh and Schenker prove a number of results for the family $\{\gamma_{T^n\omega}\colon n\in \mathbb{N}\}$, including almost sure convergence of $\gamma_{\omega}\circ \gamma_{T\omega}\circ \cdots \circ \gamma_{T^n\omega}$ to a replacement channel (see \cite[Theorem 2]{jeff}). As an application of these results, they also show that certain matrix product states exhibit an almost sure clustering estimate (see \cite[Theorem 3]{jeff}). 
 
 This work aims to provide an effective generalization of the work of \cite{jeff} to the case when the observables form a finite von Neumann algebra; that is, a von Neumann algebra admitting a faithful normal tracial state. A basic but infinite dimensional example is the hyperfinite $\mathrm{II}_1$ factor $\mathcal{R}$, which is physically relevant because it arises as the weak operator topology closure of the local algebra associated to a spin chain where the on-site observable algebras consist of the $2\times 2$ matrices (or more generally, the $n\times n$ matrices for any integer $n\geq 2$). Other sources of examples include representations of discrete groups, actions of groups on probability spaces, and measurable equivalence relations.
 
 The principal tool we use to carry out our analysis is the metric first introduced by Hennion in his 1997 paper \cite{Hennion}. Hennion was originally concerned with infinite products of random positive definite matrices and their convergence properties, and he used his metric as a means to study their rates of convergence. Movassagh and Schenker \cite{jeff} provide a finite dimensional noncommutative version of Hennion’s results on the $n\times n$ complex matrices $\mathbb{M}_{n}$. Both papers rely on what we call the $m$-\textit{quantity} of two positive matrices (or, in the case of \cite{Hennion}, vectors) $X$ and $Y$ given by $m(X,Y) := \max \{ \lambda \in \mathbb{R} \colon \lambda Y \le X\}$. In the case that $M$ is finite with faithful normal trace $\tau$, we recall that the normal state space $S\subset L^1(M,\tau)$ can be canonically identified with the set of unit-trace, positive, closed operators affiliated to $M$ \cite[Chapter 7]{ap}. Thus, from the $m$-quantity one can form a bounded metric $d$ on the normal state space of $M$ via 
    \[
        d(x,y) = \frac{1 - m(x,y)m(y,x)}{1+m(x,y)m(y,x)} \qquad\qquad x,y\in S\subset L^1(M,\tau).
    \]
We call this \emph{Hennion's metric}, and we study its geometric properties in Section~\ref{sec:dtop}. In addition to extending known results to the infinite dimensional case, we exhibit new results about the disconnected components of $S$ (see Theorem~\ref{thm:geometry}).
 
To each positive linear map $\gamma$ on $L^1(M,\tau)$, one can induce a projective action on $S$ by $\gamma \cdot x := \frac{1}{\tau(\gamma(x))} \gamma(x)$. Provided that $\tau\circ \gamma$ is non-zero on $S$, one can associate to $\gamma$ the Lipschitz constant 
    \[
        c(\gamma) := \sup_{\substack{x,y\in S\\x\neq y}} \frac{d(\gamma \cdot x, \gamma \cdot y)}{d(x,y)}.
    \] 
When $c(\gamma)<1$, we say $\gamma$ is a \emph{strict Hennion contraction}. Many properties of these maps are established in Section~\ref{sec:cmap}, including a complete classification (see Theorem~\ref{thm:strict_Hennion_projective_actions}). The duality $L^1(M,\tau)^*\cong M$ implies strict Hennion contractions can also arise from normal positive linear maps on $M$, and indeed we determine precisely when this occurs in Subsection~\ref{subsubsec:SHCfromnormal}.

In Section~\ref{sec:EQP}, we consider \emph{ergodic quantum processes}: compositions of random quantum channels on $L^1(M,\tau)$ evolving under an ergodic transformation. Our first main result roughly says that such processes collapse to a replacement channel almost surely, provided there is a chance that the process eventually contracts in Hennion's metric:

\begin{thmx}[{Theorem~\ref{thm:right_convergence}}]\label{thmx:A}
Let $(M,\tau)$ be a tracial von Neumann algebra with a separable predual, let $(\Omega, \P)$ be a probability space equipped with ergodic $T\in \Aut(\Omega, \P)$, and let $\gamma_\omega\colon L^1(M,\tau)\to L^1(M,\tau)$ be a bounded positive faithful random linear operator. Suppose that
    \[
        \P[\exists m \colon c(\gamma_{\omega}\circ \gamma_{T\omega }\circ \cdots \circ \gamma_{T^m \omega})<1]>0.
    \]
Then there is a state-valued random variable $X_\omega\in S$ so that for all $x\in S$ one has
        \[
           \lim_{m\to - \infty} \|\gamma_{\omega}\circ \gamma_{T\omega }\circ \cdots \circ \gamma_{T^m \omega} \cdot x - X_\omega\|_1 = 0
        \] 
almost surely.
\end{thmx}

Separability of the predual $L^1(M,\tau)$ in the above theorem is used extensively to avoid measurability issues, and the probabilistic assumption is analogous to \cite[Assumption 1]{jeff} (see the discussion preceding Lemma~\ref{lem:a_not_so_mild_hypothesis}). The rate of convergence is controlled by $c(\gamma_{\omega}\circ \gamma_{T\omega }\circ \cdots \circ \gamma_{T^m \omega})$, and so it depends on $\omega\in \Omega$ but is independent of $x\in S$. To show that these Lipschitz constants tend to zero almost surely, we use Kingman's ergodic theorem (see Theorem~\ref{thm:kingman}), and in fact the rate of convergence is exponentially fast almost surely (see Lemma~\ref{lem:randomconstproperties}).

Our second main result concerns ergodic quantum processes on $M$ rather than $L^1(M,\tau)$ and is essentially dual to Theorem~\ref{thmx:A}. Under similar assumptions, such processes also collapse to a replacement channel almost surely:

\begin{thmx}[{Theorem~\ref{thm:left_convergence_on_M}}]\label{thmx:B}
Let $(M,\tau)$ be a tracial von Neumann algebra with a separable predual, let $(\Omega, \P)$ be a probability space equipped with ergodic $T\in \Aut(\Omega, \P)$, and let $\phi_\omega\colon M\to M$ be a normal unital positive random linear operator. Suppose that
    \[
        \P[\exists n \colon c((\phi_{T^n\omega} \circ \cdots \circ \phi_\omega)_*)<1]>0.
    \]
Then there is a state-valued random variable $Y_\omega \in S$ so that for all $a\in M$ one has
    \[
        \lim_{n\to \infty} \| \phi_{T^n\omega}\circ\cdots\cdot \circ\phi_{\omega}(a) - \tau(a Y_\omega)\|_\infty = 0.
    \]
almost surely.
\end{thmx}

In fact, we are able to prove the above theorem when $\phi_\omega(1)$ is only assumed to be almost surely invertible. In this case, one must instead consider a normalized process of the form
    \[
       a\mapsto \left(\phi_{T^n\omega}\circ\cdots\cdot \circ\phi_{\omega}(1)^{-\frac12} \right)\phi_{T^n\omega}\circ\cdots\cdot \circ\phi_{\omega}(a) \left(\phi_{T^n\omega}\circ\cdots\cdot \circ\phi_{\omega}(1)^{-\frac12} \right).
    \]
Theorems~\ref{thmx:A} and \ref{thmx:B} recover the asymptotic results in \cite[Theorem 1]{jeff} for the finite dimensional algebra $(\mathbb{M}_n,\frac{1}{n}\text{Tr})$. Compared to \cite[Theorem 1]{jeff}, the above theorems have left-right asymmetries that are a consequence of $L^1(M,\tau)\cong M$ \emph{failing} in the infinite dimensional case.

The deterministic versions of Theorems~\ref{thmx:A} and \ref{thmx:B} can be compared with Yeadon's mean ergodic theorems for semifinite von Neumann algebras (see \cite{Yea77,Yea80}). Indeed, suppose $\gamma\colon L^1(M,\tau)\to L^1(M,\tau)$ is a unital $\tau$-preserving positive linear map, and denote its dual map by $\phi:=\gamma^*$. Then \cite[Theorem 4.2]{Yea80} implies that for every $x\in L^1(M,\tau)$ and $a\in M$ there exists $\hat{x}\in L^1(M,\tau)$ and $\hat{a}\in M$ satisfying
    \[
        \lim_{n\to\infty} \left\| \frac1n \sum_{k=0}^{n-1} \gamma^k(x) - \hat{x}\right\|_1=0 \qquad\qquad \text{ and }\qquad\qquad \lim_{n\to\infty} \left\| \frac1n \sum_{k=0}^{n-1} \phi^k(a) - \hat{a} \right\|_\infty =0.
    \]
If $\inf_{n\geq 0} c(\gamma^n)<1$, then Theorems~\ref{thmx:A} and \ref{thmx:B} imply there exists $X\in S$ so that $\hat{x}=X$ for all $x\in S$ and $\hat{a} = \tau(aX)$ for all $a\in M$. Moreover, in this case the above convergences can be upgraded to 
    \[
        \lim_{n\to\infty} \| \gamma^n(x) - X\|_1 =0 \qquad\qquad \text{ and }\qquad\qquad \lim_{n\to\infty} \| \phi^n(a) - \tau(aX)\|_\infty =0.
    \]
(Note that $\gamma(x)=\gamma\cdot x$ here since $\gamma$ is $\tau$-preserving.)

Part of the work done in \cite{jeff} is to understand the clustering properties of certain matrix product states which are generated by a family of homogeneously-distributed random matrices. Physically, these correspond to \emph{random} states on a spin chain. Therefore it is an interesting question to understand what happens in the case when the on-site algebras are infinite dimensional. Given a von Neumann algebra $M$, let $\{M_n\colon n\in \mathbb{Z}\}$ be isomorphic copies of $M$ and for a finite subset $\Lambda\subset \mathbb{Z}$ we denote
    \[
        M_{\Lambda}:=\overline{\bigotimes_{n\in \Lambda}} M_n.
    \]
Inclusions $\Lambda\subset \Pi$ of finite subsets of $\mathbb{Z}$ induce embeddings $M_\Lambda \subset M_\Pi$ so that one can consider the inductive limit \emph{C*-algebra}
    \[
        \A_{\mathbb{Z}}:= \varinjlim M_\Lambda,    
    \]
which is called the \emph{quasi-local algebra} associated to the spin chain with on-site algebras $M_n=M$ for all $n\in \mathbb{Z}$ (see \cite[Definition 2.6.3 and Example 4.2.12]{BratteliRobinson1}). This algebra admits a canonical translation action $\mathbb{Z}\overset{\alpha}{\curvearrowright}\A_{\mathbb{Z}}$, and a state on $\A_\mathbb{Z}$ is said to be \emph{locally normal} if its restriction to each $M_{\Lambda}$ is normal (see \cite[Definition 2.6.6]{BratteliRobinson1} or \cite{HudsonMoody}). Taking inspiration from the classification of translation invariant states in \cite{FannesNachtergaeleWerner}, we construct a class of random variables taking values in locally normal states which obey a kind of \emph{translation covariance} property relative to the ergodic transformation. As an application of our previous main results, we establish the following clustering estimate for our class of translation covariant states:

\begin{thmx}[{Theorem~\ref{thm:Rclustering}}]\label{thmx:C}
Let $(M,\tau_M)$ and $(W,\tau_W)$ be tracial von Neumann algebras with separable preduals, let $(\Omega, \P)$ be a probability space equipped with ergodic $T\in \Aut(\Omega, \P)$, and let $\mathcal{E}_\omega \colon M\bar\otimes W\to W$ be a normal unital positive random linear operator. Define $\phi_\omega(x):=\mathcal{E}_\omega(1\otimes x)$ and suppose that
     \[
        \P[\exists n \colon c((\phi_{T^n\omega} \circ \cdots \circ \phi_\omega)_*)<1]>0.
    \]
Then $\mathcal{E}_\omega$ determines (see Theorem~\ref{thm:thermo}.(4)) a
random locally normal state $\Psi_\omega$ on the quasi-local algebra $\A_{\mathbb{Z}}$ associated to the spin chain whose on-site algebras are isomorphic to $M$ that satisfies
    \[
            \Psi_\omega \circ \alpha_k = \Psi_{T^k \omega} \qquad\qquad \forall k\in \mathbb{Z}.
    \]
Moreover, there is a constant $\kappa \in (0,1)$ and a random variable $E_\omega\in [0,+\infty)$ such that 
    \[
            |\Psi_\omega(ab) - \Psi_\omega(a) \Psi_\omega(b)| \le E_\omega \kappa^{\dist(\Lambda, \Pi)-1} \|a\|_\infty \|b\|_\infty \qquad \qquad \forall a\in M_{\Lambda},\ b\in M_{\Pi}
    \]
almost surely for finite subsets $\Lambda \subset (-\infty,-1)$ and $\Pi \subset [1,+\infty)$.
\end{thmx} 

In words, we construct a family of \textit{random} locally normal states that exhibit almost sure exponential clustering. Our approach is modelled after \cite[Proposition 2.3 and 2.5]{FannesNachtergaeleWerner} and gives $\Psi_\omega(a)$ as an almost sure operator norm limit of
    \[
        \mathcal{E}_{T^{-N}\omega}\circ(1\otimes \mathcal{E}_{T^{-N+1}\omega})\circ \cdots \circ (1\otimes \cdots \otimes 1\otimes \mathcal{E}_{T^N\omega})(a\otimes 1_W)
    \]
for local observables $a\in M_{[m,n]}\subset M_{[-N,N]}$ (see Theorem~\ref{thm:thermo}.(4)). This result is an important step towards a better understanding of many-body systems subject to both information loss to a reservoir and large on-site disorder. 

\section*{Acknowledgements}
The authors would like to thank Jeffrey Schenker for bringing the question of an infinite dimensional generalization of \cite{jeff} to their attention, and for many productive conversations on the topic. The authors would also like to thank Stefaan Vaes for his helpful suggestions.  EBR would like to thank BN and Jeffrey Schenker for their hospitality during his visits to Michigan State University in the summer of 2022 and spring of 2023, and for their continued mentoring in this area of study. 
EBR would also like to thank his advisor, Robert Sims, for his support and patience while EBR worked on this project, 
a helpful discussion of \cite[Section IV.1]{CarLa}, and for lending EBR his hard copy of \cite{CarLa}
. BN was supported by NSF grants DMS-1856683 and DMS-2247047.

\tableofcontents

\section{Preliminaries}\label{sec:prelims}

\subsection{Finite von Neumann algebras}\label{sec:radonnikodym}
Throughout, $M$ will be a finite von Neumann algebra equipped with a faithful normal tracial state $\tau$, and we will refer to the pair $(M,\tau)$ simply as a \emph{tracial von Neumann algebra}.  We will identify $M$ with its standard representation on $L^2(M,\tau)$; that is, the Gelfand--Naimark--Segal construction associated to $\tau$. For a closed subspace $\mathcal{H}\leq L^2(M,\tau)$ we write $[\mathcal{H}]$ for the projection onto $\mathcal{H}$. We denote by $J$ the conjugate linear map on $L^2(M,\tau)$ determined by $Jx=x^*$, and we recall that $JMJ=M'$ the commutant of $M$ in $B(L^2(M,\tau))$.

One says that a closed, densely-defined operator $x$ on $L^2(M, \tau)$ is \emph{affiliated} to $M$ if and only if for the polar decomposition $x = v|x|$ one has $v,1_{[0,t]}(|x|)\in M$ for all $t\geq 0$. In this case one writes $x\aff M$. 
In particular, given $x\aff M$, one has $x\in M$ if and only if $x$ is bounded. Recall that every operator affiliated to $M$ is \textit{maximally extended} in the sense that if  $x\aff M$ and $x\subset y$, then  $x=y$. The set of affiliated operators $\widetilde{M}$ forms a $*$-algebra under the operations of closing linear combinations and products, with the usual adjoint as the involution (see \cite[Theorem 7.2.8]{ap}). We adopt the notation $\dom(x)$ for the domain of $x$ in $L^2(M,\tau)$.

Equip $M$ with the norm $\|\cdot \|_1:= \tau( |\cdot |)$. The completion of $M$ with respect to $\|\cdot \|_1$ is written $L^1(M,\tau)$. It is nontrivially isometrically isomorphic to the predual $M_*$ (see \cite[Theorem 7.4.5]{ap}). This isomorphism is implemented via $x\mapsto \tau_x(\cdot) := \tau(x\,\cdot\,)$. 
One can also identify $L^1(M,\tau)$ with the set of $x\aff M$ such that $\|x\|_1:=\sup_{t>0} \tau(|x|1_{[0,t]}(|x|)) <\infty$, so we will frequently view elements of $L^1(M,\tau)$ as unbounded operators on $L^2(M,\tau)$. More generally, for $1\leq p <\infty$ one defines $L^p(M,\tau)$ as the set of $x\aff M$ such that $\|x\|_p:= \sup_{t>0} \tau(|x|^p 1_{[0,t]})^{1/p}<\infty$. For $p=2$ one obtains a Hilbert space that is naturally isomorphic to the standard representation of $M$, and so there is no conflict of notation. Moreover, as unbounded operators $L^2(M,\tau)$ corresponds to those $x\aff M$ with $1\in \dom(x)$ (see \cite[Theorem 7.3.2]{ap}). It follows that $M\subset L^2(M,\tau)$ is a core for such $x$. Note that for $x\in L^1(M,\tau)$, $|x|^{1/p}\in L^p(M,\tau)$ for all $1\leq p <\infty$, and so in particular $M\subset L^2(M,\tau)$ is a core for $|x|^{1/2}$.

 Recall that we say $M$ has a \textbf{separable predual} if $L^1(M,\tau)$ is separable as a Banach space. Examples include finite dimensional von Neumann algebras, the hyperfinite $\mathrm{II}_1$ factor, and group von Neumann algebras for countable discrete groups. There are several equivalent formulations of this which will be useful.

\begin{thm}[{\cite[Theorem 1.3.11]{olesen}}]\label{thm:sep}
For a tracial von Neumann algebra $(M,\tau)$, the following are equivalent:
    \begin{enumerate}[label = (\roman*)]
        \item The predual $(L^1(M,\tau), \|\cdot\|_1)$ is a separable Banach space.
        \item $(M,\tau)$ is separable in the $\sigma$-WOT.
        \item $L^2(M,\tau)$ is a separable Hilbert space.
    \end{enumerate}
\end{thm}

The following lemma, which is well-known to experts, will be useful in analyzing Hennion's metric $d$.

\begin{lem}\label{lem:infinite_dimensional}
For a tracial von Neumann algebra $(M,\tau)$, the following are equivalent:
    \begin{enumerate}[label=(\roman*)]
        \item $M$ is infinite dimensional.
        
        \item There exists a family $\{p_n\colon n\in \N\}\subset M$ of non-zero pairwise orthogonal projections satisfying $\sum p_n =1$.
        
        \item There exists a sequence $(p_n)_{n\in \N}\subset M$ of non-zero projections satisfying $\tau(p_n)\to 0$.
    \end{enumerate}
\end{lem}
\begin{proof}
$(i)\Rightarrow (ii)$: Let $\mathcal{Z}(M)$ be the center of $M$, which is isomorphic to $L^\infty(X,\mu)$ for some probability space. If $(X,\mu)$ has any diffuse subsets or infinitely many inequivalent atoms, then we are done. Otherwise, $L^\infty(X,\mu)\cong \C^n$ for some $n\in \N$ and $M$ is a finite direct sum of factors. One of these factors is necessarily infinite dimensional (lest $M$ be finite dimensional) and hence contains such a family because it lacks minimal projections.\\

\noindent$(ii)\Rightarrow(iii)$: Since $\sum \tau(p_n) =1 <\infty$, we have $\tau(p_n)\to 0$.\\

\noindent$(iii)\Rightarrow (i)$: We proceed by contrapositive. If $M$ is finite dimensional, then it is necessarily a multimatrix algebra and $\tau$ is a convex combination of traces. It follows that the traces of projections in $M$ form a finite discrete subset of $[0,1]$, and in particular $0$ is isolated.
\end{proof}

\subsubsection{Positivity}
Recall that $a\in M$ is \emph{positive} if $a=b^*b$ for some $b\in M$; equivalently, if $a$ is positive semidefinite as an operator on $L^2(M,\tau)$. In this case, we write $a\geq 0$ and we denote the \emph{positive cone} of $M$ by $M_+:=\{a\in M\colon a\geq 0\}$. The positive cone induces an ordering on the self-adjoint elements of $M$: for $a,b\in M_{s.a.}$ we write $a\leq b$ if $b-a\geq 0$.

More generally, for unbounded self-adjoint operators $x,y$ on $L^2(M,\tau)$ we write $x\leq y$ when $y-x$ is positive semidefinite on $\dom(x)\cap \dom(y)$. Note that for affiliated operators $x,y\aff M$, $x\leq y$ is equivalent to saying the closure of $y-x$ is positive. In particular, $x\leq y$ and $y\leq x$ imply $x$ and $y$ agree on $\dom(x)\cap \dom(y)$ and therefore $x=y$. For each $1\leq p <\infty$, we denote $L^p(M,\tau)_+:=\{x\in L^p(M,\tau)\colon x\geq 0\}$, which we recall is the $\|\cdot\|_p$-closure of $M_+$.

\begin{rmk}\label{rmk:boundedimpliesbounded}
If $0\leq x\leq y$ and $y$ is bounded, then $x$ is necessarily bounded since
    \[
        \<x\xi,\xi\> \leq \<y\xi,\xi\>\leq \|y\|\|\xi\|^2
    \]
holds for $\xi$ in the dense subspace $\dom(y-x)$.$\hfill\blacksquare$
\end{rmk}

Also recall that we say a densely defined linear operator $x$ on $L^2(M,\tau)$ is \emph{boundedly invertible} if there exists $b\in B(L^2(M,\tau))$ satisfying $xb=1$ and $bx\subset 1$, in which case we write $x^{-1}:=b$. Note that if $x\aff M$ then $x^{-1}\in M$.

\begin{rmk}\label{rmk:boundedlyinvert_implies_boundedlyinvert}
A positive densely defined operator $x$ on $L^2(M,\tau)$ is boundedly invertible if and only if $x\geq \delta $ for some scalar $\delta>0$, and in this case one can choose $\delta=\|x^{-1}\|^{-1}$. In particular, if $x,y\aff M$ satisfy $0\leq x \leq y$, then $x$ being boundedly invertible implies $y$ is boundedly invertible.$\hfill\blacksquare$
\end{rmk}

A linear map $\phi\colon M\to N$ between von Neumann algebras is said to be \emph{positive} if $\phi(M_+)\subset N_+$, in which case we write $\phi\geq 0$. More generally, we say $\phi$ is \emph{$n$-positive} for $n\in \mathbb{N}$ if the map
    \begin{align*}
        \phi\otimes I_n \colon \quad M_n(M) &\to M_n(N)\\
            (a_{ij})_{1\leq i,j\leq n} &\mapsto (\phi(a_{ij}))_{1\leq i,j\leq n},
    \end{align*}
is positive. We say $\phi$ is \emph{completely positive} if it is $n$-positive for all $n\in \mathbb{N}$. Recall the following classical result. 

\begin{thm}[Russo--Dye Theorem]\label{thm:Russo-Dye}
Let $A$ be a unital $C^*$-algebra and $\phi:A\to A$ be a linear mapping. Then 
    \[
        \|\phi\|:= \sup_{x\colon \|x\|=1} \|\phi(x)\| = \sup_{u\in \mathcal{U}(A)} \|\phi(u)\|,
    \]
where $\mathcal{U}(A)$ denotes the unitary group of $A$. In particular, if $\phi$ is positive then $\|\phi\|=\|\phi(1)\|$.
\end{thm}



Under the identification of $M_*$ with $L^1(M,\tau)$, the positive linear maps in $M_*$ correspond to $L^1(M,\tau)_+$. In fact, one has
    \begin{align}\label{eqn:positive_iff}
    \begin{split}
        x\in L^1(M,\tau)_+ \qquad &\Longleftrightarrow  \qquad  \tau(xa)\geq 0 \qquad \forall a\in M_+,\\
        a\in M_+ \qquad &\Longleftrightarrow \qquad \tau(xa) \geq0 \qquad \forall x\in L^1(M,\tau)_+.
    \end{split}
    \end{align}
One says a linear map $\gamma\colon L^1(M,\tau)\to L^1(M,\tau)$ is \emph{positive} if $\gamma(L^1(M,\tau)_+) \subset L^1(M,\tau)_+$, and writes $\gamma\geq 0$. More generally, one can define \emph{$n$-postivitiy} and \emph{complete positivity} for such maps by considering $\gamma\otimes I_n$ defined on $M_n(L^1(M,\tau))\cong L^1( M_n(M), \tau\otimes(\frac1n \text{Tr}))$.

There is a well-known correspondence between normal linear maps $\phi\colon M\to M$ and bounded linear maps $\gamma\colon L^1(M,\tau)\to L^1(M,\tau)$. Indeed, using that $L^1(M,\tau)\cong (M,\text{weak}*)^*$, one has that $\phi_*\colon L^1(M,\tau)\to L^1(M,\tau)$ is a bounded linear map satisfying
    \[
        \tau(\phi_*(x) a) = \tau(x\phi(a)),
    \]
for all $x\in L^1(M,\tau)$ and $a\in M$. From Equation~(\ref{eqn:positive_iff}), it follows that $\phi_*$ is positive if and only if $\phi$ is positive. Similarly for $n$-positivity and complete positivity. Conversely, given a bounded linear map $\gamma$ on $L^1(M,\tau)$, using
$M\cong L^1(M,\tau)^*$ one has that $\gamma^*\colon M\to M$ is a normal linear map satisfying
    \[
        \tau( x \gamma^*(a)) = \tau( \gamma(x) a),
    \]
for all $x\in L^1(M,\tau)$ and $a\in M$. The positivity (resp. $n$-positivity or complete positivity) of $\gamma^*$ again follows from that of $\gamma$ via Equation~(\ref{eqn:positive_iff}). For future reference, we record these observations in Lemma~\ref{lem:normalbddextension} below.

Toward refining the above correspondence, we recall a bit more terminology. We say a positive linear map $\phi\colon M\to M$ is \emph{$\tau$-bounded} if there exists a constant $c>0$ so that $\tau(\phi(a)) \leq c\tau(a)$ for all $a\in M_+$. We say a positive linear map $\gamma\colon L^1(M,\tau)\to L^1(M,\tau)$ is \emph{$M$-preserving} if $\gamma(M)\subset M$. Since each element of $M$ can be decomposed as a linear combination of four positive elements, this is equivalent to $\gamma(M_+)\subset M_+$. Using $a\leq \|a\| 1$ for all $a\in M_+$, this is further equivalent to $\gamma(1)\in M_+$ by Remark~\ref{rmk:boundedimpliesbounded}. Observe that if $\phi\colon M\to M$ is normal, positive, and $\tau$-bounded with constant $c>0$, then
    \[
        \tau((c-\phi_*(1)) a) = c \tau(a) - \tau(\phi(a)) \geq 0
    \]
for all $a\in M_+$ so that $\phi_*(1)\leq c$ by Equation~(\ref{eqn:positive_iff}). Thus $\phi_*$ is $M$-preserving. Conversely, if $\gamma\colon L^1(M,\tau)\to L^1(M,\tau)$ is $M$-preserving, bounded, and positive, then for $a\in M_+$ one has
    \[
        \tau(\gamma^*(a)) = \tau(\gamma(1) a) \leq \|\gamma(1)\| \tau(a) = \|\gamma\| \tau(a).
    \]
Thus $\gamma^*$ is $\tau$-bounded with constant $\|\gamma\|$. Thus the correspondence from before restricts to a correspondence between $\tau$-bounded normal positive linear maps on $M$ and $M$-preserving bounded positive linear maps on $L^1(M,\tau)$. We also record this in Lemma~\ref{lem:normalbddextension} below.

Finally, we note that for $\tau$-bounded normal positive linear maps $\phi$ on $M$ (resp. $M$-preserving bounded positive linear maps $\gamma$ on $L^1(M,\tau)$) there is another correspondence given by extending (resp. restricting) the maps. Indeed, for $a\in M$ let $\phi(a)=v|\phi(a)|$ be the polar decomposition. Then one has
    \[
        \| \phi(a)\|_1 = |\tau( v^* \phi(a))| = |\tau( \phi_*(v^*) a)| \leq \| \phi_*(v^*)\| \|a\|_1,
    \]
where we have used that $\phi_*$ is $M$-preserving. It follows that $\phi$ admits a unique bounded linear extension $\phi|^{L^1(M,\tau)}$ to $L^1(M,\tau)$. Since $L^1(M,\tau)_+ = \overline{M_+}^{\|\cdot\|_1}$ one further has that the extension is positive, and it is $M$-preserving since $\phi(M)\subset M$. Conversely, if $\gamma$ is an $M$-preserving bounded positive  linear map on $L^1(M,\tau)$, then $\gamma|_M$ defines a positive linear map on $M$. For $a\in M_+$, one has
    \[
        \tau(\gamma(a)) = \| \gamma(a)\|_1 \leq \|\gamma\| \|a\|_1 = \|\gamma\| \tau(a)
    \]
so that $\gamma|_M$ is $\tau$-bounded and normal (see \cite[Proposition 2.5.11]{ap}). This is also recorded in the following lemma as well as an interaction between the above two correspondences (whose proof is left to the reader).

\begin{lem}\label{lem:normalbddextension}
Let $(M,\tau)$ be a tracial von Neumann algebra. There is a one-to-one correspondence between normal positive (resp. completely positive) linear maps $\phi$ on $M$ and bounded positive (resp. completely positive) linear maps $\phi_*$ on $L^1(M,\tau)$ determined by
    \[
        \tau(\phi_*(x) a)=\tau(x\phi(a))  \qquad\qquad x\in L^1(M,\tau),\ a\in M.
    \]
This correspondence restricts to a one-to-one correspondence between $\tau$-bounded normal positive (resp. completely positive) linear maps on $M$ and $M$-preserving bounded positive (resp. completely positive) linear maps on $L^1(M,\tau)$. The former maps $\phi$ also admit unique extensions $\phi|^{L^1(M,\tau)}$ to $L^1(M,\tau)$ that are $M$-preserving bounded and positive (resp. completely positive), and the latter maps $\phi_*$ also have restrictions $\phi_*|_M$ to $M$ that are $\tau$-bounded normal and positive (resp. completely positive). In this case, one has $(\phi_*|_M)_*= \phi|^{L^1(M,\tau)}$.
\end{lem}

 As with the operators themselves, for linear maps $\phi,\psi\colon M\to M$ we write $\phi\leq \psi$ if $\psi- \phi\geq 0$. Similarly, for linear maps $\gamma,\rho\colon L^1(M,\tau)\to L^1(M,\tau)$ we write $\gamma\leq \rho$ if $\rho - \gamma\geq 0$.

\begin{lem}\label{lem:adjoint&predual}
Let $\alpha,\beta \colon M\to M$ be positive normal linear maps. Then $\alpha \leq \beta$ if and only if $\alpha_*\leq \beta_*$.
\end{lem}
\begin{proof}
For $x\in L^1(M,\tau)_+$ and $a\in M_+$ we have
    \[
        \tau(x (\beta - \alpha)(a)) = \tau((\beta_* - \alpha_*)(x) a).
    \]
So $\alpha(a)\leq \beta(a)$ for all $a\in M_+$ if and only if the above quantity is non-negative for all $x\in L^1(M,\tau)_+$ and $a\in M_+$, which is in turn equivalent $\alpha_*(x) \leq \beta_*(x)$ for all $x\in L^1(M,\tau)_+$.
\end{proof}

\subsection{Probability theory}

Throughout $(\Omega,\mathcal{F}, \mathbb{P})$ will denote a probability space. An \emph{automorphism} of $(\Omega,\mathcal{F}, \mathbb{P})$ is a bijection $T\colon \Omega\to \Omega$ such that $T$ and $T^{-1}$ are measurable and measure preserving. Denote the automorphisms of $(\Omega,\mathcal{F}, \mathbb{P})$ by $\text{Aut}(\Omega, \mathbb{P})$. One says $T\in \text{Aut}(\Omega,\P)$ is \emph{ergodic} if whenever $A\in\mathcal{F}$ satisfies $T^{-1}(A)\subset A$ then $\mathbb{P}[A]\in \{0,1\}$. Note that in this case $T^{-1}$ is also necessarily ergodic. Indeed, if $T(A)\subset A$, then for
    \[
        B:= \bigcup_{n=0}^\infty T^{-n}(A),
    \]
we have $T^{-1}(B)\subset B$ and $\displaystyle \P[B]=\lim_{n\to \infty} \P[T^{-n}(A)]$ by continuity from below. Thus if $T$ is measure preserving and ergodic, it follows that $\P[A]=\P[B]\in \{0,1\}$.

\begin{rmk}
Any $T\in \text{Aut}(\Omega, \mathbb{P})$ also induces an automorphism of the von Neumann algebra $L^\infty(\Omega,\mathbb{P})$ via precomposition: $f\mapsto f\circ T$. In fact, one can induce such an automorphism using bijections of the form $T\colon \Omega\setminus N_1\to \Omega\setminus N_2$ where $N_1,N_2$ are null sets and $T$ and $T^{-1}$ are measurable and measure preserving. Also, if $T_1, T_2$ are two such bijections which agree almost surely then they induce the same automorphism. Thus after identifying maps modulo null sets, these bijections form a group $\text{Aut}_0(\Omega,\mathbb{P})$ that embeds as a subgroup of $\text{Aut}(L^\infty(\Omega,\mathbb{P}))$. This embedding is a surjection when $(\Omega, \mathcal{F},\mathbb{P})$ is a \emph{standard probability space} (see \cite[Section 3.3]{ap}); that is, if there exists a bijection $S\colon \Omega\setminus N_1\to \Omega_0\setminus N_2$ where $(\Omega_0,\mathcal{F}_0,\mathbb{P}_0)$ is the disjiont uniont of the Lebesgue measure on $[0,1]$ and countably many atoms, $N_1\subset \Omega$ and $N_2\subset \Omega_0$ are null sets, and $S$ and $S^{-1}$ are measurable and measure preserving.$\hfill\blacksquare$
\end{rmk}

\subsubsection{Kingman's ergodic theorem}

We will be interested in multiplicative stochastic processes and their ergodic properties. So we present the following:

\begin{lem}\label{lem:submultiplicativeStoppingtime}
Let $(\Omega, \P)$ be a probability space equipped with an ergodic automorphism $T$ and suppose that $X_n:\Omega\to [0,1]$ is a submultiplicative stochastic process in the sense that \[X_{n+m}\le X_nX_m\circ T^n.\] Let $\mathcal{F}_n := \sigma(X_0, \dots, X_n)$ for $n\ge 0$, denote the natural filtration taken with respect to $(X_n)_{n\ge0}$.  Suppose that $\P[\exists n_*: X_{n_*}<1]>0$. Then $\nu(X):= \inf\{n: X_n <1\}$ is a finite almost-surely stopping time with respect to $(\mathcal{F}_n)_{n\ge 0}$. 
\end{lem}
\begin{proof}
    Since $X_n$ is a decreasing sequence, observe that \[\{\omega: \nu(\omega)\le n\}=\{\omega:X_n <1\} \in \mathcal{F}_n.\] Ergodicity of $T$ informs us that $\P[ \bigcup_{k=0}^\infty \{X_{n_*}\circ T^k <1\}]=1$ and there is a finite almost surely random variable $K(\omega)$ so that $X_{n_*+K}\le X_K X_{n_*}\circ T^{K} <1$. Whence $\nu \le n_* +K$ is finite almost surely. 
\end{proof}

Furthermore, recall the Kingman ergodic theorem: \begin{thm}[Kingman's Ergodic Theorem]\label{thm:kingman}
    Let $(\Omega, \P)$ be a probability space equipped with ergodic $T\in \Aut(\Omega,\P)$. Assume that a stochastic process $\{X_n:\Omega \to \mathbb{R} \}_{n=1}^\infty$ satisfies: 
    \begin{enumerate}[label=(\roman*)]
        \item $X_{n+m}(\omega) \le X_n(\omega) + X_{m}(T^n\omega)$ almost surely;
        \item the positive part $\E[(X_n)^+] <\infty$ for all $n$. 
    \end{enumerate}
Then, $\displaystyle Z(\omega) := \lim_{n\to \infty} n^{-1} X_n(\omega)\in [-\infty, \infty)$ exists almost surely, and $\displaystyle Z = \inf_{n\in \N} n^{-1} \E [X_n]$; that is, $Z$ is constant almost surely.
\end{thm} This is a corollary to \cite[Theorem 2]{king}. See \cite[Theorem 1]{AvilaBochi} for a concise proof. One can find alternative formulations in \cite{CarLa,steele,lalley} among others. The observation underpinning our Lemma~\ref{lem:randomconstproperties}, is that one may apply Theorem~\ref{thm:kingman} to the sub-additive process $(\log X_n)_{n\ge0}$, when $(X_n)_{n\ge0}$ is sub-multiplicative.

\subsubsection{Random linear operators}\label{sec:RLO}
We recall some general notions about the theory of random variables in a Banach space and random linear operators. The material in this section is largely based upon the treatment given in \cite[Chapters 1 and 2]{Bharucha-Reid}, although an equally well-written treatment in the case $\X$ is a Hilbert space may be found in \cite{Skorohod}.

Let $\X$ be a Banach space with norm $\|\cdot\|$. By $\mathfrak{B}$ we mean the $\sigma$-algebra generated by all the closed subsets of $\X$. Let $(\Omega, \mathcal{F}, \mu)$ be a measure space. Given a finite collection of sets $A_1, A_2, \dots, A_n\in \mathcal{F}$ and $b_1, b_2, \dots, b_n \in \X$, a function of the form 
    \[
        \phi(\omega) = \sum_{j=1}^n b_j\, \chi_{A_j}(\omega) 
    \] will be called \emph{simple function}. A \emph{strongly measurable function} is one $f:\Omega \to \X$ for which there is a sequence of simple functions $\phi_n$ so that $\displaystyle \lim_{n\to \infty} \|f(\omega)-\phi_n(\omega)\|= 0$ almost surely (see \cite[Chapter II]{DiestelUhl}). Writing $\X^*$ for the Banach dual space of $\X$, we say that $g:\Omega \to \X$ is a \emph{weakly measurable function} if for every $x^*\in \X^*$, one has that $x^*\circ g (\omega)$ is a $\C$-valued Borel-measurable function. 
    
    The following is originally due to Pettis \cite{pettis} but can be found in \cite[Proposition IV.7.2]{TakesakiI} or \cite[Theorem 2]{DiestelUhl}: 

\begin{thm}[{\cite[Corollary 1.11]{pettis}}]\label{thm:Pettis}
Let $\X$ be a separable Banach space with Borel $\sigma$-algebra $\mathfrak{B}$, then a function $f:\Omega \to \X$ is strongly measurable if and only if it is weakly measurable. 
\end{thm}

In the case that $\mu=\P$ is a probability measure, we shall refer to a strongly measurable function in $\X$ as a \emph{random variable in $\X$}. A mapping $L:\Omega \times \X \to \X$ is an (everywhere-defined) \emph{random linear operator} if: $\omega\mapsto L_{\omega}(x)$ is a random variable in $\X$ for all $x\in \X$; and $\P[L_\omega (\alpha x + \beta y) = \alpha L_\omega x + \beta L_\omega y]=1$ for all $x,y\in \X$ and $\alpha, \beta \in \C$. Recall from \cite[Definition 2.24]{Bharucha-Reid} that a random linear operator $L$ is bounded if there is a random variable $M(\omega)$ so that $\P[ M(\omega)<\infty , \|L_\omega x\|\le M(\omega)\|x\| \text{ all }x\in \X]=1$.

In Sections~\ref{sec:EQP} and~\ref{sec:FCS} we will be concerned with dynamics arising from iterations of bounded random linear operators. It is not obvious that the composition of bounded random linear operators is, however, measurable. We follow the elegant proof given in Bharucha-Reid's \cite[Theorem 2.14]{Bharucha-Reid} below: 

\begin{thm}[{\cite[Theorem 2.14]{Bharucha-Reid}}]\label{thm:RLOcomp}
Let $\X$ be a separable Banach space and $x_\omega$ a random variable in $\X$. If $L$ is a bounded random linear operator then the function 
\[
    y_\omega := L_\omega x_\omega
\] is a random variable in $\X$. In particular, the composition of bounded random linear operators is a bounded random linear operator.  
\end{thm}
\begin{proof}
Let $(\phi_n)_{n\in \mathbb{N}}$ be a sequence of simple functions approximating $x_\omega$. By reducing to a subsequence if necessary, we assume $\|\phi_n(\omega) - x_\omega\|\to 0$ almost surely. Let $E:=\{\phi_n(\omega)\colon n\in \mathbb{N},\ \omega\in \Omega\}$, which we note is a countable subset of $\X$.
For each $n\in \mathbb{N}$, define $y_n(\omega) := L_\omega \phi_n(\omega)$. Then, for any Borel subset $S\in \mathfrak{B}$, we obtain the decomposition 
    
    \[
        [y_n\in S]= \bigcup_{b\in E} [L(b) \in S]\cap [\phi_n(\omega)=b],
    \] thus demonstrating that $y_n$ is a sequence of random variables in $\X$. Using that $L_\omega$ is almost surely bounded, we get that the following limit exists almost surely:
    \[
        y_\omega = \lim_{n\to \infty} y_n(\omega) = \lim_{n\to \infty} L_\omega \phi_n(\omega) = L_\omega x_\omega.
    \] That composition of random linear operators is a random linear operator is now immediate. 
    \end{proof}

\section{Metric Geometry of the Normal State Space.}\label{sec:dtop}
Let $(M,\tau)$ be a tracial von Neumann algebra. We  write
    \[
        S:=\{x\in L^1(M)_+\colon \tau(x)=1\}.
    \]
Note that $S$ corresponds to the normal states on $M$. Following \cite{jeff} we introduce a non-standard metric $d$ on the set of normal states $S\subset L^1(M,\tau)_+$ and investigate its properties. We shall show that $d$ admits the same formula (Lemma~\ref{lem:dformula}) as was shown in \cite[Lemma 3.6]{jeff}, and use this to show that $(S,d)$ is complete and is finer than $S$ with the trace-metric (Theorem~\ref{thm:tracemetricinequality_and_completeness}). Unlike \cite{jeff}, whenever $M$ is infinite dimensional $S$ admits two additional disconnected components corresponding to affiliated operators that are of separate interest and we investigate their properties in Theorem~\ref{thm:geometry}.
\subsection{Hennion's Metric}
\begin{lem}\label{lem:mismax}
For $x,y\in L^1(M,\tau)_+\setminus\{0\}$ one has $\{\lambda\in \mathbb{R} \colon \lambda y \leq x\}=(-\infty,\lambda_0]$ for some $0\leq \lambda_0\leq \frac{\tau(x)}{\tau(y)}$.
\end{lem}
\begin{proof}
Certainly $0\in \Lambda:= \{ \lambda \in \mathbb{R}: \lambda y \le x\}$, and if $\lambda\in \Lambda$ then $(-\infty, \lambda]\subset \Lambda$ since $t y\leq \lambda y\leq x$ for all $t\leq \lambda$. Note that if $\lambda y\le x$ then by applying $\tau$ we obtain $\lambda \le \tau(x)/ \tau(y)$. In particular, $\Lambda$ is bounded and we can find an increasing sequence $(\lambda_n)_{\in \N}\subset \Lambda$ converging to $\lambda_0:=\sup{\Lambda}\leq \tau(x)/\tau(y)$. Then for all $a\in M_+$ we have
    \[
        \tau(a(x-\lambda_0 y)) = \lim_{n\to\infty} \tau(a(x-\lambda_n y)) \geq 0.
    \]
Hence $x-\lambda_0 y\geq 0$ and $\lambda_0\in \Lambda$ so that $\Lambda=(-\infty,\lambda_0]$.
\end{proof}

\begin{define}\label{def:m}
Let $x,y\in L^1(M,\tau)_+$. We define their \textbf{m-quantity} to be the number
    \[
        m(x,y) := \max\{ \lambda\in \mathbb R : \lambda y \le x\}.\tag*{$\blacksquare$}
   \] 
\end{define}

\begin{thm}[Properties of $m(x,y)$]\label{thm:mproperties}
Let $x,y,z\in L^1(M,\tau)_+\setminus\{0\}$.
    \begin{enumerate}[label=(\arabic*)]
        \item\label{part:mrange} $0\le m(x,y) \le \frac{\tau(x)}{\tau(y)}$.
         \item\label{part:mscaling} $m(a x, b y) = \frac{a}{b} m(x,y)$ for scalars $a,b>0$.
        \item\label{part:mtriangle_ineq} $m(x,z)m(z,y) \le m(x,y)$.
        \item $m(x,y)m(y,x) = 1$ if and only if $\frac{1}{\tau(x)} x=\frac{1}{\tau(y)}y$. 
        \item If $x\leq y$ then one has
            \begin{align*}
                m(x,z) &\leq m(y,z) \\ m(z,x)&\geq m(z,y).
            \end{align*}
        \item\label{part:mstrictlypositive} $m(x,y)>0$ if and only if there exists non-zero $a\in M$ satisfying $y=x^{\frac12}a^*ax^{\frac12}$ with $a=a[x^{\frac12} M]$, in which case $m(x,y) = \|a\|^{-2}$.
        \item\label{part:mzeroapprox} $m(x,y)=0$ if and only if there exists a sequence of vectors $(\xi_n)_{n\in \N}\subset \dom(x)\cap \dom(y)$ so that $\|x^{1/2}\xi_n\|_2 \to 0$ and $\|y^{1/2}\xi_n\|_2\equiv 1$.
        \item\label{part:m_inf_formula} 
            \begin{align*}
                m(x,y) 
                    =\inf \left\{ \frac{\tau(xa)}{\tau(ya)}\colon a\in M_+,\ \tau(ya)> 0 \right\}.
            \end{align*}
    \end{enumerate}
\end{thm}
\begin{proof}
\begin{enumerate}[label = \textbf{(\arabic*):}]
\item This follows from Lemma~\ref{lem:mismax}.

\item The inequality $m(ax,by) by \leq ax$ immediately implies $m(ax,by)\frac{b}{a}\leq m(x,y)$. One the other hand, $m(x,y)y\leq x$ is equivalent to $m(x,y)\frac{a}{b} by \leq ax$, which gives $m(x,y)\frac{a}{b}\leq m(ax,by)$.

\item We have
    \[
        m(x,z)m(z,y) y \leq m(x,z) z\leq x,
    \]
whence $m(x,z)m(z,y)\leq m(x,y)$.

\item If $\frac{1}{\tau(x)}x=\frac{1}{\tau(y)}y$, then $m(x,y)$ and $m(y,x)$ achieve their maximum values of $\frac{\tau(x)}{\tau(y)}$ and $\frac{\tau(y)}{\tau(x)}$, respectively, and hence their product gives one. On the other hand, if $m(x,y)m(y,x)=1$ then by part~\ref{part:mrange} one necessarily has $m(x,y) =\frac{\tau(x)}{\tau(y)}$ and $m(y,x)=\frac{\tau(y)}{\tau(x)}$. Therefore $\frac{\tau(x)}{\tau(y)}y\leq x$ and $\frac{\tau(y)}{\tau(x)} x\le y$ so that $\frac{1}{\tau(x)} x= \frac{1}{\tau(y)} y$.

\item Suppose $x\leq y$. Then $m(x,z) z\leq x \leq y$ so that $m(x,z)\leq m(y,z)$. Similarly, $z\geq m(z,y) y \geq m(z,y)x$ so that $m(z,x) \geq m(z,y)$.

\item First suppose $y=x^{1/2}a^*ax^{1/2}$ for some non-zero $a\in M$. Then $y\leq \|a\|^2 x$, and so $m(x,y)\geq \|a\|^{-2}>0$. 
Conversely, if $m(x,y)>0$ then for any $b\in M$ we have
    \[
        \|b y^{1/2}\|_2^2 = \tau(b y b^*) \leq \frac{1}{m(x,y)} \tau(b x b^*) = \frac{1}{m(x,y)} \| b x^{1/2}\|_2^2.
    \]
Hence $bx^{1/2}\mapsto by^{1/2}$ extends to a bounded operator $T\colon \overline{M x^{1/2}} \to \overline{ M y^{1/2}}$ with $\|T\|\leq m(x,y)^{-1/2}$. Observe that $T$ is non-zero by virtue of $x^{1/2}$ and $y^{1/2}$ being non-zero. Let $p,q\in M$ be the support projections of $\tau_x$ and $\tau_y$, respectively, which we note satisfy $JpJ = [ M x^{1/2}]$ and $JqJ = [M y^{1/2}]$. Trivially extend $T$ to $L^2(M)$ so that $JqJ T JpJ=T$, and observe that for $b,c,d\in M$ we have
    \[
        \< T b(cx^{1/2}), d y^{1/2}\>_2 = \< bc y^{1/2}, dy^{1/2}\>_2 = \< c y^{1/2}, b^* y^{1/2}\>_2 = \< T cx^{1/2}, b^* d y^{1/2}\>_2 = \< bT cx^{1/2}, dy^{1/2}\>_2.
    \]
It follows that
    \[
        0=JqJ (Tb - bT) JpJ = JqJ T JpJ b - b JqJ T JpJ = Tb - bT.
    \]
Hence $T\in M'$ and so is of the form $T=JaJ$ for some $a\in M$. Note that $a$ is non-zero since $T$ is non-zero. Additionally, we have for all $b\in M$
    \[
        \tau(yb)=\< b y^{1/2}, y^{1/2}\>_2 =\< b Tx^{1/2}, T x^{1/2}\>_2= \< b x^{1/2}a^*, x^{1/2}a^*\>_2 = \tau(x^{1/2}a^*a x^{1/2} b).
    \]
Since $y$ is determined by $\tau_y\in M_*$, it follows that $y=x^{\frac12}a^*ax^{\frac12}$. Note  that $ap=J(TJpJ)J = JTJ=a$ and $p=J[Mx^{1/2}]J= [x^{1/2} M]$.

We saw above that $m(x,y)\geq \|a\|^{-2}$ and $\|T\| \leq m(x,y)^{-1/2}$. Since $\|T\|=\|a\|$, this gives $m(x,y)=\|a\|^{-2}$.

\item Assume $m(x,y)=0$ so that the closed operator $x-\lambda y$ is not positive for any $\lambda>0$. Since $x-\lambda y$ is self-adjoint, this implies there exists some $\xi\in \dom(x-\lambda y)$ such that
    \[
        \<(x-\lambda y)\xi, \xi\>_2 <0.
    \]
Because $\dom(x)\cap\dom(y)$ is a core for all $x-\lambda y$, we can in fact find $\xi\in \dom(x)\cap \dom(y)$ satisfying the above, and by scaling we can further assume $\|y^{1/2}\xi\|_2=1$. For each $n\in \N$, let $\xi_n$ be the vector obtained in this way for $\lambda = \frac1n$. Then
    \[
        \|x^{1/2}\xi_n\|_2^2 = \<x\xi_n, \xi_n\>_2 < \frac{1}{n} \<y\xi_n, \xi_n\>_2 = \frac1n \|y^{1/2} \xi_n \|_2^2 = \frac1n \to 0,
    \]
as desired.

Conversely, if $m(x,y)>0$ then $y=x^{\frac12} a^* a x^{\frac12}$ for some $a\in M$ by part~\ref{part:mstrictlypositive}. So for any $\xi \in \dom(x)\cap \dom(y)$ we have
    \[
        \| y^{1/2}\xi\|_2^2 = \< y\xi, \xi\>_2 = \<x^{\frac12}a^* a x^{\frac12} \xi, \xi\>_2 = \|ax^{1/2}\xi\|^2 \leq \|a\|^2 \|x^{1/2}\xi\|^2.
    \]
Consequently, for $(\xi_n)_{n\in \N}\subset \dom(x)\cap \dom(y)$ the condition $\|x^{1/2}\xi_n\|_2\to 0$ precludes $\|y^{1/2}\xi_n\|_2\equiv 1$.

\item If $m(x,y)=0$ then the equality follows from part~\ref{part:mzeroapprox}. Indeed, for $\epsilon>0$ let $\xi\in \dom(x)\cap \dom(y)$ be such that $\|x^{1/2}\xi\|_2<\frac{\epsilon}{2}$ and $\|y^{1/2}\xi\|_2=1$. Since $M\subset L^2(M,\tau)$ is a core for $x^{1/2}$ and $y^{1/2}$, we can find $b\in M$ satisfying $\|x^{1/2}b\|_2 < \epsilon$ and $\|y^{1/2} b\|_2 > 1- \epsilon$. Letting $a=bb^*\in M_+$, we have $\tau(ya)=\|y^{1/2}b\|_2^2 > 1-\epsilon>0$ and
    \[
        \frac{\tau(xa)}{\tau(ya)} = \frac{\|x^{1/2}b\|_2^2}{\|y^{1/2}b\|_2^2} < \frac{\epsilon^2}{(1-\epsilon)^2}.
    \]
Thus the infinimum is also zero.

Now suppose, $m(x,y)>0$, then $y=x^{\frac12}c^*c x^{\frac12}$ for some non-zero $c\in M$ with $c=c[x^{\frac12}M]$ by (4). Moreover,
    \[
        m(x,y) = \frac{1}{\|c\|^2} = \frac{1}{\sup\{\|c\xi\|_2/\|\xi\|_2\colon \xi\in L^2(M)\setminus\{0\} \}^2} = \inf\left\{ \frac{\|\xi\|_2^2}{\|c\xi\|_2^2} \colon \xi\in L^2(M),\ \|c\xi\|_2> 0\right\}.
    \]
The condition $c=c[x^{\frac12}M]$ implies that in the infimum one can restrict to $\xi = x^{1/2}\eta$ for $\eta\in\dom(x)$ with $\|cx^{1/2}\eta\|_2> 0$. Since $M\subset L^2(M)$ is a core for $x^{1/2}$, we can further restrict to $\xi= x^{1/2} b$ for $b\in M$ with $\|cx^{1/2}b\|_2>0$. Using $\|x^{1/2}b\|_2^2 = \tau(xbb^*)$ and $\|cx^{1/2}b\|_2^2 = \|y^{1/2}b\|_2^2 = \tau(ybb^*)$, we obtain the claimed equality.\qedhere
\end{enumerate}
\end{proof}

\begin{define}\label{def:mmetric}
For $x,y\in S$, let 
    \begin{equation*}
        d(x,y) = \frac{1 - m(x,y)m(y,x)}{1+m(x,y)m(y,x)}.
    \end{equation*}
\end{define}

\begin{thm}\label{thm:mmetricproperties}
The function $d$ forms a metric on $S$ such that $\diam_d(S) =1$ when $M\neq \C$.
\end{thm}
\begin{proof}
The proof that $d$ is a metric is the same as in \cite{jeff}. Note that $d$ is valued in $[0,1]$ since $m$ is, and so $\diam_d(S)\le 1$. Also for $M\neq \C$ there exists a non-trivial projection $p\in M$, and one has $\lambda 1 \not\leq \frac{1}{\tau(p)}p$ for all $\lambda >0$. Hence $m(\frac{1}{\tau(p)} p, 1)=0$ and so $d(\frac{1}{\tau(p)} p, 1)=1$.
\end{proof}

The following lemma provides some alternate formulas for the Hennion metric $d$ that are more geometric in nature.

\begin{lem}\label{lem:dformula}
Let $x,y\in S$ be distinct. Then
    \begin{align}\label{eqn:dformula}
        d(x,y) = \frac{t_+-t_-}{t_-+t_+-2t_-t_+} ,
    \end{align}
where 
    \begin{align*}
            t_+&:= \sup\{t\in\mathbb{R}: tx+(1-t)y \in S\} \in \left[1,\frac{2}{\|x-y\|_1}\right],\\
            t_- &:= \inf \{t\in \mathbb{R}: tx+(1-t)y \in S\} \in \left[\frac{-2}{\|x-y\|_1},0\right].
    \end{align*}
Equivalently, if $A_{\pm}=t_{\pm}x + (1-t_{\pm})y$ are the extreme points of the convex set $\{tx+(1-t)y\colon t\in \mathbb{R}\}\cap S$ then
    \begin{align}\label{eqn:eqn:conversedformula}
        d(x,y) = \left|\frac{r(1-s) - (1-r)s}{r(1-s) + (1-r)s}\right| = \frac{|r-s|}{r+s - 2rs},
    \end{align}
where $r,s\in [0,1]$ are determined by $x=rA_- + (1-r)A_+$ and $y=sA_- + (1-s) A_+$.
\end{lem}
\begin{proof}
Since $S$ is convex we immediately have $t_+\geq 1$ and $t_-\leq 0$. To see their other bounds, let $x-y=v|x-y|$ be the polar decomposition. Since $v^*\in (M)_1$, if $tx+(1-ty)\in S$ then we have
    \begin{align*}
        |t|\|x-y\|_1 &= | \tau(tv^*(x-y))| = |\tau(v^*[tx+(1-t)y]) - \tau(v^*y)| \leq 2.
    \end{align*}
Hence $t_+\leq 2/\|x-y\|_1$ and $t_- \geq -2/\|x-y\|_1$.

Consider $A_+:= t_+ x + (1-t_+)y$ and $A_-:= t_-x + (1-t_-) y$, which both belong to $S$ since it is $\|\cdot\|_1$-closed. Additionally, $A_{\pm}$ are the extreme points of the $\|\cdot\|_1$-compact convex set $\{tx + (1-t)y\colon t_-\leq t \leq t_+\}$, and so $x=rA_- + (1-r) A_+$ and $y=sA_- + (1-s)A_+$ for some $r,s\in [0,1]$. In fact, one can explicitly solve a linear system as in \cite{jeff} to show that
    \begin{align*}
        r&= \frac{t_+-1}{t_+-t_-}   & 1-r&= \frac{1-t_-}{t_+-t_-} \\
        s&= \frac{t_+}{t_+ - t_-}   & 1-s&=\frac{-t_-}{t_+-t_-}.
    \end{align*}
We claim  that $m(A_+,A_-)=0$ and $m(A_-, A_+)=0$. Indeed, if there existed $\lambda >0$ so that $\lambda A_- \le A_+$ then we would further have $\lambda A_- \leq A_+ +\lambda x$. Hence
    \[
        0\leq A_+ + \lambda(x- A_-)= (t_+ + \lambda[1- t_-])x + (1- (t_+ + \lambda[1- t_-]))y,
    \]
and $t_+ +\lambda[1-t_-] > t_+$ contradicts the supremacy  of $t_+$. Thus we must have $m(A_+,A_-)=0$, and a similar argument using $t_-$ shows $m(A_-,A_+)=0$.

Now, $m(A_+,A_-)=0$ implies  by Theorem~\ref{thm:mproperties}.\ref{part:mzeroapprox} that there is $(\xi_n)_{n\in \N}\subset \dom(x) \cap \dom(y)$ so that $\|A^{1/2}_-\xi_n\|_2 \equiv 1$ while $\|A^{1/2}_+\xi_n\|_2 \to 0$ as $n\to \infty$. Since
    \begin{align}\label{eqn:A_plus_minus_inequality}
        m(x,y)(sA_- + (1-s) A_+)= m(x,y)y \leq x = r A_- + (1-r) A_+,
    \end{align}
for all $\epsilon >0$ there is $N\in \N$ so that for all $n\ge N$ one has 
    \[
        m(x,y) s = m(x,y) \<s A_- \xi_n, \xi_n\> \leq \<rA_-\xi_n,\xi_n\> + \epsilon = r + \epsilon.
   \]
Letting $\epsilon\to 0$ we see that $m(x,y)\leq \frac{r}{s}$. A similar argument using $m(A_-,A_+)=0$ yields
    \[
        m(x,y) \le \min\left\{ \frac{r}{s}, \frac{1-r}{1-s} \right\}.
    \]
In fact, the above is an equality: using $y=sA_-+(1-s)A_+$ one has
    \[
        \min\left\{\frac{r}{s}, \frac{1-r}{1-s}\right\} y \leq r A_- + (1-r) A_+ = x.
    \]
Reversing the roles of $x$ and $y$ and using the resulting version of (\ref{eqn:A_plus_minus_inequality}) gives
    \[
        m(y,x) = \min\left\{ \frac{s}{r}, \frac{1-s}{1-r} \right\}.
    \]
Thus one has
    \[
        0\le m(x,y)m(y,x) =\min\left\{ \frac{r}{s}, \frac{1-r}{1-s} \right\}\min\left\{ \frac{s}{r}, \frac{1-s}{1-r} \right\} = \min \left\{ \frac{r(1-s)}{s (1-r)}, \frac{(1-r) s}{(1-s) r} \right\}.
   \] 
One then explicitly calculates
    \begin{align*}
        d(x,y) = \frac{1-m(x,y)m(y,x)}{1+m(x,y)m(y,x)} = \left| \frac{r(1-s) - (1-r)s}{r(1-s) + (1-r)s} \right| = \frac{t_+-t_-}{t_-+t_+-2t_+t_-},
    \end{align*}
as claimed.
\end{proof}

\begin{thm}\label{thm:tracemetricinequality_and_completeness}
One has for all $x,y\in S$
    \begin{equation}\label{eqn:dinequality}
        \frac{1}{2}\|x-y\|_1 \le d(x,y).
    \end{equation}
Furthermore, $(S,d)$ is a complete metric space.
\end{thm}
\begin{proof}
Let $A_{\pm}\in S$ and $r,s\in[0,1]$ be as in Lemma~\ref{lem:dformula}. Note that
    \[
        r-s = (1-s) - (1-r) = r(1-s) - (1-r)s,
    \]
and $r(1-s) + (1-r)s\leq 1$. Using these observations we have
    \[
        \|x-y\|_1 = \|(r - s)A_- + ((1-r) - (1-s))A_+\|_1 = |r-s| \|A_- - A_+\|_1 \leq 2 |r-s| \leq 2 \left|\frac{r(1-s) - (1-r)s}{r(1-s) + (1-r) s}\right|,
    \]
which equals $2d(x,y)$ by Equation~(\ref{eqn:eqn:conversedformula}).

To see that $(S,d)$ is complete, let $(x_n)_{n\in \N}\subset S$ be a Cauchy sequence with respect to $d$. The first part of the proof implies $(x_n)_{n\in \N}$ is also Cauchy with respect to $\|\cdot\|_1$ and hence converges to some $x\in S$ with respect to $\|\cdot\|_1$. In particular, one has $\tau(x_na)\to \tau(xa)$ for all $a\in M_+$. Given $\epsilon>0$, set $\eta:=(1-\epsilon)/(1+\epsilon)$. Let $N\in \N$ be such that for $m,n\geq N$ one has $d(x_m,x_n)<(1-\eta^{1/4})/(1+\eta^{1/4})$, which implies
    \[
        \min\{m(x_m,x_n),m(x_n,x_m)\} \geq m(x_m,x_n)m(x_n,x_m)>\eta^{1/4}.
    \]
Fix $n\geq N$. For $a\in M_+$ with $\tau(xa)>0$, let $m\geq N$ be large enough so that $|\tau(x_m a) - \tau(xa)| < (1-\eta^{1/4})\tau(xa)$ and $\tau(x_m a)\neq 0$. Then using Theorem~\ref{thm:mproperties}.\ref{part:m_inf_formula} one has
    \[
        \frac{\tau(x_na)}{\tau(xa)} = \frac{\tau(x_na)}{\tau(x_ma)}\frac{\tau(x_ma)}{\tau(xa)} \geq m(x_n,x_m) \frac{\tau(xa) - (1-\eta^{1/4})\tau(xa)}{\tau(xa)} = m(x_n,x_m) \eta^{1/4}> \eta^{1/2}.
    \]
Taking an infimum over $a\in M_+$ with $\tau(xa)>0$ yields $m(x_n,x)\geq \eta^{1/2}$. Next, for $a\in M_+$ with $\tau(x_na)>0$ let $m\geq N$ be large enough so that $|\tau(x_m a) - \tau(xa)| < (\eta^{-1/4} -1)\tau(xa)$ (and note $\tau(x_ma)\neq 0$ lest $m(x_m,x_n)=0$ and $d(x_m,x_n)=1$). Then
    \[
        \frac{\tau(xa)}{\tau(x_n a)} = \frac{\tau(x_ma)}{\tau(x_n a)}\frac{\tau(xa)}{\tau(x_m a)} \geq m(x_m,x_n) \frac{\tau(xa)}{\tau(xa) + (\eta^{-1/4} -1)\tau(xa)} = m(x_m,x_n) \eta^{1/4} > \eta^{1/2}.
    \]
Taking an infimum gives $m(x,x_n) \geq \eta^{1/2}$. Altogether this gives $d(x_n,x) \leq (1-\eta)/(1+\eta)=\epsilon$, so that $x_n\to x$ with respect to the metric $d$.
\end{proof}

\begin{rmk}
Superficially, there is striking visual resemblance between Inequality~(\ref{eqn:dinequality}) and the famous Pinsker's inequality \cite[Theorem 3.1]{HiaiOhyaTsukuda} for the relative entropy: $\frac{1}{2}\|\phi - \psi\|_1^2 \le S(\phi|\psi)$. Moreover the norm of the Radon-Nikodym derivatives of states (when defined) satisfies $\|\frac{d\phi}{d\psi}\|^2=\inf\{ \lambda: \psi \le \lambda \phi\}$ \cite[Theorem 12]{araki}, \cite[Appendix 7]{hiai} which is inversely proportional to $m( \frac{d\phi}{d\tau}, \frac{d\psi}{d\tau})$.
\end{rmk}

In light of Inequality~(\ref{eqn:dinequality}), it is natural to ask if the metric $d$ is equivalent to the metric induced by $\|\cdot\|_1$. This is always false when $M$ is infinite dimensional (see Remark~\ref{rmk:d_not_homeo_to_1-norm}), but nevertheless the Hennion metric is jointly lower semicontinuous with respect to $\|\cdot\|_1$:

\begin{thm}\label{thm:lower_semicontinuity_of_d}
Suppose $(x_n)_{n\in \N}, (y_n)_{n\in \N}\subset S$ satisfy
    \[
        \lim_{n\to\infty} \|x_n - x\|_1 =0 \qquad \text{ and } \qquad \lim_{n\to\infty} \|y_n - y\|_1=0,
    \]
for some $x,y\in S$. Then
    \[
        \limsup_{n\to\infty} m(x_n,y_n) \leq m(x,y),
    \]
and
    \[
        \liminf_{n\to\infty} d(x_n,y_n) \geq d(x,y).
    \]
\end{thm}
\begin{proof}
The definition of $d$ implies its joint lower semicontinuity will follow from the joint upper semicontinuity of $m$. To see the latter, let $\eta> m(x,y)$ and use Theorem~\ref{thm:mproperties}.\ref{part:m_inf_formula} to find $a\in M_+$ such that
    \[
        m(x,y)\leq \frac{\tau(xa)}{\tau(ya)} < \eta.
    \]
Then we can find $N\in \N$ so that 
    \[
        m(x_n,y_n) \leq \frac{\tau(x_na)}{\tau(y_na)} <\eta
    \]
holds for all $n\geq N$. Thus
    \[
        \limsup_{n\to\infty} m(x_n,y_n) \leq \eta,
    \]
and letting $\eta\to m(x,y)$ completes the proof.
\end{proof}

\begin{rmk}
This is another striking similarity between the Hennion metric and the relative entropy. We recall that the relative entropy is jointly lower semicontinuous in its arguments \cite[Theorem 4.1]{Kosaki}.
\end{rmk}


\subsection{On components of \texorpdfstring{$S$}{S}}\label{subsec:Components}
We demonstrate below that $S$ is the disjoint union of at least four connected components. Recall that we say a densely defined linear operator $x$ on $L^2(M,\tau)$ is \emph{boundedly invertible} if there exists $b\in B(L^2(M,\tau))$ satisfying $xb=1$ and $bx\subset 1$, in which case we write $x^{-1}:=b$. Note that if $x\in L^1(M,\tau)$ (and hence affiliated with $M$), then necessarily $x^{-1}\in M$.

\begin{note}\label{def:Sstar}
We let $S^\times$ denote the set of operators $x\in S$ that are boundedly invertible, and we denote $S^\circ:=S\setminus S^\times$. We also set the following notation:
    \begin{align*}
        S^\times_b &:=S^\times \cap M & S^\times_u&:=S^\times \setminus M\\
        S^\circ_b&:= S^\circ\cap M & S^\circ_u&:= S^\circ\setminus M.
    \end{align*}
Observe that
    \[
        S=S^\times_b\sqcup S^\times_u\sqcup S^\circ_b \sqcup S^\circ_u,
    \]
and $S^\times= S^\times_b \sqcup S^\times_u$ and $S^\circ = S^\circ_b \sqcup S^\circ_u$. We shall also write $S_b := S\cap M (= S^\times _b \sqcup S^\circ _b)$ and $S_u :=S\setminus M (= S^\times_u \sqcup S^\circ_u)$. Also note the sets $S^\times$, $S_b$, and $S_b^\times$ are convex.$\hfill\blacksquare$
\end{note}

\begin{thm}\label{thm:geometry}
Let $(M,\tau)$ be a tracial von Neumann algebra. We have the following: 
    \begin{enumerate}[label=(\arabic*)]
        \item  Suppose $x,y\in S$ satisfy $d(x,y)<1$. Then, $x$ is bounded if and only if $y$ is bounded; and $x$ is boundedly invertible if and only if $y$ is boundedly invertible. In particular, $S_b^\times, S_u^\times, S_b^\circ, S_u^\circ$ are disjoint $d$-clopen sets that are distance one apart (provided they are non-empty).
        \item For all $x,y\in S^\times_b$, one has $d(x,y)<1$.
        \item If $M\neq \C$, then $\diam(S_b^\times)=1$ and $\diam(S_b^\circ)=1$, and the latter is achieved.
        \item If $M$ is infinite dimensional, then $\diam(S_u^\times)=1$ and $\diam(S_u^\circ)=1$, and both are achieved.
        \item If $M$ is infinite dimensional, $S_b^\times$ is not $d$-totally bounded.
    \end{enumerate}
\end{thm}
\begin{proof} 
\begin{enumerate}[label = \textbf{(\arabic*):}]
\item Notice that $d(x,y)<1$ implies that $m(x,y),m(y,x)>0$. We have $m(x,y) y \leq x \leq m(y,x)^{-1} y$, hence $x$ is bounded if and only if $y$ is bounded by Remark~\ref{rmk:boundedimpliesbounded}, and $x$ is boundedly invertible if and only if $y$ is boundedly invertible by Remark~\ref{rmk:boundedlyinvert_implies_boundedlyinvert}.

Now, if $x,y\in S$ belong to distinct subsets $S_b^\times$, $S_b^\circ$, $S_u^\times$, or $S_u^\circ$, then the above implies $d(x,y)=1$. Consequently, one has
    \[
        S_b^\times = \bigcup_{x\in S_b^\times} B_d(x,\frac12),
    \]
and similarly for $S_b^\circ$, $S_u^\times$, and $S_u^\circ$. Thus each of these sets are open, and since they partition $S$ they are also closed.
\item Note that $m(x,y)\geq \frac{1}{\|y\|\|x^{-1}\|} >0$ (by Remark~\ref{rmk:boundedlyinvert_implies_boundedlyinvert}) and similarly $m(y,x)>0$. Hence $d(x,y)<1$.
\item Let $p\in M$ be a non-trivial projection, and for $\epsilon\geq 0$ let
    \[
        x_\epsilon:= \frac{1}{1+\epsilon}(p + \epsilon(1-p)) \qquad \qquad y_\epsilon:= \frac{1}{1+\epsilon}(\epsilon p + (1-p)).
    \] 
For $\epsilon>0$, $x_\epsilon,y_\epsilon \in S_b^\times$ with $m(x_\epsilon, y_\epsilon)= m(y_\epsilon, x_\epsilon)=\epsilon$, and therefore $d(x_\epsilon, y_\epsilon)\to 1$ as $\epsilon \to 0$. For $\epsilon=0$, $x_0,y_0\in S_b^\circ$ with $m(x_0, y_0)=m(y_0,x_0)=0$, and therefore $d(x_0, y_0)=1$.

\item Since $M$ is infinite dimensional, Lemma~\ref{lem:infinite_dimensional} yields a family $\{p_n\colon n\in \N\}\subset M$ of non-zero pairwise orthogonal projections mutually orthogonal projections satisfying $\sum p_n =1$. Note that $\tau(p_n) \to 0$ as $n\to \infty$. Pick a subsequence $(p_{n_k})_{k\in \N}$ so that $\tau(p_{n_k})<(k2^k)^{-1}$ and $n_k$ is even. Let $A = \{n \in \N : n \neq n_k \text{ for all } k\}$. Set $x_k = kp_{n_k}$ and observe that
    \[
        \tau \left(\sum_{k=1}^\infty x_k\right) = \sum_{k=1}^\infty k\tau(p_{n_k}) < \sum_{k=1}^\infty  \frac{1}{2^k} = 1,
    \] 
while by pairwise orthogonality, 
    \[
        \left \| \sum_{k=1}^d x_k\right\| = d.
    \] 
Now, set $\alpha := \sum_{k=1}^\infty \tau(x_k)<1$ and $\delta := (1-\alpha) (\sum_{n\in A} \tau(p_n))^{-1}$. Note that
    \[
        \alpha > \sum_{k=1}^\infty \tau(p_{n_k}) = 1 - \sum_{n\in A} \tau(p_n)
    \]
implies $\delta<1$. Thus if we let
    \[
        x:= \sum_{k=1}^\infty x_k + \delta \sum_{n\in A} p_n,
    \]
then $x\in S_u^\times$ with $x\geq \delta$. Repeat the above construction but this time choosing the subsequence $(p_{m_k})_{k\in \N}$ so that $m_k$ is odd to obtain $y\in S_u^\times$ with $y\geq \delta$. Then for any $\lambda>0$, $\lambda y \leq x$ fails since multiplying by $p_{m_k}$ with $k> \frac{\delta}{\lambda}$ gives the contradiction $\lambda k p_{m_k} \leq \delta p_{m_k}$. Hence $d(x,y)=1$ and $\diam(S_u^\times)=1$ and is achieved.

The same construction with $\delta=0$ shows $\diam(S_u^\circ)=1$ and is achieved.
    
\item Let $\{p_n\colon n\in \N\}\subset M$ be as in the previous part. Put $s_n = \sum_{k> n}p_k$ and observe that $\tau(s_n) \to 0$ as $n\to \infty$. Pick a subsequence $s_{n_k}$ so that $\tau(s_{n_k}) \le (2k)^{-1}$, and set $x_k = \frac12 \sum_{j=1}^{n_k} p_j + k s_{n_k}$. Note that $\alpha_k:= \frac{1}{\tau(x_k)}\in (1,2)$. 

Now, $(\alpha_k x_k)_{k\in \N}$ is a sequence of states that are bounded below by $\frac12$ and above by $2k$, and so $(\alpha_k x_k)_{k\in \N} \subset S_b^\times$. For $k>k'$ one has $m(x_k, x_{k'})\le \frac{1}{2k'}$ and $m(x_{k'}, x_k) \le \frac{k'}{k}$. Indeed, we see $m(x_k,x_{k'}) x_{k'} \le x_k$ demands $m(x_k,x_{k'}) k'\sum_{j=n_{k'}+1}^{n_k} p_j\le \frac12\sum_{j=n_{k'}+1}^{n_k} p_j$, while $m(x_{k'},x_k)x_k \le x_{k'}$ implies $m(x_{k'},x_k) k s_{n_k} \leq k' s_{n_k}$. Therefore by Theorem~\ref{thm:mproperties}.\ref{part:mscaling},
    \begin{align*}
        d(\alpha_k x_k, \alpha_{k'} x_{k'})&\ge \frac{1-\frac{k'}{k}\frac{1}{2k'}}{1 + \frac{k'}{k}\frac{1}{2k'}}\\
        &= \frac{2k-1}{2k+1}\ge 1/3,
    \end{align*}  uniformly in $k$. Thus $S_b^\times$ does not admit a finite cover by balls of radius $\frac16$ or less.\qedhere
\end{enumerate}
\end{proof}

\begin{rmk}\label{rmk:d_not_homeo_to_1-norm}
Note that unlike in the finite dimensional case (see \cite[Lemma 3.9]{jeff}), for infinite dimensional $M$ the $d$-topology and $\|\cdot\|_1$-topologies are not homeomorphic on $S^\times$ or even on $S^\times_b$. Indeed, let $(p_n)_{n\in \N}\subset M$ be as in Lemma~\ref{lem:infinite_dimensional}.(iii). Then for $0<\alpha<1$ and
    \[
        x_n:= \frac{1}{\alpha + (1-\alpha)\tau(p_n)}(p_n + \alpha(1-p_n)),
    \]
$x_n\to 1$ in $L^1(M,\tau)$ but it is not a Cauchy sequence with respect to $d$.$\hfill\blacksquare$
\end{rmk}

\begin{lem}\label{thm:cauchytocauchy}
Let $(M,\tau)$ be a tracial von Neumann algebra with Hennion metric $d$. Let $(x_n)_{n\in \N}\subset S_b$ and $x\in S$.
    \begin{enumerate}[label=(\arabic*)]
    \item If $d(x_n,x)\to 0$ then $x\in S_b$ and $\|x_n - x\|\to 0$.
    
    \item Let $(x_n)_{n\in \N}\subset S_b^\times$ with $\|x_n -x\| \to 0$ and suppose $x\in S_b^\times$. Then, $d(x_n,x)\to 0$.
    \end{enumerate}
Consequently, $(S_b^\times,d)$ and $(S_b^\times,\|\cdot\|)$ are homeomorphic.
\end{lem}
\begin{proof}
\begin{enumerate}[label=\textbf{(\arabic*):}]
\item Suppose $(x_n)_{n\in \N}\subset S_b$ converges to some $x$ with respect to $d$. Note that necessarily $x\in S_b$ by Theorem~\ref{thm:geometry}. Let $N_0\in \N$ be sufficiently large so that $d(x_n,x)<\frac{1}{2}$ for all $n\ge N_0$. Then, we see from the definition of $d$ that, 
    \[
        \frac{1}{3} < m(x_n, x)m(x, x_n) \le \min\{m(x_n, x),m(x, x_n)\}
    \] 
using the fact that $m(x,y)\le 1$ for states. In particular, this means that for any $n\ge N_0$, we have $x_n \le 3 x$ whence $\| x_n\| \le 3 \|x\|$. Let $R = \max_{1\le j \le N_0-1}\{ \|x_j\|, 3 \| x\|\}$ so that $(x_n)_{n\in\N}\subset (M)_R$.

Now, for $1>\epsilon>0$ let $N\in \N$  be sufficiently large so that for $n\geq N$, we have $d(x_n, x)<\epsilon$. In particular, this implies that 
    \[
        \frac{1-\epsilon}{1+\epsilon}\le m(x_n, x)m(x, x_n)\le \min\{ m(x_n,x), m(x,x_n) \}.
    \] 
This yields the following inequalities
    \begin{equation*}
        \frac{1-\epsilon}{1+\epsilon} x \le x_n \qquad \text{ and }\qquad \frac{1-\epsilon}{1+\epsilon}  x_n \le x,
    \end{equation*} 
for all $n \ge N$. Rearranging, we get 
        \[
                x_n \le x + 2 \epsilon R \qquad \text{ and }\qquad
                x\leq x_n + 2\epsilon R,
        \] 
for any $n\geq N$. Therefore
    \[
        -2\epsilon R = x-2\epsilon R  - x \le x_n - x \le 2\epsilon R + x - x = 2\epsilon R,
    \] 
whence for any $\xi \in L^2(M,\tau)$ with $\|\xi\|_2 =1$ we have 
    \[
        |\langle (x_n-x) \xi, \xi \rangle | \le 2\epsilon R,
    \] for any $n,k\ge n_0$. Taking the supremum over $\xi$ this yields $\|x_n - x\| < 2\epsilon R$. Hence $x_n\to x$ in operator norm.
    
\item Suppose $(x_n)_{n\in \N}\subset S_b^\times$ converges to $x\in S_b^\times$ with respect to the operator norm. Let $\epsilon>0$. Since $x\in S_b^\times$, there exists $\delta>0$ so that $x\geq \delta$. So for sufficiently large $n\in \N$ we have
    \[
        x_n \leq x + \|x_n-x\| < x + \epsilon\delta \leq (1+\epsilon)x, 
    \]
and hence $m(x_n,x) \geq (1+\epsilon)^{-1}$. Next, for sufficiently large $n$ one has
    \[
        x_n \geq x - \|x-x_n\| \geq \frac{\delta}{2}.
    \]
Increasing $n$ if necessary, we then obtain
    \[
        x\leq x_n + \|x-x_n\| < x_n + \epsilon\frac{\delta}{2} \leq  (1+\epsilon)x_n,
    \]
so that $m(x,x_n) \geq (1+\epsilon)^{-1}$. Thus
    \[
        \lim_{n\to\infty} d(x_n,x) \leq \frac{1-(1+\epsilon)^{-2}}{1+(1+\epsilon)^{-2}},
    \]
and letting $\epsilon\to 0$ completes the proof.\qedhere
\end{enumerate} 
\end{proof}

Part (2) of Lemma~\ref{thm:cauchytocauchy} fails when the limit is not boundedly invertible---even in the finite dimensional case---as the following example demonstrates. Moreover, this example also shows that the Hennion metric is \emph{not} equivalent to metric induced by the operator norm on $S_b^\times$, despite inducing the same topology by Lemma~\ref{thm:cauchytocauchy}.

\begin{exmp}\label{exmp:norm_convergent_not_d_convergent}
Let $M = \mathbb{M}_{2\times 2}$ be the algebra of $2\times 2$ matrices and consider the family of states (with respect to the normalized tracial state)
    \[
        X_\eta := \frac{2}{1+\eta} \begin{bmatrix}
            1 & 0 \\ 0 & \eta \\
        \end{bmatrix}.
    \] 
As $\eta\to 0$, the state $X_\eta$ converges in operator norm to the re-scaled projection $P = \begin{bmatrix} 2  & 0\\ 0 & 0\end{bmatrix}$. However, for any $\eta' < \eta$, one can calculate that 
    \[
        m(X_\eta, X_\eta')m(X_{\eta'}, X_\eta)=  \frac{\eta'}{\eta}.
    \] 
Therefore by choosing $\eta \ge 3\eta '$, one obtains $d(X_\eta, X_{\eta'}) \ge \frac12$, thus $(X_\eta)_{\eta>0}$ has no subsequences which are Cauchy with repsect to the Hennion metric. We also mention that because of \cite[Lemma 3.3(4)]{jeff}, $d(X_\eta, P) \equiv 1$ for all $\eta>0$.$\hfill\blacksquare$
\end{exmp}

\section{Contraction Mappings on \texorpdfstring{$S$}{S}}\label{sec:cmap}
In this section we will consider linear maps $\gamma$ on $L^1(M,\tau)$ and their contraction properties with respect to the Hennion metric. We completely characterize faithful maps that contract on $S$ in terms of an operator inequality (see Theorem~\ref{thm:strict_Hennion_projective_actions}). Of particular interest for application are those $\gamma$ which arise as the preduals of normal positive maps on $M$ which are investigated in Section~\ref{subsubsec:SHCfromnormal}.
\subsection{The class of mappings and projective actions}
\begin{lem}\label{lem:clopenkernel}
For a positive linear map $\gamma$ on $L^1(M,\tau)$, $\ker\gamma\cap S$ is open in the Hennion metric. If $\gamma$ is bounded, then this set is also closed.
\end{lem}
\begin{proof}
For any $x\in \ker \gamma \cap S$, the open ball $B_d(x, 1)$ is contained in $\ker \gamma \cap S$. Indeed, since $y\in B_d(x,1)$ implies that $m(x,y)m(y,x)>0$, we then have $y \le \frac{1}{m(x,y)}x$ and applying $\gamma$ gives $y\in \ker \gamma \cap S$. If $\gamma$ is bounded, then $\ker\gamma\cap S$ is closed with respect to $\|\cdot\|_1$, and so Theorem~\ref{thm:tracemetricinequality_and_completeness} implies it is also closed in the Hennion metric.
\end{proof}

\begin{define}\label{def:projectiveaction}
Let $\gamma$ be a positive map on $L^1(M,\tau)$. The \textbf{projective action} of $\gamma$ on $x\in S$ is the map $\gamma\, \cdot\, \colon S\setminus \ker \gamma \to S$ given by 
    \begin{equation*}
        \gamma \cdot x = \frac{1}{\tau(\gamma(x))} \gamma(x).
    \end{equation*}  
\end{define}

\begin{rmk}\label{rmk:projective_actions_preserve_lines}
Although the projective action is non-linear, it does preserve lines in the following sense. Given $x,y\in S\setminus \ker\gamma$, for any $t\in \mathbb{R}$ satisfying $tx+(1-t)y\in S\setminus \ker\gamma$ there exists $s\in \mathbb{R}$ so that
    \[
        \gamma\cdot (tx + (1-t)y) = s\gamma\cdot x + (1-s)\gamma\cdot y;
    \]
namely, $s:=\tau(\gamma(tx))\tau(\gamma(tx+(1-t)y))^{-1}$. In particular, if $t\in [0,1]$ then $s\in [0,1]$.$\hfill\blacksquare$
\end{rmk}

The following lemma implies projective actions of positive maps are always Lipschitz continuous with respect to the Hennion metric. We explore the associated Lipschitz constants in greater detail in Section~\ref{sec:contraction_mappings}.

\begin{lem}\label{lem:dantitone}
For a positive linear map $\gamma$ on $L^1(M,\tau)$,
    \begin{equation}\label{eqn:dantitone}d(\gamma \cdot x, \gamma \cdot y) \le d(x,y).\end{equation} 
for all $x,y\in S\setminus \ker\gamma$.
\end{lem}
\begin{proof}
First observe that $m(\gamma\cdot x,\gamma\cdot y)m(\gamma\cdot y, \gamma\cdot x) = m(\gamma(x),\gamma(y))m(\gamma(y),\gamma(x))$ by Theorem~\ref{thm:mproperties}.\ref{part:mscaling}. Thus
    \begin{align*}
        d(\gamma \cdot x, \gamma \cdot y) &= \frac{1 - m(\gamma\cdot x, \gamma \cdot y)m(\gamma \cdot y, \gamma \cdot x)}{1+m(\gamma\cdot x, \gamma \cdot y)m(\gamma \cdot y, \gamma \cdot x)} = \frac{1 - m(\gamma( x), \gamma(y))m(\gamma(y), \gamma(x))}{1+m(\gamma(x), \gamma(y))m(\gamma(y), \gamma(x))}.
    \end{align*}
Next, applying $\gamma$ to $m(x,y) y \leq x$ yields $m(x,y) \gamma(y)\leq \gamma(x)$ so that $m(x,y)\leq m(\gamma(x),\gamma(y))$. Since $\frac{1-t}{1+t}$ is decreasing, we can continue the above computation with
    \[
        d(\gamma \cdot x, \gamma \cdot y)\le \frac{1 - m(x,y)m(y,x)}{1+m(x,y)m(x,y)} = d(x,y). \tag*{\qedhere}
    \]
\end{proof}

For faithful maps, the above inequality can be refined as follows.

\begin{lem}\label{lem:jeffinequality}
Let $\gamma$ be a faithful positive linear map on $L^1(M,\tau)$. For any $x,y\in S$, we have 
\begin{equation}\label{eqn:jeffsinequality}
    d(\gamma \cdot x, \gamma \cdot y) \le d(\gamma \cdot A_-, \gamma \cdot A_+) d(x,y)
\end{equation} where $A_{\pm}$ are defined as in Lemma~\ref{lem:dformula}.
\end{lem}
\begin{proof}
The proof is essentially the same as in \cite[Lemma 3.10(1)]{jeff}. Let $x,y\in S$ be distinct and note that we can assume $ \gamma\cdot x, \gamma\cdot y$ are also distinct since otherwise the inequality is trivially true. Applying Lemma~\ref{lem:dformula} to the pairs $x,y$ and $\gamma\cdot x, \gamma\cdot y$ gives $A_{\pm},B_{\pm}\in S$, $t_{\pm}, w_{\pm}\in \mathbb{R}$, and $r,s, u,v\in [0,1]$ satisfying
    \begin{align}\label{eqn:convex_relations_for_xygammaxgammay}
    \begin{split}
        A_{\pm} = t_{\pm} x + (1-t_{\pm})y \qquad\qquad  x &= r A_- + (1-r) A_+ \qquad\qquad  y = s A_- + (1-s) A_+\\
        B_{\pm} = w_{\pm} \gamma\cdot x + (1-w_{\pm})\gamma\cdot y \qquad\qquad  \gamma\cdot x &= u B_- + (1-u) B_+ \qquad\qquad  \gamma\cdot y = v B_- + (1-v) B_+.
    \end{split}
    \end{align}
Additionally, we have
    \[
        d(x,y) = \left|\frac{r(1-s)-(1-r)s}{r(1-s)+(1-r)s}\right| \qquad\qquad d(\gamma\cdot x,\gamma\cdot y)= \left|\frac{u(1-v)-(1-u)v}{u(1-v)+(1-u)v}\right|.
    \]
By Remark~\ref{rmk:projective_actions_preserve_lines}, we have
    \begin{align*}
        \gamma\cdot A_{\pm} = \frac{\tau(\gamma(t_\pm x))}{\tau(\gamma(t_\pm x+ (1-t_\pm) y))} \gamma\cdot x + \frac{\tau(\gamma((1-t_\pm) y))}{\tau(\gamma(t_\pm x+ (1-t_\pm) y))} \gamma\cdot y.
    \end{align*}
In particular, $\gamma\cdot A_{\pm}$ lie in the convex set
        \[
            \{w\gamma\cdot x + (1-w)\gamma\cdot y\in S\colon w_-\leq w\leq w_+\}\cap S, 
        \]
and therefore $\{p\gamma\cdot A_- + (1-p)\gamma\cdot A_+\colon p\in \mathbb{R}\}\cap S$ has the same extreme points; namely, $B_{\pm}$. Thus there exist $p,q\in [0,1]$ satisfying
    \begin{align*}
        \gamma\cdot A_- &= p B_- + (1-p)B_+\\
        \gamma\cdot A_+ &= q B_- + (1-q)B_+,
    \end{align*}
and
    \[
        d(\gamma\cdot A_-, \gamma\cdot A_+) = \left|\frac{p(1-q)-(1-p)q}{p(1-q)+(1-p)q}\right|,
    \]
by Lemma~\ref{lem:dformula}. Using the above formulae we have
    \begin{align*}
        u B_- + (1-u) B_+ &= \gamma\cdot x \\
            &= \frac{1}{\tau(\gamma(x))}(r \gamma(A_-) + (1-r) \gamma(A_+))\\
            &= \frac{\tau(\gamma(A_-)) rp + \tau(\gamma(A_+)) (1-r)q}{\tau(\gamma(x))} B_- + \frac{\tau(\gamma(A_-)) r(1-p) + \tau(\gamma(A_+)) (1-r)(1-q)}{\tau(\gamma(x))}B_+.
    \end{align*}
This determines $u,1-u$ since $B_{\pm}$ are distinct states (the line connecting them contains the distinct states $\gamma\cdot x$ and $\gamma\cdot y$) and are therefore linearly independent. A similar computation yields
    \[
        v = \frac{\tau(\gamma(A_-)) sp + \tau(\gamma(A_+)) (1-s)q}{\tau(\gamma(y))} \qquad\qquad 1-v = \frac{\tau(\gamma(A_-)) s(1-p) + \tau(\gamma(A_+)) (1-s)(1-q)}{\tau(\gamma(y))}.
    \]
Writing $\alpha:=\tau(\gamma(A_-))$ and $\beta:=\tau(\gamma(A_+))$, we have:
    \begin{align*}
        d(\gamma\cdot x , \gamma\cdot y) &= \left|\frac{u(1-v)-(1-u)v}{u(1-v)+(1-u)v}\right| \\
        &= \frac{\alpha \beta |p(1-q)-(1-p)q||r(1-s)- (1-r)s|}{2\alpha^2 p(1-p)rs+\alpha\beta (p(1-q)+(1-p)q)(r(1-s)+(1-r)s) + 2\beta^2q(1-q)(1-r)(1-s)}\\
        &\leq \frac{|p(1-q)-(1-p)q||r(1-s)- (1-r)s|}{(p(1-q)+(1-p)q)(r(1-s)+(1-r)s)} = d(\gamma\cdot A_-, \gamma\cdot A_+)d(x,y),
    \end{align*}
where the second equality follows from an abundance of arithmetic and the inequality follows from $2\alpha^2 p (1-p)rs + 2\beta^2 q(1-q)(1-r)(1-s)\geq 0$.
\end{proof}

\subsection{Contraction mappings}\label{sec:contraction_mappings}

Given a metric space $(X,\rho)$, recall that a mapping $T:X\to X$ is said to be a \textit{contraction mapping} if there is a constant $0\le q<1$ so that $\rho(T(p_1), T(p_2))\le q \rho(p_1, p_2)$ for all $p_1,p_2\in X$. Since we have shown (Theorem \ref{thm:tracemetricinequality_and_completeness}) that the normal state space $S\subset L^1(M,\tau)$ is complete in Hennion's Metric, we shall aim to give a characterization of a large class of contractions with respect to Hennion's metric. 

\begin{define}\label{def:Hennion_contraction}
The \textbf{Hennion contraction constant} of a faithful positive linear map $\gamma$ on $L^1(M,\tau)$ is the quantity
    \begin{equation}\label{eqn:contraction}
        c(\gamma):=\sup_{x,y\in S} \frac{d(\gamma\cdot x,\gamma\cdot y)}{d(x,y)}.
    \end{equation}
We say $\gamma$ is a \textbf{strict Hennion contraction} if $c(\gamma)< 1$, and we denote by $SHC(M)$ the family of all such maps.
\end{define}

Lemma~\ref{lem:dantitone} implies $c(\gamma)\leq 1$ for all faithful positive linear maps $\gamma$ on $L^1(M,\tau)$. Using Lemma~\ref{lem:jeffinequality} we actually have
    \[
        \frac{d(\gamma\cdot x, \gamma\cdot y)}{d(x,y)} \leq d(\gamma\cdot A_-, \gamma\cdot A_+),
    \]
and consequently $c(\gamma)\leq \diam(\gamma\cdot S)$. The reverse inequality is a consequence of $d(x,y)\leq 1$ for all $x,y\in S$, and so we have proven the following:

\begin{prop}\label{prop:cont_const_is_diam}
For a faithful positive linear map $\gamma$ on $L^1(M,\tau)$ one has
    \[
        c(\gamma)=\diam(\gamma\cdot S).
    \]
\end{prop}

If we further assume $\gamma$ is bounded as linear operator on $L^1(M,\tau)$, then
the joint lower semicontinuity of $d$ with respect to $\|\cdot\|_1$ (see Theorem~\ref{thm:lower_semicontinuity_of_d}) implies the contraction constant can be witnessed on any $\|\cdot\|_1$-dense subset of $L^1(M,\tau)$. 

\begin{prop}\label{prop:contr_const_for_bounded_maps}
Let $\gamma$ be a bounded faithful positive linear map on $L^1(M,\tau)$. For any $\|\cdot\|_1$-dense subset $S_0\subset S$, one has
    \[
        c(\gamma) = \diam(\gamma\cdot S_0).
    \]
\end{prop}
\begin{proof}
Let $\eta < c(\gamma)$ and let $x,y\in S$ be such that 
    \[
        \eta < d(\gamma\cdot x, \gamma\cdot y) \leq c(\gamma).
    \]
Letting $(x_n)_{n\in \N}, (y_n)_{n\in \N}\subset S_0$ be sequences converging to $x$ and $y$, respectively, with respect to $\|\cdot\|_1$. Then $\gamma\cdot x_n \to \gamma\cdot x$ and $\gamma\cdot y_n\to \gamma\cdot y$ with respect to $\|\cdot\|_1$, and so Theorem~\ref{thm:lower_semicontinuity_of_d} implies
    \[
        \diam(\gamma\cdot S_0) \geq \liminf_{n\to\infty} d(\gamma\cdot x_n , \gamma\cdot y_n) \geq d(\gamma\cdot x, \gamma\cdot y ) > \eta.
    \]
Letting $\eta\to c(\gamma)$ completes the proof.
\end{proof}

\begin{exmp}\label{exmp:strictly_positive}
Let $(M,\tau) = (\M_n, \frac1n \text{Tr})$ be the $n\times n$ matrices equipped with its normalized trace. If $\gamma\colon \M_n\to \M_n$ is \emph{strictly positive} in the sense that $\gamma\cdot S\subset S^\times$, then it is a strict Hennion contraction. Indeed, $S$ is compact in this case and therefore  
    \[
        c(\gamma)= \diam(\gamma\cdot S)=d(\gamma\cdot x, \gamma\cdot y),
    \]
for some $x,y\in S$. Using Theorem~\ref{thm:geometry}.(2), we see that $c(\gamma)<1$. Conversely, if $\gamma$ is a strict Hennion contraction and $\gamma\cdot 1\in S^\times$, then $\gamma$ is strictly positive by Theorem~\ref{thm:geometry}. $\hfill\blacksquare$
\end{exmp}

\begin{exmp}[A Non-Example]
Let $(M,\tau) = (\M_n, \frac1n \text{Tr})$ be the $n\times n$ matrices equipped with its normalized trace. The transpose map $x\mapsto x^T$ is a unital and tracial map that is well known to be positive but not completely positive (see \cite{paulsen}). Moreover, for $x,y\in (\M_n)_+$ one has $x\leq y$ iff $x^T\leq y^T$. Thus $S\ni x\mapsto x^T$ is an isometry with respect to $d$ and is therefore not a strict Hennion contraction. $\hfill\blacksquare$
\end{exmp}

We have seen in Theorem~\ref{thm:tracemetricinequality_and_completeness} that $(S,d)$ is a complete metric space, and so the Banach Fixed Point Theorem implies a strict Hennion contraction $\gamma$ has a unique fixed point: $\gamma\cdot x_0 = x_0$ for some $x_0\in S$. In fact, this along with the partial order on the normal state space can be used to characterize when one has a strict Hennion contraction:

\begin{thm}\label{thm:strict_Hennion_projective_actions}
For a faithful positive linear map $\gamma$ on $L^1(M,\tau)$, the following are equivalent:
    \begin{enumerate}[label=(\roman*)]
    \item $\gamma$ is a strict Hennion contraction.
    
    \item\label{part:fixed_point_dilation_constant} For some $x_0\in S$ there exists $\eta\in (0,1]$ so that
        \[
            \eta x_0 \leq \gamma\cdot x \leq \eta^{-1} x_0 \qquad \forall x\in S.
        \]

    \end{enumerate}
In this case, $x_0$ can be taken to be the unique fixed point of the projective action. Moreover, if $(\gamma\cdot S)\cap S_b^\times \neq \emptyset$ then the above are further equivalent to
    \begin{enumerate}
        \item[(iii)] For each $y_0\in S_b^\times$ there exists $\kappa\in (0,1]$ so that
            \[
                \kappa y_0 \leq \gamma\cdot x \leq \kappa^{-1} y_0 \qquad \forall x\in S.
            \]
    \end{enumerate}
In this case, one has $\gamma\cdot S \subset S_b^\times$.
\end{thm}
\begin{proof}
$(i)\Rightarrow (ii)$: By the discussion preceding the statement of the theorem, there exists a unique $x_0\in S$ so that $\gamma\cdot x_0 = x_0$. Then Proposition~\ref{prop:cont_const_is_diam} gives
    \[
        d(\gamma \cdot x, x_0) = d(\gamma\cdot x, \gamma\cdot x_0) \leq c(\gamma)<1,
    \]
for all $x\in S$. Arguing as in Theorem~\ref{thm:tracemetricinequality_and_completeness}, we see that 
    \[
        \min\{ m(\gamma\cdot x, x_0), m(x_0, \gamma\cdot x)\}\ge \frac{1-c(\gamma)}{1+c(\gamma)} =:\eta,
    \] 
for all $x\in S$. Thus
    \[
        \eta x_0 \leq m(\gamma\cdot x,x_0) x_0 \leq \gamma\cdot x \leq m(x_0,\gamma\cdot x)^{-1}x_0 \leq \eta^{-1} x_0,
    \]
as claimed.\\

\noindent$(ii)\Rightarrow (i)$: It follows from $\eta x_0 \leq \gamma\cdot x\leq \eta^{-1}x_0$ that $m(\gamma\cdot x, x_0)m(x_0,\gamma\cdot x) \geq \eta^2$ for all $x\in S_0$. Using Theorem~\ref{thm:mproperties}.\ref{part:mtriangle_ineq}, we then have 
    \[
        m(\gamma\cdot x,\gamma\cdot y)m(\gamma\cdot y,\gamma\cdot x) \geq \eta^4
    \]
for all $x,y\in S$, and therefore
    \[
        d(\gamma\cdot x,\gamma\cdot y) \leq \frac{1-\eta^4}{1+\eta^4}.
    \]
Thus
    \[
        c(\gamma)=\diam(\gamma \cdot S)\leq \frac{1-\eta^4}{1+\eta^4} < 1,
    \]
where the first equality follows from Proposition~\ref{prop:cont_const_is_diam}.\\

\noindent Now suppose $(\gamma\cdot S)\cap S_b^\times \neq \emptyset$. If $\gamma$ is a strict Hennion contraction, then $\gamma\cdot S\subset S_b^\times$ holds by Theorem~\ref{thm:geometry}. Consequently, for the fixed point $x_0=\gamma\cdot x_0$ if we let $\delta := \min\{\|x_0\|^{-1}, \|x_0^{-1}\|^{-1}\}\in (0,1]$ then $\delta 1 \leq x_0 \leq \delta^{-1} 1$. Let $y_0\in S_b^\times$, let $\eta\in (0,1)$ be as in $\ref{part:fixed_point_dilation_constant}$, and set $\kappa:=\delta \eta\min\{\|y_0\|^{-1},\|y_0^{-1}\|^{-1}\}$. Then one has
    \[
        \kappa y_0 \leq \delta \eta \leq  \eta x_0 \leq \gamma\cdot x \leq \eta^{-1} x_0 \leq (\delta\eta)^{-1} 1 \leq \kappa^{-1} y_0,
    \]
for all $x\in S$. The converse (namely, $(iii)\Rightarrow (ii)$) is immediate.
\end{proof}

\begin{exmp}\label{exmp:strongly_summable}
Fix a non-zero $m\in L^1(M,\tau)_+$ and $\eta\in (0,1]$. Let $\{(a_i,m_i)\in M_+\times L^1(M,\tau)_+\colon i\in I\}$ be a family satisfying:
    \begin{enumerate}[label=(\arabic*)]
    \item $a:=\sum_{i\in I} a_i$ converges in the strong operator topology;
    
    \item $\tau(xa)>0$ for all $x\in L^1(M,\tau)_+\setminus \{0\}$;
    
    \item $\eta m \leq m_i\leq \eta^{-1} m$ for all $i\in I$.
    \end{enumerate}
For $x\in L^1(M,\tau)$, we claim that
    \[
        \gamma(x):= \sum_{i\in I} \tau(xa_i) m_i
    \]
converges. Indeed, let $x=v|x|$ be the polar decomposition and let $\epsilon>0$. The strong summability of the $a_i$ implies there exists a finite $F_0\subset I$ so that whenever a finite subset $F\subset I$ satisfies $F\cap F_0 = \emptyset$ then
    \[
        \left\| \sum_{i\in F} a_i |x|^{\frac12} \right\|_2,\ \left\| \sum_{i\in F} a_i v|x|^{\frac12} \right\|_2 <\epsilon.
    \]
Let $F,G\subset I$ be finite subsets both containing $F_0$ so that $F\Delta G$ is disjoint from $F_0$. Let $w$ be the polar part of
    \[
        \sum_{i\in F} \tau(xa_i)m_i - \sum_{i\in G} \tau(xa_i) m_i.
    \]
Then using $\|m_i\|_1 =\tau(m_i)\leq \eta^{-1} \tau(m)$ we have
    \begin{align*}
        \left\| \sum_{i\in F} \tau(xa_i)m_i - \sum_{i\in G} \tau(xa_i) m_i\right\|_1 &\leq \sum_{i\in F\Delta G} |\tau(a_ix)| |\tau(w^*m_i)|\\
        &\leq \sum_{i\in F\Delta G} |\<a_i^{\frac12} v|x|^{\frac12}, a_i^{\frac12} |x|^{\frac12}\>_2|  \eta^{-1} \tau(m)\\
            &\leq  \left( \sum_{i\in F\Delta G} \|a_i^{\frac12} v|x|^{\frac12}\|_2^2 \right)^{\frac12} \left( \sum_{i\in F\Delta G} \|a_i^{\frac12} |x|^{\frac12}\|_2^2 \right)^{\frac12} \eta^{-1} \tau(m)\\
            &= \left\< \sum_{i\in F\Delta G} a_i v|x|^{\frac12},v|x|^{\frac12}\right\>_2^{\frac12} \left\< \sum_{i\in F\Delta G} a_i |x|^{\frac12}, |x|^{\frac12} \right\>_2^{\frac12} \eta^{-1} \tau(m)\\
            &< \epsilon \|x\|_1^{\frac12} \eta^{-1} \tau(m).
    \end{align*}
Thus the net of partial sums is Cauchy and converges in $L^1(M,\tau)$. We therefore have a positive linear map $\gamma$ on $L^1(M,\tau)$, which is faithful by (2). One can also show $\gamma$ is bounded using similar estimates as above:
    \begin{align*}
        |\tau(\gamma(x) b)|  &\leq   \sum_{i\in I} |\tau(xa_i)\tau(m_ib)| \leq \sum_{i\in I} |\tau(xa_i)| \eta^{-1}\tau(m)\|b\| \\
        &\leq \< a v|x|^{\frac12}, v|x|^{\frac12}\>_2^{\frac12} \< a|x|^{\frac12}, |x|^\frac12\>_2^{\frac12} \eta^{-1} \tau(m) \|b\| \leq \|a\| \|x\|_1 \eta^{-1} \tau(m)\|b\|,
    \end{align*}
for all $b\in M$. Hence $\|\gamma\| \leq \|a\| \eta^{-1} \tau(m)$.

\noindent Now, by applying $\tau$ to the inequalities in (3), one can show that
    \[
        \tau(m_i) \eta^2 \frac{m}{\tau(m)} \leq m_i \leq \tau(m_i) \eta^{-2} \frac{m}{\tau(m)}.    
    \]
Since
    \[
        \tau(\gamma(x)) = \sum_{i\in I} \tau(xa_i) \tau(m_i),
    \]
for $x\in L^1(M,\tau)_+$, it follows that for $x_0:= \frac{m}{\tau(m)}\in S$ one has
    \[
        \gamma(x) \leq \sum_{i\in I} \tau(xa_i) \tau(m_i) \eta^{-2} x_0 = \tau(\gamma(x)) \eta^{-2} x_0,
    \]
and similarly one has $\gamma(x)\geq \tau(\gamma(x))\eta^2 x_0$. This shows that $\gamma$ is a strict Hennion contraction by Theorem~\ref{thm:strict_Hennion_projective_actions}. In particular, there exists a fixed point; that is, $\gamma(x) = \tau(\gamma(x)) x$ for some $x\in L^1(M,\tau)$. Note that $m$ controls the component of $S$ in which the projective action of $\gamma$ is valued. $\hfill\blacksquare$
\end{exmp}    

The following lemma will be needed in later sections once we start considering quantum processes, which in the context of this article are compositions of positive maps.

\begin{lem}\label{lem:compositions_of_contractions}
For faithful positive linear maps $\alpha,\beta$ on $L^1(M,\tau)$, one has
    \[
        c(\alpha\circ \beta) \leq c(\alpha)c(\beta).
    \]
Consequently, the family of strict Hennion contractions is invariant under (pre or post) composition with faithful positive linear maps on $L^1(M,\tau)$.
\end{lem}
\begin{proof}
For $x\in L^1(M,\tau)_+\setminus\{0\}$ one has
        \[
            (\alpha \circ \beta)\cdot x = \frac{1}{\tau(\alpha(\beta(x)))} \alpha(\beta(x)) = \frac{1}{\tau\left(\alpha(\beta\cdot x\right))} \alpha(\beta\cdot x)=\alpha \cdot (\beta \cdot x),
        \]
so that
    \begin{align*}
        d((\alpha\circ \beta)\cdot x, (\alpha \circ \beta) \cdot y)) = d(\alpha\cdot (\beta\cdot x), \alpha\cdot (\beta\cdot y)) \leq c(\alpha) d(\beta \cdot x, \beta \cdot y)\leq c(\alpha) c(\beta)d(x,y).
    \end{align*} 
Hence $c(\alpha\circ\beta)\leq c(\alpha)c(\beta)$.
\end{proof}

\subsubsection{Strict Hennion contractions from normal maps}\label{subsubsec:SHCfromnormal}

A positive normal map $\phi\colon M\to M$ can induce bounded positive maps on $L^1(M,\tau)$ in two ways: by the predual map $\phi_*$, or by extending $\phi$ itself $L^1(M,\tau)$ provided it is $\tau$-bounded (see Lemma~\ref{lem:normalbddextension}). We will investigate when these induced maps are strict Hennion contractions, but we must first characterize when they are faithful:

\begin{lem}\label{lem:when_is_predual_faithful}
Let $\phi\colon M\to M$ be a normal completely positive map.
    \begin{enumerate}[label=(\arabic*)]
        \item $\phi_*$ is faithful if and only if $\phi(M)M$ is weak* dense in $M$.
        \item $\phi$ admits a bounded faithful extension to $L^1(M,\tau)$ if and only if $\phi_*$ is $M$-preserving and $\phi_*(M)M$ is weak* dense in $M$.
    \end{enumerate}
\end{lem}
\begin{proof}
\begin{enumerate}[label=\textbf{(\arabic*):}]
    \item The weak* closure of $\phi(M)M$ is a weak* closed right ideal in $M$ and hence of the form $pM$ for $p:=[\phi(M)M]$. Thus $\phi(M)M$ is weak* dense if and only if $p=1$.

Now, suppose $\phi_*$ is not faithful and let $x\in L^1(M,\tau)_+$ be a non-zero element in its kernel. For all $a,b\in M$ we then have
    \begin{align*}
        \|x^{\frac12} \phi(a)b\|_2^2 &= \tau(b^*\phi(a)^* x\phi(a)b) = \tau(x^{1/2} \phi (a^*)^*b^*b \phi(a^*) x^{1/2}) \\
            &\leq \|b\|^2\|\phi(1)\| \tau(x^{1/2} \phi(aa^*) x^{1/2})= \|b\|^2\|\phi(1)\|  \tau(\phi_*(x)aa^*)=0.
    \end{align*}
Thus $\overline{\phi(M)M}^{\|\cdot\|_2} \subset \ker{x^{\frac12}}$. Since $x$ is non-zero, so is $x^{\frac12}$, and therefore $\overline{\phi(M)M}^{\|\cdot\|_2}$ must be a proper subspace of $L^2(M,\tau)$. The projection onto this subspace is $p$ by definition,  so $p\neq 1$.

Conversely, if $p\neq 1$ then for all $a\in M$ we have
    \[
        \tau(\phi_*(1-p) a)= \tau( (1-p)\phi(a)) = \<(1-p) \phi(a), 1\>_2 = 0.
    \]
Since this holds for all $a\in M$, it follows that $\phi_*(1-p)=0$.

\item The equivalence of $\phi$ admitting a bounded extension to $L^1(M,\tau)$ and $\phi_*$ being $M$-preserving follows from Lemma~\ref{lem:normalbddextension}. The rest is then a consequence of part (1), with the roles of $\phi$ and $\phi_*$ reversed.\qedhere
\end{enumerate}
\end{proof}

Note that if $\phi$ is unital or even satisfies that $1$ belongs to the weak* closure of $\phi(M)$, then $\phi_*$ is faithful by the previous lemma. Another way to guarantee faithfulness of $\phi_*$ (which avoids assuming that $\phi$ is completely positive) is to assume that $\phi(1)$ is invertible. Indeed, this implies $\phi(1)\geq \delta 1$ for some $\delta>0$, which is equivalent via duality to $\tau(\phi_*(x)) \geq \delta \tau(x)$ for all $x\in L^1(M,\tau)_+$.

\begin{rmk}
In the proof of Lemma~\ref{lem:when_is_predual_faithful}.(1), complete positivity was only used for the Schwarz inequality:
    \[
        \phi(x)^*\phi(x) \leq \|\phi(1)\| \phi(x^*x) \qquad \forall x\in M.
    \]
Thus it would be sufficient to assume $\phi$ was merely $2$-positive (see \cite[Proposition 3.3]{paulsen}).$\hfill\blacksquare$
\end{rmk}

As a consequence of the following lemma, it will turn out one of $\phi_*$ or the extension of $\phi$ is a strict Hennion contraction if and only if the other is, provided both maps exist and are faithful (see Corollaries~\ref{cor:when_do_preduals_give_SHC} and \ref{cor:when_do_extensions_gives_SHC}).

\begin{lem}\label{lem:contraction_constant_of_duals}
Let $\phi\colon M\to M$ be a positive normal map with a bounded faithful extension to $L^1(M,\tau)$ and a faithful predual $\phi_*$. Then $c(\phi)=c(\phi_*)$.  
\end{lem}
\begin{proof}
By Proposition~\ref{prop:cont_const_is_diam} it suffices to show $\diam(\phi\cdot S) = \diam(\phi_*\cdot S)$, and by the formula for $d$ in terms of $m$ along with Theorem~\ref{thm:mproperties}.\ref{part:m_inf_formula} it further suffices to show
        \begin{align*}
            \inf&\left\{\frac{\tau(a\phi(x))}{\tau(a\phi(y))}\frac{\tau(b\phi(y))}{\tau(b\phi(x))}\colon x,y\in L^1(M,\tau)_+,\ a,b\in M_+,\  \tau(b\phi(x)), \tau(a\phi(y))>0 \right\}\\
                &=\inf\left\{\frac{\tau(a\phi_*(x))}{\tau(a\phi_*(y))}\frac{\tau(b\phi_*(y))}{\tau(b\phi_*(x))}\colon x,y\in L^1(M,\tau)_+,\ a,b\in M_+,\  \tau(b\phi_*(x)), \tau(a\phi_*(y))>0 \right\}.
        \end{align*}
    Since $\phi$ and $\phi_*$ are both bounded on $L^1(M,\tau)$, neither of the above infima is changed if we restrict to $x,y\in M_+$. But then
        \[
            \frac{\tau(a\phi(x))}{\tau(a\phi(y))}\frac{\tau(b\phi(y))}{\tau(b\phi(x))} = \frac{\tau(\phi_*(a)x)}{\tau(\phi_*(a)y)}\frac{\tau(\phi_*(b)y)}{\tau(\phi_*(b)x)} = \frac{\tau(x\phi_*(a))}{\tau(x\phi_*(b))}\frac{\tau(y\phi_*(b))}{\tau(y\phi_*(a))}
        \]
    implies the two infima are equal.
\end{proof} 

The next results are corollaries to Theorem~\ref{thm:strict_Hennion_projective_actions}.

\begin{cor}\label{cor:when_do_preduals_give_SHC}
Let $\phi\colon M\to M$ be a positive normal map with faithful predual $\phi_*$. Then $\phi_*$ is a strict Hennion contraction if and only if for some $x_0\in S$ there exists $\eta\in (0,1]$ so that
    \[
        \eta \tau(x_0a)\phi(1) \leq \phi(a) \leq \eta^{-1} \tau(x_0 a)\phi(1) \qquad \forall a\in M_+.
    \]
Moreover, $\phi_*\cdot S\subset S_b^\times$ if and only if one can choose $x_0=1$, and in this case $\phi$ extends to a bounded faithful map on $L^1(M,\tau)$ which is also a strict Hennion contraction.
\end{cor}
\begin{proof}
Noting that
    \[
        (a\mapsto \tau(x_0a)\phi(1))_* = x\mapsto \tau(\phi_*(x)) x_0,
    \]
we see that from Lemma~\ref{lem:adjoint&predual} that the inequalities on $\phi$ are equivalent to
    \[
        \eta \tau(\phi_*(x)) x_0 \leq \phi_*(x) \leq \eta^{-1} \tau(\phi_*(x)) x_0 \qquad \forall x\in L^1(M,\tau)_+.
    \]
These are in turn equivalent to $\phi_*$ being a strict Hennion contraction by Theorem~\ref{thm:strict_Hennion_projective_actions}. The last part of this theorem also implies the second if and only if statement.

Now, suppose $\eta \tau(a) \phi(1) \leq \phi(a) \leq \eta^{-1} \tau(a) \phi(1)$ holds for all $a\in M_+$. Taking the trace yields $\eta \tau(\phi(1)) \tau(a)\leq \tau\circ \phi(a) \leq \eta^{-1} \tau(\phi(1)) \tau(a)$. This shows $\phi$ admits a bounded faithful extension to $L^1(M,\tau)$ and it is necessarily a strict Hennion contraction by Lemma~\ref{lem:contraction_constant_of_duals}.
\end{proof}

\begin{cor}\label{cor:when_do_extensions_gives_SHC}
Let $\phi\colon M\to M$ be a positive normal map with a bounded faithful extension to $L^1(M,\tau)$. Then $\phi$ is a strict Hennion contraction if and only if for some $b_0\in S_b$ there exists $\eta\in (0,1]$ so that
    \[
        \eta \tau(\phi(a)) b_0 \leq \phi(a) \leq \eta^{-1} \tau(\phi(a)) b_0 \qquad \forall a\in M_+.
    \]
Moreover, $\phi\cdot S\subset S_b^\times$ if and only if one can choose $b_0=1$, and in this case $\phi_*$ is a bounded faithful map on $L^1(M,\tau)$ which is also a strict Hennion contraction.
\end{cor}
\begin{proof}
The first if and only if statement follows immediately from Theorem~\ref{thm:strict_Hennion_projective_actions}, and the last part of this theorem also implies the second if and only if statement.

Now, suppose $\eta\tau(\phi(a))\leq \phi(a)\leq \eta^{-1}\tau(\phi(a))$ holds for all $a\in M_+$. Invoking Lemma~\ref{lem:adjoint&predual} we obtain
    \[
        \eta \tau(x)\phi_*(1) \leq \phi_*(x) \leq  \eta^{-1} \tau(x) \phi_*(1) \qquad \forall x\in L^1(M,\tau)_+.
    \]
This implies $\phi_*$ is faithful, and hence is a strict Hennion contraction by Lemma~\ref{lem:contraction_constant_of_duals}. 
\end{proof}

\begin{cor}\label{cor:tau_interval_gives_SHC}
Let $\phi\colon M\to M$ be a positive normal map. Then the following are equivalent.
    \begin{enumerate}[label=(\roman*)]
        \item There exists $\eta\in (0,1]$ so that $\eta \tau\leq \phi \leq \eta^{-1} \tau$.
        
        \item There exists $\delta\in (0,1]$ so that $\delta\tau\leq \tau\circ \phi\leq \delta^{-1} \tau$, and $\phi$ has a bounded faithful extension to $L^1(M,\tau)$ which is a strict Hennion contraction valued in $S_b^\times$.
        
        \item There exists $\delta_*\in (0,1]$ so that $\delta_*\tau \leq \tau\circ \phi_*\leq \delta_*^{-1} \tau$, and $\phi_*$ is a faithful map on $L^1(M,\tau)$ which is a strict Hennion contraction valued in $S_b^\times$.
    \end{enumerate}
\end{cor}
\begin{proof}
$(i)\Rightarrow (ii),(iii):$ By Lemma~\ref{lem:adjoint&predual}, $\eta\tau \leq \phi\leq \eta^{-1}\tau$ is equivalent to $\eta\tau \leq \phi_*\leq \eta^{-1} \tau$, since $\tau_*=\tau$. Applying $\tau$ to these inequalities gives $\eta\tau\leq \tau\circ \phi\leq \eta^{-1} \tau$, so that $\phi$ extends to a bounded faithful map on $L^1(M,\tau)$, and $\eta\tau\leq \tau\circ \phi_*\le \eta^{-1} \tau$, so that $\phi_*$ is faithful. Since
    \[
        \eta^2 \tau(\phi(a)) \leq \eta \tau(a) \leq \phi(a) \leq \eta^{-1} \tau(a)\leq \eta^{-2} \tau(\phi(a)),
    \]
for all $a\in M_+$, Corollary~\ref{cor:when_do_extensions_gives_SHC} implies $\phi$ and $\phi_*$ are strict Hennion contractions. They are both necessarily valued in $S_b^\times$ since $\eta \leq \phi(1), \phi_*(1) \leq \eta^{-1}$.\\

\noindent$(ii)\Rightarrow(i):$ Theorem~\ref{thm:strict_Hennion_projective_actions} yields a $\kappa\in (0,1]$ satisfying $\kappa \tau\circ \phi \leq \phi \leq \kappa^{-1}\tau\circ \phi$, and therefore
    \[
        \kappa \delta \tau \leq \kappa \tau\circ \phi \leq \phi \leq \kappa^{-1} \tau\circ \phi \leq \kappa^{-1} \delta^{-1} \tau.
    \]
So we take $\eta:=\kappa\delta$.\\

\noindent$(iii)\Rightarrow (i):$ Since $\phi_*\cdot 1\in M$, we have that $\phi_*|_M\colon M\to M$ is a positive normal map (see \cite[Proposition 2.5.11]{ap}). So the same argument as in $(ii)\Rightarrow(i)$ gives $\eta \tau\leq \phi_* \leq \eta^{-1} \tau$ for some $\eta\in (0,1]$, and appealing to Lemma~\ref{lem:adjoint&predual} completes the proof.
\end{proof}

\begin{thm}\label{thm:converseconditions}
Let $\phi\colon M\to M$ be a normal completely positive map such that $\tau\circ\phi\leq \tau$ ($\phi$ is subtracial), $\phi(1)\leq 1$ ($\phi$ is subunital), and $1-\phi(1)$ belongs to the weak* closure of $\phi(M_+)$. If $\phi_*$ is a strict Hennion contraction then the unital tracial map
    \[
        \tilde{\phi}_*(x):= \phi_*(x) + \frac{(\tau-\tau\circ \phi_*)(x)}{(\tau- \tau\circ \phi_*)(1)}(1-\phi_*(1))
    \]
is a strict Hennion contraction valued in $S_b^\times$. In this case, the extension of
    \[
        \tilde{\phi}(x):= \phi(x) + \frac{(\tau-\tau\circ \phi)(x)}{(\tau- \tau\circ \phi)(1)}(1-\phi(1))
    \]
to $L^1(M,\tau)$ is also a strict Hennion contraction valued in $S_b^\times$.
\end{thm}
\begin{proof}
We first observe that the assumption on $\phi$ implies $1=(1-\phi(1)) + \phi(1)$ belongs to the weak* closure of $\phi(M)M$ so that $\phi_*$ is faithful by Lemma~\ref{lem:when_is_predual_faithful}. We also claim that for all $x,y\in L^1(M,\tau)_+\setminus\{0\}$ one has
    \begin{align*}
        m(\phi_*(x), \phi_*(y))(\tau- \tau\circ\phi_*)(y) \leq (\tau - \tau\circ\phi_*)(x).
    \end{align*}
Indeed, if $(\tau- \tau\circ\phi_*)(y)=0$ then this holds trivially, so suppose $(\tau- \tau\circ \phi_*)(y)>0$. If we let $(a_i)_{i\in I}\subset M_+$ be a net such that $\phi(a_i)\to 1-\phi(1)$ weak*, then
    \[
        \frac{ (\tau - \tau\circ \phi_*)(x)}{(\tau - \tau\circ \phi_*)(y)} = \frac{\tau(x(1-\phi(1))}{\tau(y(1-\phi(1))} = \lim_{i\to\infty} \frac{\tau(x \phi(a_i))}{\tau(y\phi(a_i))} = \lim_{i\to\infty} \frac{\tau(\phi_*(x) a_i)}{\tau(\phi_*(y) a_i)} \geq m(\phi_*(x), \phi_*(y)),
    \]
where we have used Theorem~\ref{thm:mproperties}.\ref{part:m_inf_formula} in the last inequality. Using this we have
    \[
        m(\phi_*(x),\phi_*(y)) \tilde{\phi}_*(y) = m(\phi_*(x),\phi_*(y)) \phi_*(y) +  \frac{m(\phi_*(x),\phi_*(y))(\tau-\tau\circ \phi_*)(y)}{(\tau- \tau\circ \phi_*)(1)}(1-\phi_*(1)) \leq \tilde{\phi}_*(x).
    \]
Thus $m(\phi_*(x),\phi_*(y)) \leq m(\tilde{\phi}_*(x), \tilde{\phi}_*(y))$, and consequently
    \[
        m(\tilde{\phi}_*(x), \tilde{\phi}_*(y)) m(\tilde{\phi}_*(y), \tilde{\phi}_*(x)) \geq m(\phi_*(x), \phi_*(y)) m(\phi_*(y), \phi_*(x)) = m(\phi_*\cdot x,\phi_*\cdot y) m(\phi_*\cdot y, \phi_*\cdot x),
    \]
where the equality is due to Theorem~\ref{thm:mproperties}.\ref{part:mscaling}. From this we obtain $d(\tilde{\phi}_*(x), \tilde{\phi}_*(y))\leq d(\phi_*\cdot x, \phi_*\cdot y)$ for all $x,y\in S$, and hence
    \[
        \diam(\tilde{\phi}_*(S))\leq \diam(\phi_*\cdot S).
    \]
Thus if $\phi_*$ is a strict Hennion contraction, then Proposition~\ref{prop:cont_const_is_diam} implies $\tilde{\phi}_*$ is as well, and in fact we have $\tilde{\phi}_*(S) \subset S_b^\times$ because it is unital. Since $\tilde{\phi}_*$ is the predual map of $\tilde{\phi}$, Corollary~\ref{cor:tau_interval_gives_SHC} gives that the extension of $\tilde{\phi}$ is also a strict Hennion contraction valued in $S_b^\times$.
\end{proof}

Recall \cite{farenik} that a positive normal map $\phi\colon M\to M$ is said to be \emph{reducible} if there is a nontrivial projection $p\in M$ and a constant $\lambda>0$ so that $\phi(p)\le \lambda p$. This is equivalent \cite[Proposition 1]{farenik} to the statement that $\phi(pMp) \subset pMp$. If there is no such nontrivial projection, $\phi$ is said to be \emph{irreducible}. 

\begin{cor}\label{cor:irred}
If $\gamma$ is a bounded strict Hennion contraction valued in $S_b^\times$, then $\gamma|_M$ is irreducible.
\end{cor}
\begin{proof}
Note that $\gamma(M)\subset M$ by virtue of $\gamma\cdot S\subset S_b^\times$. Thus if $\gamma$ is bounded as a map on $L^1(M,\tau)$ then $\gamma|_M\colon M\to M$ is normal. Theorem~\ref{thm:strict_Hennion_projective_actions} then gives a $\kappa\in (0,1]$ so that $\gamma(p)\geq \kappa \tau(\gamma(p))1$ for all projections $p\in M$. Thus if $\gamma(p)\leq \lambda p$ for some $\lambda>0$, then necessarily $p=1$.
\end{proof}

\section{Ergodic Quantum Processes}\label{sec:EQP}

\begin{define}
Let $(M,\tau)$ be a tracial von Neumann algebra, let $(\Omega,\P)$ be a probability space equipped with $T\in \Aut(\Omega,\P)$, and let $\gamma_\omega \colon L^1(M,\tau)\to L^1(M,\tau)$ be a bounded random linear operator. We call the family of bounded random linear operators
            \[
                \Gamma_{n,m}^T(\omega) = \gamma_{T^n\omega}\circ \gamma_{T^{n-1}\omega}\circ \cdots \circ \gamma_{T^m \omega} \qquad\qquad n,m\in \mathbb{Z},\ n\geq m,
            \]
an (\textbf{interval}) \textbf{quantum process} on $L^1(M,\tau)$. If $T$ is ergodic, then we call the above an \textbf{ergodic} quantum process on $L^1(M,\tau)$. When no confusion can arise, we will suppress the superscript $T$.
\end{define}

In order to avoid measurability issues, we will from now on assume our von Neumann algebras all have separable predual. In Section~\ref{subsec:convergence_properties_of_EQP} we establish a number of convergence properties for ergodic quantum processes under the assumptions that $\gamma_\omega$ is bounded, positive, and faithful almost surely and that with positive probability $\Gamma_{n,m}$ is eventually a strict Hennion contraction. These results are infinite dimensional generalizations of \cite[Theorems 1 and 2]{jeff}. Of course, here one lacks the reflexivity $M_n(\C)_*\cong M_n(\C)$ present in the finite dimensional case, so in Section~\ref{subsec:right_EQP_from_normal_maps} we consider quantum processes on $M$.

\subsection{Contraction constant asymptotics}

This first lemma addresses some technical aspects of measurability. Recall that in Section~\ref{sec:cmap} we rarely required boundedness of the positive maps on $L^1(M,\tau)$. However, it will be essential in this section in order to apply Proposition~\ref{prop:contr_const_for_bounded_maps} and leverage our assumption that $L^1(M,\tau)$ is separable.

\begin{lem}\label{lem:mismeasurable}
Let $(M,\tau)$ be a tracial von Neumann algebra with a separable predual, let $(\Omega, \P)$ be a probability space, and let $\gamma_\omega:L^1(M,\tau) \to L^1(M,\tau)$ be a  random linear operator that is positive and faithful almost surely. For any $x,y\in S$ one has $m(\gamma_\omega\cdot x, \gamma_\omega \cdot y)\in L^\infty(\Omega, \P)$. Furthermore, if $\gamma_\omega$ is bounded almost surely then $c(\gamma_\omega)\in L^\infty(\Omega,\P)$.
\end{lem}
\begin{proof}
    Since $L^1(M,\tau)$ is separable, there is a countable $\sigma$-WOT dense subset $\{a_n\}_{n\in \mathbb{N}}\subset M_+$ by Theorem~\ref{thm:sep}. By setting $a_n \mapsto \frac{\tau(\gamma(y))}{\tau(\gamma(x))} \frac{\tau(a_n\gamma(x))}{\tau(a_n \gamma(y))}$ equal to $+\infty$ whenever $\tau(\gamma_\omega(x))=0$ or $\tau(a_n\gamma(y))=0$, we see that \[m(\gamma \cdot x, \gamma \cdot y) = \frac{\tau(\gamma(y))}{\tau(\gamma(x))} \inf\left\{\frac{\tau(a\gamma(x))}{\tau(a \gamma(y))}: a\in M_+ \setminus \{0\}\right\} = \frac{\tau(\gamma(y))}{\tau(\gamma(x))} \inf_{n\ge 1} \frac{\tau(a_n\gamma(x))}{\tau(a_n \gamma(y))}.\] We have therefore expressed $m(\gamma \cdot x, \gamma \cdot y)$ as the infimum of a sequence of random variables, hence it is also a random variable.

    Now suppose $\gamma_\omega$ is bounded almost surely. Fix a countable $\|\cdot\|_1$-dense subset $S_0\subset S$. Then,
        \[
            c(\gamma_\omega) = \diam(\gamma_\omega \cdot S_0)= \sup_{x,y\in S_0} d(\gamma_\omega\cdot x, \gamma_\omega\cdot y)
        \]
    almost surely by Proposition~\ref{prop:contr_const_for_bounded_maps}. The first part of the proof implies $d(\gamma_\omega\cdot x, \gamma_\omega\cdot y)$ is measurable for each $x,y\in S_0$, and hence the right-most expression above is measurable as the supremum of a countable set of measurable functions. Since the contraction constant is always bounded, we get $c(\gamma_\omega)\in L^\infty(\Omega,\P)$.
\end{proof}

\begin{rmk}\label{rem:almost_surely_bounded_is_bounded_by_measurable}
For $\gamma_\omega$ as in Lemma~\ref{lem:mismeasurable}, when $\gamma_\omega$ is bounded almost surely the separability of $L^1(M,\tau)$ also implies that $\omega\mapsto \|\gamma_\omega\|$ is measurable. Indeed, let $A_0$ and $X_0$ be countable subsets of the unit balls of $M$ and $L^1(M,\tau)$, respectively, that are dense in the $\sigma$-weak operator and $\|\cdot\|_1$ topologies, respectively. Then,
    \[
        \|\gamma_\omega\| = \sup_{\substack{a\in A_0\\ x\in X_0}} |\tau(a \gamma_\omega(x))|,  
    \]
almost surely, and thus $\|\gamma_\omega\|$ is measurable since each $\tau(a \gamma_\omega(x))$ is measurable by assumption. It follows that 
    \[
        \P[\|\gamma_\omega\|<\infty \text{ and } \|\gamma_\omega(x)\|_1\leq \|\gamma_\omega\| \|x\|_1\ \forall x\in L^1(M,\tau)]=1.
    \]
In this case, one says that $\gamma_\omega$ is a \emph{bounded} random linear operator (see \cite[Definition 2.24]{Bharucha-Reid}). $\hfill\blacksquare$
\end{rmk}

In light of the above remark, we henceforth adopt the convention of saying a random linear operator $\gamma\colon \Omega\times L^1(M,\tau)\to L^1(M,\tau)$ has a property associated to linear maps on $L^1(M,\tau)$ if for almost every $\omega\in \Omega$ the map $L^1(M,\tau)\ni x\mapsto \gamma_\omega(x)$ has the corresponding property (e.g. bounded, positive, completely positive, faithful etc.). Provided the list of properties is countable, the event that $\gamma_\omega$ has all of the properties still occurs with probability one.

Recall from Definition~\ref{def:Hennion_contraction} that $SHC(M)$ denotes the set of all strict Hennion contractions on $M$. The following lemma analyzes a standard hypothesis in our convergence results. It tells us that as long as the event $[\Gamma_{n,m}\in SHC(M)]$ occurs with positive probability for \emph{some} $n\geq m$, then for \emph{any} $m\in\mathbb{Z}$ the sequence $\Gamma_{m,m}, \Gamma_{m+1,m},\ldots$ will almost surely land in $SHC(M)$ (and by Lemma~\ref{lem:compositions_of_contractions} remain there forever after), and likewise for the sequence $\Gamma_{n,n}, \Gamma_{n,n-1},\ldots$ for any $n\in \mathbb{Z}$. This is comparable to \cite[Assumption 1]{jeff} by Example~\ref{exmp:strictly_positive} (see also \cite[Lemma 2.1]{jeff}).

\begin{lem}\label{lem:a_not_so_mild_hypothesis}
Let $(M,\tau)$ be a tracial von Neumann algebra with a separable predual, let $(\Omega, \P)$ be a probability space equipped with ergodic $T \in \Aut(\Omega,\P)$, let $\gamma_\omega:L^1(M, \tau) \to L^1(M,\tau)$ be a bounded positive faithful random linear operator, and let $\Gamma_{n,m}$ be the associated interval quantum process. Suppose 
    \[
        \P[\exists n,m\in \mathbb{Z} \text{ such that } n\geq m \text{ and } \Gamma_{n,m}\in SHC(M)]>0.
    \]
Then
    \[
        \P[\forall m\in \mathbb{Z}\ \exists n\geq m \text{ such that } \Gamma_{n,m}\in SHC(M)]=\P[\forall n\in \mathbb{Z}\ \exists m\leq n \text{ such that } \Gamma_{n,m}\in SHC(M)] =1.
    \]
\end{lem}
\begin{proof}
Since
    \[
        0< \P[\exists n,m\in \mathbb{Z} \text{ such that } n\geq m \text{ and } \Gamma_{n,m}\in SHC(M)] \leq \sum_{m\in \mathbb{Z}} \P[\exists n\geq m \text{ such that } \Gamma_{n,m}\in SHC(M)],
    \]
it follows that
    \[
        \P[\exists n\geq m_0 \text{ such that } \Gamma_{n,m_0}\in SHC(M)]>0
    \]
for some $m_0\in \mathbb{Z}$. Observe that $\Gamma_{n,m_0}(\omega) = \Gamma_{n+m-m_0, m}(T^{m_0-m}\omega)$ for each $m\in \mathbb{Z}$ and consequently
    \begin{align}\label{eqn:T-invariant_tails}
        [\exists n\geq m \text{ such that } \Gamma_{n,m}\in SHC(M)] = T^{m_0-m}[\exists n\geq m_0 \text{ such that } \Gamma_{n,m_0}\in SHC(M)].
    \end{align}
Also note that for any $m\geq  m'$ one has
    \[
        [\exists n\geq m \text{ such that } \Gamma_{n,m}\in SHC(M)] \subset [\exists n\geq m' \text{ such that } \Gamma_{n,m'}\in SHC(M)]
    \]
by Lemma~\ref{lem:compositions_of_contractions}. In particular, we have
    \begin{align*}
        T^{-1}[\exists n\geq m_0 \text{ such that } \Gamma_{n,m_0}\in SHC(M)] &= [\exists n\geq m_0+1 \text{ such that } \Gamma_{n,m_0+1}\in SHC(M)]\\
            &\subset [\exists n\geq m_0 \text{ such that } \Gamma_{n,m_0}\in SHC(M)].
    \end{align*}
Thus this event occurs with probability one by the ergodicity of $T$. Since $T$ is measure preserving, this along with Equation~(\ref{eqn:T-invariant_tails}) yields
    \[
        \P[\exists n\geq m \text{ such that } \Gamma_{n,m}\in SHC(M)] = \P[\exists n\geq m_0 \text{ such that } \Gamma_{n,m_0}\in SHC(M)] =1
    \]
for all $m\in \mathbb{Z}$. Finally, it follows from continuity from above that
    \[
        \P[\forall m\in \mathbb{Z}\ \exists n\geq m \text{ such that } \Gamma_{n,m}\in SHC(M)] = 1.\qedhere
    \]
\end{proof}

\begin{rmk}
If one removes the ergodicity assumption in Lemma~\ref{lem:a_not_so_mild_hypothesis}, the same proof shows
    \[
        \P[\exists n\geq m \text{ such that } \Gamma_{n,m}\in SHC(M)] = \P[\exists n\geq m_0 \text{ such that } \Gamma_{n,m_0}\in SHC(M)]>0
    \]
for all $m\in \mathbb{Z}$, and consequently
    \[
        \P[\forall m\in \mathbb{Z}\ \exists n\geq m \text{ such that } \Gamma_{n,m}\in SHC(M)]  = \P[\exists n\geq m_0 \text{ such that } \Gamma_{n,m_0}\in SHC(M)]>0.
    \]
Similarly, $\P[\forall n\in \mathbb{Z}\ \exists m\leq n \text{ such that } \Gamma_{n,m}\in SHC(M)]>0$.$\hfill\blacksquare$
\end{rmk}

These processes induce a natural filtration on the probability space. We will see that one can define a stopping time with respect to the following filtration:

\begin{note}
Let $(M,\tau)$ be a tracial von Neumann algebra with a separable predual, let $(\Omega, \P)$ be a probability space equipped with ergodic $T \in \Aut(\Omega,\P)$, let $\gamma_\omega:L^1(M, \tau) \to L^1(M,\tau)$ be a bounded positive faithful random linear operator, and let $\Gamma_{n,m}$ be the associated interval quantum process.
    \begin{enumerate}[label = (\alph*)]
        \item Fix an integer $m\in \mathbb{Z}$. For $n\ge m$ let
        \[
            \mathcal{F}^+(n,m) := \sigma( c(\Gamma^T_{n,m}), c(\Gamma_{n-1,m}^T), \dots, c(\Gamma_{m,m}^T)),
        \] 
        so that the family $\mathcal{F}^+(\bullet,m):=(\mathcal{F}^+(n,m))_{n\in  [m,\infty)}$ is the natural filtration with respect to the stochastic process $c(\Gamma_{\bullet,m}):=(c(\Gamma_{n,m}))_{n\in  [m,\infty)}$.
    
        \item Similarly, we will denote by $\mathcal{F}^-(n,\bullet):=(\mathcal{F}^-(n,m))_{m\in (-\infty, n]}$ the natural filtration associated with $c(\Gamma_{n,\bullet}):=(c(\Gamma_{n,m}))_{m\in (-\infty,n]}$ given by 
        \[
            \mathcal{F}^-(n,m) := \sigma( c(\Gamma^T_{n,n}), c(\Gamma_{n,n-1}^T), \dots, c(\Gamma_{n,m}^T)).
        \]
    \end{enumerate}
\end{note}

The following lemma is similar to \cite[Lemma 3.11]{jeff} and uses many of the ideas present in its proof as well as the proof of \cite[Lemma 2.1]{jeff}.

\begin{lem}\label{lem:randomconstproperties}
Let $(M,\tau)$ be a tracial von Neumann algebra with a separable predual, let $(\Omega, \P)$ be a probability space equipped with ergodic $T \in \Aut(\Omega,\P)$, let $\gamma_\omega:L^1(M, \tau) \to L^1(M,\tau)$ be a bounded positive faithful random linear operator, and let $\Gamma_{n,m}$ be the associated interval ergodic quantum process. Suppose that
    \[
        \P [\exists n,m\in \mathbb{Z}\colon n\geq m \text{ and }\Gamma_{n,m}\in SHC(M)] >0.    
    \]
    \begin{enumerate}[label=(\arabic*)]
    \item There exists a constant $C\in [0,1)$ so that for all $m\in \mathbb{Z}$
        \[
            \lim_{n\to +\infty} c(\Gamma_{n,m})^{\frac{1}{n-m+1}} = C
        \]
    almost surely, and for all $n\in \mathbb{Z}$
        \[
            \lim_{m\to-\infty} c(\Gamma_{n,m})^{\frac{1}{n-m+1}} = C
        \]
    almost surely.

    \item For $\kappa\in (C,1)$ and each $k\in \mathbb{Z}$, there exist finite almost surely random variables $D_{\bullet, k}$ and $D_{k,\bullet}$ satisfying
        \[
            D_{\bullet,k}(T^{\ell} \omega) = D_{\bullet, k+\ell}(\omega) \qquad \text{ and } \qquad D_{k,\bullet}(T^\ell \omega) = D_{k+\ell,\bullet}(\omega) 
        \]
    for all $\ell\in \mathbb{Z}$, and
        \begin{align*}
            c(\Gamma_{n,k}) &\leq D_{\bullet, k}\kappa^{n-k+1}\\
            c(\Gamma_{k,m)} &\leq D_{k,\bullet} \kappa^{k-m+1},
        \end{align*}
    almost surely for all $n\geq k\geq m$. In particular, $c(\Gamma_{n,k}), c(\Gamma_{k,m})\to 0$ exponentially fast almost surely as $n\to +\infty$ and $m\to -\infty$.

    \item For each $k\in \mathbb{Z}$, with respect to $\mathcal{F}^+(\bullet,k)$ (resp. $\mathcal{F}^{-}(k,\bullet)$) the random variable $\nu(\Gamma_{\bullet,k}):= \inf\{ n\geq k \colon \Gamma_{n,k}\in SHC(M)\}$ (resp. $\nu(\Gamma_{k,\bullet}):= \inf\{ m\leq k \colon \Gamma_{k,m}\in SHC(M)\}$) is a stopping time that is finite almost surely.
    \end{enumerate}
\end{lem}
\begin{proof}

\begin{enumerate}[label = \textbf{(\arabic*):}]
\item Fix $m \in \mathbb{Z}$ and let $X_k:=\log[c(\Gamma_{m+k-1,m})]$ for $k\in \mathbb{N}$. By Lemma~\ref{lem:compositions_of_contractions}, 
    \[
        X_{k+\ell}(\omega) \leq X_k(\omega)+ \log [ c(\Gamma_{m+k+\ell-1,m+k}(\omega)] =X_k(\omega) + X_{\ell}(T^k\omega),
    \] 
holds almost surely. Additionally, $X_k^{+}\equiv 0$ since $c(\Gamma_{m+k-1,m})\le 1$.  Thus, by writing $n=m+k-1$, an application of Theorem~\ref{thm:kingman} tells us that the sequence $(c(\Gamma_{n,m})^{\frac{1}{n-m+1}})_{n\colon n\geq m}$ converges almost surely to the constant
    \[
        C_m:= \exp\left(\inf_{n : n\geq m} \frac{1}{n-m+1} \mathbb{E}[\log(c(\Gamma_{n,m}))] \right).
    \]
Now, Lemma~\ref{lem:a_not_so_mild_hypothesis} tells us that
    \[
       \bigcup_{n\geq m} [c(\Gamma_{n,m})<1] = [\exists n\geq m \text{ such that } \Gamma_{n,m}\in SHC(M)],
    \]
occurs with probability one. Therefore there exists $n_1\geq m$ so that $\mathbb{P}[c(\Gamma_{n_1,m})<1]>0$. Consequently,
    \[
        \log(C_m) = \inf_{n\geq m} \frac{1}{n-m+1} \mathbb{E}[\log(c(\Gamma_{n,m}))] \leq \frac{1}{n_1-m+1}\mathbb{E}[\log(c(\Gamma_{n_1,m}))]<0,
    \]
and so $C_m<1$. Also, since $T$ is measure preserving for any $m,m'\in \mathbb{Z}$ we have
    \begin{align*}
        C_m &= \inf_{k\geq 0} \frac{1}{k+1} \E[\log(c(\Gamma_{m+k,m}))]\\
            &= \inf_{k\geq 0} \frac{1}{k+1} \E[\log(c(\Gamma_{m+k,m}\circ T^{m'-m}))] \\
            &= \inf_{k\geq 0} \frac{1}{k +1}\E[\log(c(\Gamma_{m'+k,m'}))] = C_{m'}.
    \end{align*}
So we set $C:= C_0$.

Now, using $Y_k:= \log[c(\Gamma_{n,n-k-1})]$, $k\in \N$ and the fact that $T^{-1}$ is also ergodic, the same argument as above yields another constant $C'\in [0,1)$
such that for all $n\in \mathbb{Z}$
    \[
        C'= \exp\left(\inf_{m: m\leq n}\frac{1}{n-m+1} \E[\log(c(\Gamma_{n,m}))]\right)
    \]
and
    \[
        \lim_{m\to-\infty} c(\Gamma_{n,m})^{\frac{1}{n-m+1}} = C'
    \]
almost surely. Given $\epsilon>0$, let $n\geq 1$ be large enough so that
    \[
        \exp\left(\frac{1}{n} \E[\log(c(\Gamma_{n,1}))]\right) \leq C+\epsilon.
    \]
Then
    \[
        C' = \exp\left(\inf_{m\colon m\leq n} \frac{1}{n-m+1} \E[\log(c(\Gamma_{n,m}))]\right) \leq \exp\left(\frac{1}{n} \E[\log(c(\Gamma_{n,1}))]\right) \leq C+ \epsilon.
    \]
Letting $\epsilon\to 0$ yields $C'\leq C$, and reversing the roles of $n$ and $m$ gives $C'=C$.

\item Fix $\kappa\in (C,1)$ and $k\in \mathbb{Z}$ and define
    \[
        D_{\bullet,k}:= 1 \vee \sup_{n: n\geq k} \frac{c(\Gamma_{n,k})}{\kappa^{n-k+1}} \qquad \text{ and } \qquad D_{k,\bullet}:= 1 \vee \sup_{m: m\leq k} \frac{c(\Gamma_{k,m})}{\kappa^{k-m+1}},
    \]
which are random variables by Lemma~\ref{lem:mismeasurable}. By the previous part, for almost every $\omega\in \Omega$
    \[
        \lim_{n\to+\infty} c( \Gamma_{n,k})^{\frac{1}{n-k+1}} = \lim_{m\to -\infty} c(\Gamma_{k,m})^{\frac{1}{k-m+1}}=C< \kappa.
    \]
Consequently there is an $\ell_0\geq 0$ (depending on $\kappa$, $k$, and $\omega$) so that
    \[
        \frac{c(\Gamma_{n,k})}{\kappa^{n-k+1}}, \frac{c(\Gamma_{k,m})}{\kappa^{k-m+1}} \leq 1
    \]
for all $n> k+\ell_0$ and $m < k- \ell_0$, which implies
    \begin{align*}
        D_{\bullet,k}(\omega) &= 1 \vee \max_{n: k \leq n \leq k+\ell_0} \frac{c(\Gamma_{n,k}(\omega))} {\kappa^{n-k+1}} \leq \frac{1}{\kappa^{\ell_0+1}} <\infty,\\
        D_{k,\bullet}(\omega) &= 1 \vee \max_{m: k-\ell_0\leq m \leq k} \frac{c(\Gamma_{k,m}(\omega))}{\kappa^{k-m+1}} \leq \frac{1}{\kappa^{\ell_0+1}}<\infty.
    \end{align*}
Hence $D_{\bullet,k}$ and $D_{k,\bullet}$ are finite almost surely, and the remaining properties follow from their definition.

\item This is an immediate consequence of Lemma~\ref{lem:submultiplicativeStoppingtime}.\qedhere
\end{enumerate}
\end{proof}

\subsection{Ergodic quantum processes from normal maps}\label{subsec:right_EQP_from_normal_maps}

As noted above, in order to emulate the proofs of \cite{jeff} it is necessary to consider the duals of quantum processes, which according to Lemma~\ref{lem:normalbddextension} should correspond to \emph{random} normal positive linear maps on $M$. In fact, from the perspective of von Neumann algebras, such maps may be even more natural than those on $L^1(M,\tau)$. The goal of this section is to formalize this notion and relate it to random linear operators on $L^1(M,\tau)$.

\begin{define}\label{def:wstarRLO}
Let $(M,\tau)$ be a tracial von Neumann algebra and $(\Omega, \P)$ be a probability space. A \textbf{weak* random variable} is a function $f\colon \Omega\to M$ such that $\tau(f(\omega)x)$ is measurable for all $x\in L^1(M,\tau)$. We say that a mapping $\phi\colon \Omega\times M\to M$ is a \textbf{weak* random linear operator} if $\omega\mapsto \phi_\omega(a)$ is a weak* random variable for all $a\in M$ and $\P[\phi(\alpha a + b)= \alpha \phi(a) + \phi(b)]=1$ for all $a,b\in M$ and $\alpha\in \C$.
\end{define}

As we did with random linear operators on $L^1(M,\tau)$, we adopt the convention of saying a weak* random linear operator $\phi\colon \Omega\times M\to M$ has a property associated to linear maps on $M$ if for almost every $\omega\in \Omega$ the map $M\ni a\mapsto \phi_\omega(a)$ has the corresponding property (e.g. normal, positive, completely positive, $\tau$-bounded). 

We first establish the random version of the correspondence in Lemma~\ref{lem:normalbddextension} between normal linear maps on $M$ and bounded linear maps on $L^1(M,\tau)$.

\begin{lem}\label{lem:measurabilityTest}
Let $(M,\tau)$ be a tracial von Neumann algebra with separable predual, and let $(\Omega,\P)$ be a probability space. Up to almost sure equality, there is a one-to-one correspondence between normal positive (resp. completely positive) weak* random linear operators $\phi_\omega$ on $M$ and bounded positive (resp. completely positive) random linear maps $(\phi_\omega)_*$ on $L^1(M,\tau)$ determined by
    \begin{align}\label{eqn:random_tau_relation}
        \tau((\phi_\omega)_*(x) a)=\tau(x\phi_\omega(a))  \qquad\qquad \omega\in \Omega,\ x\in L^1(M,\tau),\ a\in M.
    \end{align}
This correspondence restricts to a one-to-one correspondence between $\tau$-bounded normal positive (resp. completely positive) weak* random linear operators on $M$ and $M$-preserving bounded positive (resp. completely positive) random linear operators on $L^1(M,\tau)$. The former maps $\phi_\omega$ also admit unique extensions $\phi_\omega|^{L^1(M,\tau)}$ to $L^1(M,\tau)$ that are $M$-preserving bounded and positive (resp. completely positive), and the latter maps $(\phi_\omega)_*$ also have restrictions $(\phi_\omega)_*|_M$ to $M$ that are $\tau$-bounded normal and positive (resp. completely positive). In this case, one has $((\phi_\omega)_*|_M)_*= \phi_\omega|^{L^1(M,\tau)}$.
\end{lem}
\begin{proof}
By Lemma~\ref{lem:normalbddextension}, it suffices to check measurability. For the first correspondence, this is a consequence of Equation~(\ref{eqn:random_tau_relation}). (Note that we are invoking the separability of $L^1(M,\tau)$ here to reduce checking that $(\phi_\omega)_*(x)$ is a random variable to checking that it is weakly measurable.) For the measurability of extensions and restrictions, note that if $(b_n)_{n\in \N}\subset M$ converges to $x\in L^1(M,\tau)$ in $\|\cdot\|_1$-norm, then for any $a\in M$ one has
    \[
        \tau(\phi_\omega(x) a) = \lim_{n\to\infty} \tau(\phi_\omega(b_n) a) = \lim_{n\to\infty} \tau(b_n (\phi_\omega)_*(a)) = \tau(x (\phi_\omega)_*(a))
    \]
almost surely. Each $\tau(\phi_\omega(b_n) a)$ is measurable, from which it follows that the extension $\phi_\omega|^{L^1(M,\tau)}$ is a random linear operator and the restriction $(\phi_\omega)_*|_M$ is a weak* random linear operator.
\end{proof}

\begin{rmk}
Note that if $c_\omega$ denotes the norm of $\phi_\omega$ on $L^1(M,\tau)$ whenever the extension exists, then $\omega\to c_\omega$ is a random variable by Remark~\ref{rem:almost_surely_bounded_is_bounded_by_measurable}. Thus when $L^1(M,\tau)$ is separable, $\phi_\omega$ being $\tau$-bounded almost surely is equivalent to there existing a (finite almost surely) random variable $c\colon \Omega\to [0,\infty]$ so that $\mathbb{P}[\tau\circ \phi_\omega(x^*x) \leq c_\omega \tau(x^*x)\ \forall x\in M]=1$. $\hfill\blacksquare$
\end{rmk}

Of course, the first correspondence in Lemma~\ref{lem:measurabilityTest} is still true without any positivity assumptions, and consequently compositions of normal weak* random linear operators on $M$ give normal weak* random linear operators. Indeed, the predual maps are bounded random linear operators on $L^1(M,\tau)$ whose composition (in the reverse order) is also a bounded random linear operator by Lemma~\ref{thm:RLOcomp}. The dual of this then gives a weak* random linear operator that is almost surely equal to the composition of the original weak* random linear operators.

\begin{define}
Let $(M,\tau)$ be a tracial von Neumann algebra, let $(\Omega,\P)$ be a probability space equipped with $T\in \Aut(\Omega,\P)$, and let $\phi_\omega \colon M\to M$ be a normal weak* random linear operator. We call the family of normal weak* random linear operators
    \[
        \Phi_{n,m}^T(\omega):= \phi_{T^n\omega}\circ \cdots \circ \phi_{T^m\omega} \qquad\qquad n,m\in \mathbb{Z},\ n\geq m
    \]
an (\textbf{interval}) \textbf{quantum process} on $M$. If $T$ is ergodic, then we call the above an \textbf{ergodic} quantum process on $M$. When no confusion can arise, we will supress the superscript $T$.
\end{define}

The most common example of a quantum process on $M$ that we shall consider is the following. Suppose $\Gamma_{n,m}^T$ is a quantum process on $L^1(M,\tau)$ associated to some bounded positive random linear operator $\gamma_\omega$ and $T\in \Aut(\Omega,\P)$. Denote $\phi_\omega:=\gamma_\omega^*$, which is a normal positive weak* random linear operator on $M$ by Lemma~\ref{lem:measurabilityTest}. Then quantum process $\Phi_{n,m}^{T^{-1}}$ on $M$ associated to $\phi_\omega$ and $T^{-1}$ satisfies
    \begin{align}\label{eqn:dual_quantum_process}
    \begin{split}
        (\Phi_{n,m}^{T^{-1}}(\omega))_* &= \Gamma_{-m,-n}^{T}(\omega),\\
        (\Gamma_{n,m}^T(\omega))^* &= \Phi_{-m,-n}^{T^{-1}}(\omega),
    \end{split}
    \end{align}
for all $n\geq m$ and $\omega\in \Omega$.

\subsection{Convergence properties}\label{subsec:convergence_properties_of_EQP}

We now prove the first main result of this section. This is the analogue of \cite[Lemma 3.14]{jeff}, which gives roughly half the proof of \cite[Theorem 1]{jeff} (see Remark~\ref{rmk:recovering_jeff_results} below). The key observation is that $\Gamma_{n,m-1}\cdot S \subset \Gamma_{n,m}\cdot S$ and that the diameter of these sets tends to zero almost surely as $m\to -\infty$ by Lemma~\ref{lem:randomconstproperties}. This is more or less the same proof as in the finite dimensional case, where one can also treat the limit $n\to +\infty$ by considering the dual process $\Gamma_{n,m}^*(\omega)$. This limit is analyzed in \cite[Lemma 3.12]{jeff}, which forms the other half of the proof of \cite[Theorem 1]{jeff}. However, in the infinite dimensional setting $\Gamma_{n,m}^*$ is a process on $M$ rather than $L^1(M,\tau)$ and therefore requires a separate argument that we present as Theorem~\ref{thm:right_convergence_on_M} (see also the proof of Corollary~\ref{cor:analogue_of_jeff2}).

\begin{thm}[{Theorem~\ref{thmx:A}}]\label{thm:right_convergence}
Let $(M,\tau)$ be a tracial von Neumann algebra with a separable predual, let $(\Omega, \P)$ be a probability space equipped with ergodic $T\in \Aut(\Omega, \P)$, let $\gamma_\omega\colon L^1(M,\tau)\to L^1(M,\tau)$ be a bounded positive faithful random linear operator, and let $\Gamma_{n,m}$ be the associated interval ergodic quantum process. Suppose that
    \[
        \P [\exists n,m\in \mathbb{Z}\colon n\geq m \text{ and }\Gamma_{n,m}\in SHC(M)] >0.    
    \]
Then there exists a family of random variables $X_n\colon \Omega\to S$, $n\in \mathbb{Z}$, satisfying
    \begin{align}\label{eqn:right_process_limit_properties}
        \gamma_{T^{n+1}\omega}\cdot X_n(\omega) = X_{n+1}(\omega) \qquad \text{ and } \qquad X_n(T^{\pm1}\omega) = X_{n\pm 1}(\omega)
    \end{align}
almost surely, and for all $x\in S$
    \[
            \lim_{m\to -\infty} \|\Gamma_{n,m}\cdot x - X_n\|_1 =0
    \]
almost surely for all $n\in \mathbb{Z}$.
\end{thm}
\begin{proof}
For each $\omega\in \Omega$ define the family of (random) sets $S_{n,m}(\omega):= \Gamma_{n,m}(\omega)\cdot S$, $m\leq n$, and observe that $S_{n,m-1}\subset S_{n,m}$ by construction. By Proposition~\ref{prop:cont_const_is_diam} and Lemma~\ref{lem:randomconstproperties}, we know that $\diam(S_{n,m})= c(\Gamma_{n,m}) \to 0$ almost surely as $m\to - \infty$. Moreover, since $S$ is complete by Therorem~\ref{thm:tracemetricinequality_and_completeness}, we may invoke the Cantor intersection theorem to conclude 
    \[
        \bigcap_{m\colon m\leq n} S_{n,m}(\omega)
    \]
consists of a single element $X_n(\omega)$ for almost every $\omega$. Note that the relations in Equation~(\ref{eqn:right_process_limit_properties}) follow from the relations
    \[
        \gamma_{T^{n+ 1}\omega} \cdot S_{n,m}(\omega) = S_{n+ 1,m}(\omega) \qquad { and } \qquad S_{n,m}(T^{\pm 1}\omega) = S_{n\pm 1,m\pm 1}(\omega)
    \]
for each $\omega\in \Omega$.

Now, fix $n\in \mathbb{Z}$ and $x\in S$. By Theorem~\ref{thm:tracemetricinequality_and_completeness} we have
    \[
        \|\Gamma_{n,m}\cdot x - X_n\|_1 \leq 2 d(\Gamma_{n,m}\cdot x, X_n) \leq 2 \diam(S_{n,m}) = 2 c(\Gamma_{n,m}),
    \]
which tends to zero almost surely as $m\to-\infty$ by Lemma~\ref{lem:randomconstproperties}. Consequently, for all $a\in M$ we have
    \[
        \tau(X_n a) = \lim_{m\to -\infty} \tau((\Gamma_{n,m}\cdot x) a),
    \]
almost surely so that $X_n$ is weakly measurable, and hence a random variable by Theorem~\ref{thm:Pettis}.
\end{proof}

\begin{exmp}
Let $(M,\tau)$ be a tracial von Neumann algebra with separable predual and let $(\Omega,\mathbb{P})$ be a locally compact Hausdorff space with a Radon probability measure. Denote $N:=M\bar\otimes L^\infty(\Omega,\mathbb{P})$ and $\varphi:=\tau\otimes \int_\Omega\cdot\ d\mathbb{P}$. By \cite[Theorem IV.7.17]{TakesakiI}  $L^1(N,\varphi)$ can be identified with functions $f\colon \Omega\to L^1(M,\tau)$ such that $\omega\mapsto f_\omega(a)$ is measurable for all $a\in M$ and
    \[
        \int_{\Omega} \|f_\omega\|_1\ d\mathbb{P}(\omega) <\infty
    \]
(see also \cite[Propostions IV.7.2 and IV.7.4]{TakesakiI}). In particular, if $f\in L^1(N,\varphi)_+$ then $f_\omega\in L^1(M,\tau)_+$ almost surely.  Fix $f\in L^1(N,\varphi)_+$ which is non-zero almost surely and $\eta\in (0,1]$. Let $\{(a_i, f^{(i)})\in M_+\times L^1(N,\varphi)_+\colon i\in I\}$ be a countable family satisfying:
    \begin{enumerate}[label=(\arabic*)]
    \item $a:= \sum_{i\in I} a_i$ converges in the strong operator topology;
    
    \item $\tau(xa)>0$ for all $x\in L^1(M,\tau)_+\setminus\{0\}$;
    
    \item $0<f^{(i)} \leq \eta^{-1} f$ almost surely for all $i\in I$;

    \item $\P[\forall i\in I\ \eta f \leq f^{(i)}]>0$.
    \end{enumerate}
Then by Example~\ref{exmp:strongly_summable},
    \[
        \gamma(x):= \sum_{i\in I} \varphi(x(a_i\otimes 1)) f^{(i)}
    \]
defines a bounded positive faithful linear map on $L^1(N,\varphi)$ satisfying $\gamma\cdot x\leq \kappa^{-1} f^{(0)}$, where $f^{(0)} = \frac{f}{\varphi(f)}$ and $\kappa = \eta^2$. Consequently, 
    \[
        \gamma_\omega(x):=\gamma(x\otimes 1) = \sum_{i\in I} \tau(x a_i) f^{(i)}_\omega
    \]
defines a bounded random linear operator on $L^1(M,\tau)$. Moreover, $\gamma$ is positive and faithful almost surely. Indeed, recalling that $I$ is countable we see that $f^{(i)} \in L^1(M,\tau)_+$ almost surely implies $\gamma(x)$ is positive almost surely for $x\in L^1(M,\tau)_+$. Next (3) implies $f^{(i)}$ is non-zero almost surely so that
    \[
        \mathbb{P}[ \exists i\in I \text{ such that } f_\omega^{(i)}=0]=0.
    \]
Thus for $x\in L^1(M,\tau)_+\setminus\{0\}$
    \[
        \mathbb{P}[\gamma(x)=0]=\mathbb{P}[\forall i\in I\  \tau(xa_i)=0]=\mathbb{P}[\tau(xa) =0]=0
    \]
by (2). We also note that $\kappa f^{(0)} \leq \gamma\cdot x\leq \kappa^{-1} f^{(0)}$ holds with positive probability by (3) and (4), which tells us that $\gamma_\omega$ is a strict Hennion contraction with positive probability by Theorem~\ref{thm:strict_Hennion_projective_actions}. 

Now, let $T\in \text{Aut}(\Omega,\mathbb{P})$ be an ergodic automorphism. Then the associated interval ergodic quantum process is given by
    \[
        [\Gamma_{n,m}(\omega)](x) = \sum_{i_m,\ldots, i_n\in I} \tau(xa_{i_m})\tau(f_{T^m\omega}^{(i_m)}a_{i_{m+1}}) \cdots  \tau(f_{T^{n-1}\omega}^{(i_{n-1})} a_{i_n}) f_{T^n\omega}^{(i_n)}, 
    \]
and satisfies
    \[
        \P[\exists n,m\in \mathbb{Z}\colon n\geq m \text{ and } \Gamma_{n,m}\in SHC(M)] >0
    \]
since we noted above that $\Gamma_{0,0}=\gamma$ is a strict Hennion contraction with positive probability. Theorem~\ref{thm:right_convergence} therefore yields a family of random variables $F_n\colon \Omega \to S$ (which we can identify with elements of $L^1(N,\varphi)_+$) so that 
    \[
        \lim_{m\to-\infty} \frac{1}{\tau([\Gamma_{n,m}(\omega)](x)} \sum_{i_m,\ldots, i_n\in I} \tau(xa_{i_m})\tau(f_{T^m\omega}^{(i_m)}a_{i_{m+1}}) \cdots  \tau(f_{T^{n-1}\omega}^{(i_{n-1})} a_{i_n}) f_{T^n\omega}^{(i_n)} = F_n(\omega)
    \]
in $\|\cdot\|_1$-norm almost surely. Additionally, one has
    \[
        F_{n+1}(\omega) = \gamma_{T^{n+1}\omega}\cdot F_n(\omega) = \frac{1}{\tau(\gamma_{T^{n+1}\omega} (F_n(\omega)))} \sum_{i\in I} \tau(F_n(\omega) a_i) f^{(i)}_\omega
    \]
almost surely by Equation~(\ref{eqn:right_process_limit_properties}).$\hfill\blacksquare$
\end{exmp}

The following uses the second part of Lemma~\ref{lem:measurabilityTest} to extend a $\tau$-bounded weak* random linear operator on $M$ to a bounded random linear operator on $L^1(M,\tau)$.

\begin{thm}\label{thm:right_convergence_on_M}
Let $(M,\tau)$ be a tracial von Neumann algebra with a separable predual, let $(\Omega, \P)$ be a probability space equipped with ergodic $T\in \Aut(\Omega, \P)$, let $\phi_\omega\colon M\to M$ be a $\tau$-bounded normal positive weak* random linear operator. Suppose that the extension of $\phi_\omega$ to $L^1(M,\tau)$ is faithful almost surely and that the associated ergodic quantum process $\Phi_{n,m}$ on $L^1(M,\tau)$ satisfies
    \[
        \P [\exists n,m\in \mathbb{Z}\colon n\geq m \text{ and }\Phi_{n,m}\in SHC(M)] >0.    
    \]
Then there exists a family of random variables $A_n\colon \Omega\to S_b$, $n\in \mathbb{Z}$, satisfying
    \[
        \gamma_{T^{n+ 1}\omega}\cdot A_n(\omega) = A_{n+1}(\omega) \qquad \text{ and }\qquad A_n(T^{\pm 1}\omega)= A_{n\pm 1}(\omega)
    \]
almost surely, and for all $x\in S$
    \[
            \lim_{m\to -\infty} \left\|\Phi_{n,m}\cdot x  - A_n \right\|_1 =0
    \]
almost surely for all $n\in \mathbb{Z}$. Moreover, the above convergence holds in $\|\cdot\|_\infty$-norm for $x\in S_b$.
\end{thm}
\begin{proof}
Applying Theorem~\ref{thm:right_convergence} gives us this family of random variables $A_n$, $n\in \mathbb{Z}$, though we must argue they are almost surely valued in $S_b$. Using Lemma~\ref{lem:randomconstproperties} for almost every $\omega$, $c(\Phi_{n,m}) <1$ for sufficiently small $m$ (depending on $\omega$). When this occurs, $\Phi_{n,m}(M)\subset M$ implies $\Phi_{n,m}\cdot S \subset S_b$ by Theorem~\ref{thm:geometry}. Recalling from the proof of Theorem~\ref{thm:right_convergence} that $A_n\in \bigcap_{m\leq n} \Phi_{n,m}\cdot S$, we thus have $A_n\in S_b$ almost surely. The final statement then follows from Lemma~\ref{thm:cauchytocauchy}.(1):  $\Phi_{n,m}\cdot S_b\subset S_b$ since $\phi_\omega$ is $M$-preserving and the proof of Theorem~\ref{thm:right_convergence} in fact shows $d(\Phi_{n,m}\cdot x, A_n)\to 0$ as $m\to -\infty$. 
\end{proof}

\begin{rmk}\label{rmk:recovering_jeff_results}
Theorems~\ref{thm:right_convergence} and \ref{thm:right_convergence_on_M} can be used to recover \cite[Lemmas 3.12 and 3.14]{jeff}, respectively, which together yields \cite[Theorem 1]{jeff}. Indeed, the hypotheses of Theorems~\ref{thm:right_convergence} and \ref{thm:right_convergence_on_M} follow for $\Gamma_{n,m}$ and $\Gamma_{n,m}^*$, respectively, from \cite[Assumption 1 and Lemma 2.1]{jeff} (or are automatic in the finite dimensional case). Consequently, any fixed points of the projective actions of $\Gamma_{n,m}$ and $\Gamma_{n,m}^*$ converge to $X_n$ as $m \to -\infty$ and $A_m$ as $n\to+\infty$, respectively. In fact, our results are slightly more general than those of \cite{jeff} because our hypotheses allow $X_n$ and $A_m$ to be valued outside of $S_b^\times$. $\hfill\blacksquare$
\end{rmk}

As a corollary to Theorems~\ref{thm:right_convergence} and \ref{thm:right_convergence_on_M}, we also obtain an analogue of \cite[Theorem 2]{jeff}. The discrepancy between that result and the one below are due to the inequivalence of the $\|\cdot\|_1$-norm and $\|\cdot\|_\infty$-norm in the general case, but it is clear how to reconcile the difference in the finite dimensional case.

\begin{cor}\label{cor:analogue_of_jeff2}
Let $(M,\tau)$ be a tracial von Neumann algebra with a separable predual, let $(\Omega, \P)$ be a probability space equipped with ergodic $T\in \Aut(\Omega, \P)$, let $\gamma_\omega\colon L^1(M,\tau)\to L^1(M,\tau)$ be an $M$-preserving bounded positive faithful random linear operator, and let $\Gamma_{n,m}$ be the associated interval ergodic quantum process. Suppose that the extension of $\gamma_\omega^*$ to $L^1(M,\tau)$ is faithful almost surely and that
    \[
        \P [\exists n,m\in \mathbb{Z}\colon n\geq m \text{ and }\Gamma_{n,m}\in SHC(M)] >0.    
    \]
Fix $k\in \mathbb{Z}$ and let $C$ be as in Lemma~\ref{lem:randomconstproperties}. Then for any $\kappa\in (C,1)$ and $n\geq k> m$, there exists random variables $X_n, B_m\colon \Omega\to S_b$ and $E_{ k}\colon\Omega\to [0,\infty)$ such that for all $a\in M$
    \[
        \left\|\frac{1}{\tau(\Gamma_{n,m}(1))} \Gamma_{n,m}(a) - \tau(B_m a) X_n \right\|_1 \leq E_{k} \kappa^{n-m+1}\|a\|_\infty
    \]
almost surely.
\end{cor}
\begin{proof}
Applying Theorem~\ref{thm:right_convergence} gives the random variables $X_n$, which we note are valued in $S_b$ almost surely since $\gamma_\omega$ is $M$-preserving (see the proof of Theorem~\ref{thm:right_convergence_on_M}). Also recall from the proof of Theorem~\ref{thm:right_convergence} that we actually have
    \begin{align}\label{eqn:first_factor_bound}
        \|\Gamma_{n,m}\cdot x - X_n \|_1 \leq 2 c(\Gamma_{n,m})
    \end{align}
for all $x\in S$ almost surely. 

Next, denote $\phi_\omega:=\gamma_\omega^*$, which by Lemma~\ref{lem:measurabilityTest} is a $\tau$-bounded normal positive weak* random linear operator. It also almost surely has a faithful extension to $L^1(M,\tau)$ be assumption. Denote by $\Phi_{n,m}^{T^{-1}}$ the ergodic quantum processes associated to $\phi_\omega$ and $T^{-1}$ so that one has $(\Gamma_{n,m}^T)^* = \Phi_{-m,-n}^{T^{-1}}$. Using Lemma~\ref{lem:contraction_constant_of_duals}, it follows that
    \[
        \P [\exists n,m\in \mathbb{Z}\colon n\geq m \text{ and }\Phi_{n,m}^{T^{-1}}\in SHC(M)] = \P [\exists n,m\in \mathbb{Z}\colon n\geq m \text{ and }\Gamma_{n,m}^{T}\in SHC(M)]>0.
    \]
Thus we can apply Theorem~\ref{thm:right_convergence_on_M} to obtain random variables $A_n\colon \Omega\to S_b$ satisfying
    \[
        \| \Phi_{n,m}\cdot x - A_n \|_1 \leq 2 c(\Phi_{n,m}) = 2 c(\Gamma_{-m,-n}).
    \]
for all $x\in S$ almost surely. If we denote $B_m:=A_{-m}$, then by the above we have
    \begin{align}\label{eqn:second_factor_bound}
        | \tau([\Gamma_{n,m}^*\cdot 1 - B_m]a) | \leq 2 c(\Gamma_{m,n}) \|a\|_\infty
    \end{align}
for $a\in M$ almost surely.

Now, observe that
    \[
        \frac{1}{\tau(\Gamma_{n,m}(1))} \Gamma_{n,m}(a) = \frac{\tau(\Gamma_{n,m}(a))}{\tau(\Gamma_{n,m}(1))} \Gamma_{n,m}\cdot a = \tau( [\Gamma_{n,m}^*\cdot 1]a) \Gamma_{n,m}\cdot a.
    \]
Thus combining Estimates~(\ref{eqn:first_factor_bound}) and (\ref{eqn:second_factor_bound}), for $a\in S_b$ we have
    \begin{align*}
        \left\| \frac{1}{\tau(\Gamma_{n,m}(1))} \Gamma_{n,m}(a) - \tau(B_m a) X_n \right\|_1 &\leq  |\tau([\Gamma_{n,m}^*\cdot 1]a)| \| \Gamma_{n,m}\cdot a - X_n\|_1 + | \tau([\Gamma_{n,m}^*\cdot 1 - B_m]a) | \|X_n\|_1\\
            &\leq  4 c(\Gamma_{n,m}) \|a\|_\infty.
    \end{align*}
Applying the above estimate to arbitrary $a\in M$ by decomposing it into a linear combination of four positive elements and scaling gives an upper bound of $16 c(\Gamma_{n,m}) \|a\|_\infty$. For $n\geq k> m$, using Lemmas~\ref{lem:compositions_of_contractions}  and \ref{lem:randomconstproperties}, we can further bound this by 
    \[
        16 c(\Gamma_{n,k}) c(\Gamma_{k-1,m}) \|a\| \leq 16 D_{\bullet, k} \kappa^{n-k+1} D_{k-1,\bullet} \kappa^{k -m} \|a\|_\infty = 16 D_{\bullet, k} D_{k-1,\bullet} \kappa^{n-m+1}\|a\|_\infty.
    \]
So Taking $E_k(\omega):= 16 D_{\bullet, k}(\omega) D_{k-1,\bullet}(\omega) $ completes the proof.
\end{proof}

We next prove the second main result of the section, which is essentially a dual version of Theorem~\ref{thm:right_convergence} for ergodic quantum processes on $M$. Interestingly, this result and the next do \emph{not} have analogues in \cite{jeff}.

\begin{thm}[{Theorem~\ref{thmx:B}}]\label{thm:left_convergence_on_M}
Let $(M,\tau)$ be a tracial von Neumann algebra with a separable predual, let $(\Omega, \P)$ be a probability space equipped with ergodic $T\in \Aut(\Omega, \P)$, let $\phi_\omega\colon M\to M$ be a normal positive weak* random linear operator, and let $\Phi_{n,m}$ be the associated ergodic quantum process. Suppose that $\phi_\omega(1)$ is invertible almost surely and
    \[
        \P [\exists n,m\in \mathbb{Z}\colon n\geq m \text{ and }\eta \tau_{x_0} \leq \Phi_{m,n} \leq \eta^{-1} \tau_{x_0} \text{ for some }\eta\in (0,1],\ x_0\in S] >0.    
    \]
Then there exists a family of random variables $Y_m\colon \Omega\to S$, $m\in \mathbb{Z}$, satisfying
    \[
        (\phi_{T^{m- 1}\omega})_*\cdot Y_m(\omega) = Y_{m-1}(\omega) \qquad \text{ and } \qquad Y_m(T^{\pm 1}\omega)= Y_{m\pm 1}(\omega)
    \]
almost surely, and for all $a\in M$
    \[
            \lim_{n\to\infty } \left\|\Phi_{n,m}(1)^{-\frac12}\Phi_{n,m}(a)\Phi_{n,m}(1)^{-\frac12} - \tau(a Y_m) \right\|_\infty =0
    \]
almost surely for all $m\in \mathbb{Z}$.
\end{thm}
\begin{proof}
Denote $\gamma_\omega:= (\phi_\omega)_*$, which is a bounded positive random linear operator on $L^1(M,\tau)$ by Lemma~\ref{lem:measurabilityTest}. Moreover, $\gamma_\omega$ is faithful almost surely by the discussion following Lemma~\ref{lem:when_is_predual_faithful}. Letting $\Gamma_{n,m}^{T^{-1}}$ denote the associated ergodic quantum process on $L^1(M,\tau)$, we have
    \[
        (\Phi_{n,m}^T(\omega))_* = \Gamma_{-m,-n}^{T^{-1}}(\omega),
    \]
by Equation~(\ref{eqn:dual_quantum_process}). Lemma~\ref{cor:when_do_preduals_give_SHC} implies $\Gamma_{-m,-n}\in SHC(M)$ if and only if there exists $x_0\in S$ and $\eta\in (0,1]$ so that  $\eta \tau(x_0 a)\Phi_{m,n}(1) \leq \Phi_{m,n}(a) \leq \eta^{-1} \tau(x_0 a) \Phi_{m,n}(1)$ for all $a\in M_+$. Since $\Phi_{n,m}(1)$ is invertible almost surely by assumption, the latter is almost surely equivalent to $\eta \tau_{x_0} \leq \Phi_{n,m} \leq \eta^{-1} \tau_{x_0}$ after adjusting $\eta$ as necessary. Thus, altogether our hypotheses imply
    \[
        \P[\exists n,m\in \mathbb{Z}\colon n\geq m \text{ and } \Gamma_{n,m}^{T^{-1}}\in SHC(M)] >0.
    \]
Let $X_n\colon \Omega\to S$, $n\in \mathbb{Z}$, be the family of random variables obtained by applying Theorem~\ref{thm:right_convergence} to $\Gamma_{n,m}^{T^{-1}}$. Set $Y_m:= X_{-m}$ for each $m\in \mathbb{Z}$ so that the claimed relations follow from Equation~(\ref{eqn:right_process_limit_properties}) with $T$ replaced by $T^{-1}$. 

Now, fix $a\in M$ and $x\in S$ and denote $y:= \Phi_{n,m}(1)^{-\frac12} x \Phi_{n,m}(1)^{-\frac12}\in L^1(M,\tau)_+$. We have
    \begin{align*}
        \tau\left( \left[\Phi_{n,m}(1)^{-\frac12}\Phi_{n,m}(a)\Phi_{n,m}(1)^{-\frac12} - \tau(a Y_m) \right] x \right) &= \tau( a [\Gamma_{-m,-n}(y) - Y_m])\\
            &= \tau( a [ \Gamma_{-m,-n}\cdot y - X_{-m}]).
    \end{align*}
where in the last equality we have used $\tau(\Gamma_{-m,-n}(y)) = \tau(\Phi_{n,m}(1) y) = \tau(x)=1$. Denote $y_0:= \frac{y}{\tau(y)}$, which satisfies $y_0\in S$ and $\Gamma_{-m,-n}\cdot y_0 = \Gamma_{-m,-n}\cdot y$. We can therefore use the above computation to obtain the following estimate:
    \[
        \left| \tau\left( \left[\Phi_{n,m}(1)^{-\frac12}\Phi_{n,m}(a)\Phi_{n,m}(1)^{-\frac12} - \tau(a Y_m) \right] x \right)\right| \leq \|a\| \| \Gamma_{-m,-n}\cdot y_0 - X_{-m} \|_1.
    \]
Recall from the proof of Theorem~\ref{thm:right_convergence} that the second factor in the last expression is bounded almost surely by $2 c(\Gamma_{-m,-n})$, and the $\omega$ for which this fails does not depend on $y_0$. Consequently, decomposing an arbtitrary $x\in L^1(M,\tau)$ into a linear combination of four positive elements and scaling gives
    \[
        \left| \tau\left( \left[\Phi_{n,m}(1)^{-\frac12}\Phi_{n,m}(a)\Phi_{n,m}(1)^{-\frac12} - \tau(a Y_m) \right] x \right)\right| \leq \|a\| 4 \|x\|_1 2 c(\Gamma_{-m,-n}).
    \]
almost surely. Therefore
    \[
        \| \Phi_{n,m}(1)^{-\frac12}\Phi_{n,m}(a)\Phi_{n,m}(1)^{-\frac12} - \tau(a Y_m)\| \leq 8 \|a\| c(\Gamma_{-m,-n})
    \]
almost surely by the duality $M\cong L^1(M,\tau)^*$, and as $n\to \infty$ the above tends to zero almost surely by Lemma~\ref{lem:randomconstproperties}.
\end{proof}

Finally, we conclude with a dual version of Theorem~\ref{thm:right_convergence_on_M}. 

\begin{thm}\label{thm:left_convergence}
Let $(M,\tau)$ be a tracial von Neumann algebra with a separable predual, let $(\Omega, \P)$ be a probability space equipped with ergodic $T\in \Aut(\Omega, \P)$, let $\gamma_\omega\colon L^1(M,\tau) \to L^1(M,\tau)$ be an $M$-preserving bounded positive faithful random linear operator, and let $\Gamma_{n,m}$ be the associated ergodic quantum process. Suppose that $\gamma_\omega(1)$ is boundedly invertible almost surely and
    \[
        \P [\exists n,m\in \mathbb{Z}\colon n\geq m \text{ and }\Gamma_{n,m}\in SHC(M)] >0.    
    \]
Then there exists a family of random variables $B_m\colon \Omega\to S_b$, $m\in \mathbb{Z}$, satisfying:
    \[
        \gamma_{T^{m-1}\omega}\cdot B_m(\omega)= B_{m-1}(\omega) \qquad \text{ and } \qquad B_m(T^{\pm 1}\omega)= B_{m\pm 1}(\omega)
    \]
almost surely, and for all $a\in M$
    \[
        \lim_{n\to +\infty} \left\| \Gamma_{n,m}(1)^{-\frac12} \Gamma_{n,m}(a) \Gamma_{n,m}(1)^{-\frac12} - \tau(a B_m) \right\|_\infty =0
    \]
almost surely for all $m\in \mathbb{Z}$. Furthermore, whenever $m\in \mathbb{Z}$ satisfies $\sup_n \| \Gamma_{n,m}^*(\Gamma_{n,m}(1)^{-1})\|_\infty <\infty$ almost surely, then for all $x\in S$
    \[
        \lim_{n\to +\infty} \left\| \Gamma_{n,m}(1)^{-\frac12} \Gamma_{n,m}(x) \Gamma_{n,m}(1)^{-\frac12} - \tau(x B_m) \right\|_1 =0
    \]
almost surely.
\end{thm}
\begin{proof}
Denote $\phi_\omega:= (\gamma_\omega)^*$, which is a normal positive weak* random linear operator on $M$ by Lemma~\ref{lem:measurabilityTest}. Moreover, the assumptions on $\gamma_\omega(1)$ imply $\phi_\omega$ is $\tau$-bounded with a faithful extension to $L^1(M,\tau)$ almost surely. Letting $\Phi_{n,m}^{T^{-1}}$ denote the associated ergodic quantum process on $M$, we have that
    \[
        (\Gamma_{n,m}^{T}(\omega))^* =\Phi_{-m,-n}^{T^{-1}}(\omega).
    \]
by Equation~(\ref{eqn:dual_quantum_process}). Lemma~\ref{lem:contraction_constant_of_duals} implies $c(\Gamma_{n,m}^{T})=c(\Phi_{-m,-n}^{T^{-1}})$ and hence
    \[
        \P [\exists n,m\in \mathbb{Z}\colon n\geq m \text{ and }\Phi_{n,m}^{T^{-1}}\in SHC(M)]=\P [\exists n,m\in \mathbb{Z}\colon n\geq m \text{ and }\Gamma_{n,m}^{T}\in SHC(M)] >0.
    \]
Let $A_n\colon \Omega\to S$, $n\in \mathbb{Z}$, be the family of random variables obtained by applying Theorem~\ref{thm:right_convergence_on_M} to $\Phi_{n,m}^{T^{-1}}$. Set $B_m:=A_{-m}$ for each $m\in \mathbb{Z}$. Arguing exactly as in the proof of Theorem~\ref{thm:left_convergence_on_M}, for $a\in M$ we have
    \[
        \lim_{n\to +\infty} \left\| \Gamma_{n,m}(1)^{-\frac12} \Gamma_{n,m}(a) \Gamma_{n,m}(1)^{-\frac12} - \tau(a B_m) \right\|_\infty =0
    \]
almost surely for all $m\in \mathbb{Z}$. 

Finally, suppose $R:=\sup_n \| \Gamma_{n,m}^*(\Gamma_{n,m}(1)^{-1})\|_\infty <\infty$ almost surely. Then for $x\in L^1(M,\tau)$ and $a\in M$ we have
    \begin{align*}
        |\tau\left( \Gamma_{n,m}(1)^{-\frac12} \Gamma_{n,m}(x) \Gamma_{n,m}(1)^{-\frac12} a\right)| &= |\tau (x \Gamma_{n,m}^* \left(\Gamma_{n,m}(1)^{-\frac12}a \Gamma_{n,m}(1)^{-\frac12} \right) )|\\
            &\leq \|x\|_1 \left\|\Gamma_{n,m}^* (\Gamma_{n,m}(1)^{-\frac12}a \Gamma_{n,m}(1)^{-\frac12}) ) \right\|_\infty.
    \end{align*}
Since $a\mapsto \Gamma_{n,m}^* (\Gamma_{n,m}(1)^{-\frac12}a \Gamma_{n,m}(1)^{-\frac12}) )$ is a positive linear map on $M$, its norm is given by $\| \Gamma_{n,m}^*(\Gamma_{n,m}(1)^{-1})\|_\infty \leq R$. Thus, the last expression above is further bounded by $R\|x\|_1\|a\|_\infty$, implying that 
    \[
        \left\|\Gamma_{n,m}(1)^{-\frac12} \Gamma_{n,m}(x) \Gamma_{n,m}(1)^{-\frac12}\right\|_1 \leq R \|x\|_1.
    \]    
Now, fix $x\in S$, and given $\epsilon>0$ let $a\in M_+$ be such that $\|x-a\|_1<\epsilon$. Then using that the $\|\cdot\|_1$-norm is dominated by the $\|\cdot\|_\infty$-norm we have
    \begin{align*}
        \limsup_{n\to\infty} &\| \Gamma_{n,m}(1)^{-\frac12} \Gamma_{n,m}(x) \Gamma_{n,m}(1)^{-\frac12} - \tau(x B_m) \|_1 \\
            &\leq \limsup_{n\to\infty} \| \Gamma_{n,m}(1)^{-\frac12} \Gamma_{n,m}(x-a) \Gamma_{n,m}(1)^{-\frac12}\|_1 +  |\tau((a-x) B_m)| \leq \epsilon (R + \|B_m\|_\infty)
    \end{align*}    
almost surely. Hence the limit is zero almost surely.
\end{proof}

Note that the hypotheses in the previous theorem strong are enough that both Theorems~\ref{thm:right_convergence} and Corollary~\ref{cor:analogue_of_jeff2} can be applied. In fact, the identity
    \begin{align*}
        \frac{1}{\tau(\Gamma_{n,m}(1))}& \Gamma_{n,m}(a) - \tau(a B_m)X_n\\
            &= (\Gamma_{n,m}\cdot 1)^{\frac12} \left[\Gamma_{n,m}(1)^{-\frac12} \Gamma_{n,m}(a) \Gamma_{n,m}(1)^{-\frac12} - \tau(a B_m) \right](\Gamma_{n,m}\cdot 1)^{\frac12} + [\Gamma_{n,m}\cdot 1 - X_n] \tau(aB_m)
    \end{align*}
offers an alternative proof of Corollary~\ref{cor:analogue_of_jeff2} in this case.

\begin{rmk}
In the finite dimensional case of $(\M_N, \frac{1}{N} \text{Tr})$ the condition $\sup_n \| \Gamma_{n,m}^*(\Gamma_{n,m}(1)^{-1})\|_\infty <\infty$ is always satisfied since
    \begin{align*}
        \| \Gamma_{n,m}^*(\Gamma_{n,m}(1)^{-1})\|_\infty \leq \text{Tr}(\Gamma_{n,m}^*(\Gamma_{n,m}(1)^{-1}))  = N (\frac1N \text{Tr})\left(\Gamma_{n,m}(1) \Gamma_{n,m}(1)^{-1}\right) = N.
    \end{align*}
That is, the condition is automatic because the $\|\cdot\|_1$-norm and $\|\cdot\|_\infty$-norm are equivalent. Another assumption that guarantees the condition (even in the infinite dimensional case) is that $\gamma_\omega$ is unital. Indeed, then $\Gamma_{n,m}^*$ is tracial and
    \[
        \|\Gamma_{n,m}^*(\Gamma_{n,m}(1)^{-1})\|_\infty \leq \|\Gamma_{n,m}^*\cdot 1 - B_m\|_\infty + \|B_m\|_\infty\to 0,
    \]
as $n\to\infty$ almost surely by the final statement in Theorem~\ref{thm:right_convergence_on_M}. $\hfill\blacksquare$
\end{rmk}

\section{Application to Locally Normal States}\label{sec:FCS}

In  \cite{FannesNachtergaeleWerner}, Fannes, Nachtergaele, and Werner characterize translation invariant states on spin chains and establish clustering properties of finitely correlated states for local observables. In particular, this characterization offers a way to construct translate invariant states. Taking inspiration from this, we find a wide class of random variables $\Psi_\omega$ taking values in the locally normal states of a spin chain that satisfy a \emph{translation covariance} condition. Moreover, these states also exhibit clustering properties for local observables, and through averaging can yield deterministic translation invariant states.

\subsection{Spin chains and  their quasi-local algebras and locally normal states}\label{sec:appendix:SpinChainsLNS}

 The theory of spin chains arises as a class of quantum mechanical models from quantum statistical mechanics \cite{BratteliRobinson1,BratteliRobinson2}. We shall recall some basic facts about spin chains formed from tracial von Neumann algebras and translation invariant states. 

Let $(M,\tau)$ be a fixed tracial von Neumann algebra.  Consider an isomorphic copy of $M$ for each \textbf{site} $n\in \mathbb{Z}$, written $(M_n, \tau_n)$. These algebras represent the observable algebras of some physical quantity localized to $n$. For each finite subset $\Lambda\subset \mathbb{Z}$ we denote the von Neumann algebraic tensor product
    \[
        M_{\Lambda} := \overline{\bigotimes_{n\in \Lambda}} M_n,
    \]
which is equipped with the tensor product trace $\tau_{\Lambda}$. Set inclusions in $\mathbb{Z}$ naturally induce inclusions in the corresponding von Neumann algebras so that we may consider the inductive limit \emph{C*-algebra}
    \[
        \A_{\mathbb{Z}}:= \varinjlim M_{\Lambda},
    \]
which we call the \textbf{quasi-local algebra} associated to the spin chain with on-site algebras $M_n$. Note that this algebra can be faithfully represented in the infinite von Neumann algebra tensor product $(M_\mathbb{Z}, \tau_{\mathbb{Z}}):=\overline{\bigotimes}_{n\in \mathbb{Z}} (M_n, \tau_n)$, and consequently $\A_\mathbb{Z}$ admits a faithful tracial state $\tau_\mathbb{Z}|_{\A_\mathbb{Z}}$ (see \cite[Chapter XIV]{TakesakiIII}). Identifying $M_{\Lambda}\subset \A_{\mathbb{Z}}$ for each finite subset $\Lambda\subset\mathbb{Z}$, we call the unital $*$-subalgebra
    \[
        \A_{\mathbb{Z}}^{\text{loc}}:=\bigcup_{\Lambda \subset \mathbb{Z}} M_\Lambda
    \]
the \textbf{local algebra} and its elements are called \textbf{local observables}. The \textbf{support} of a local observable $a$ is the smallest $\Lambda\subset \mathbb{Z}$ such that $a\in M_{\Lambda}$. Given a state $\psi$ on $\A_\mathbb{Z}$, after \cite{HudsonMoody} we say it is \textbf{locally normal} if $\psi|_{M_\Lambda}$ is normal for all finite subsets $\Lambda\subset \mathbb{Z}$.

For $n,k\in \mathbb{Z}$, the map  $M_n\ni a\mapsto a\in M_{n+k}$ extends to a group action $\mathbb{Z}\overset{\alpha}{\curvearrowright} \A_{\mathbb{Z}}$.  We say a state $\psi$ on $\A_\mathbb{Z}$ is \textbf{translation invariant} if $\psi\circ\alpha_k = \psi$ for all $k\in \mathbb{Z}$.

    \begin{thm}[Proposition 2.3 and 2.5 \cite{FannesNachtergaeleWerner}]\label{thm:infinitedimensionalFCS}
        Let $\psi$ be a locally normal state on $\A_{\mathbb{Z}}$. Then, the following are equivalent: 
        \begin{enumerate}[label = (\roman*)]
            \item $\psi$ is translation invariant
            \item there exists a finite von Neumann algebra $W$, a normal unital completely positive map $\mathcal{E}:M\bar \otimes W \to W$, and a normal state $\varrho$ on $W$ so that for all $a_m\otimes\cdots \otimes a_n\in M_{[m,n]}$,
                \[
                     \psi( a_m  \otimes \cdots   \otimes a_{n} ) = \varrho\circ\mathcal{E}\circ (1\otimes \mathcal{E})\circ\cdots \circ (\underbrace{1\otimes \cdots \otimes 1}_{(n-m) \text{ times }} \otimes \mathcal{E})(a_m\otimes\cdots \otimes a_{n}\otimes 1_W).
                \]
        \end{enumerate} 
    \end{thm}
\begin{proof}
$(i)\Rightarrow (ii)$: Put $W = M_{\mathbb{N}}$, viewed as a subalgebra of $M_{\mathbb{Z}}$. Observe that the automorphism $\alpha_1$ extends to a normal automorphism of $M_\mathbb{Z}$ because it preserves $\tau_{\mathbb{Z}}$. In particular, $\alpha_1(W)\subset M_{[2,+\infty)}$ so that we can define $\mathcal{E}:= id_M\otimes \alpha_1$. Then, taking $\varrho = \psi|_{W}$, the claimed identity holds since the translation invariance of $\psi$ means it is index agnostic.\\

\noindent$(ii)\Rightarrow (i)$:  Mimicking \cite[Proposition 2.3]{FannesNachtergaeleWerner}, we see that the family of maps 
    \begin{equation}\label{eqn:quoteiteratesendquote}
        \mathcal{E}^{(n+1)}:M \bar \otimes\underbrace{M\bar \otimes \cdots M}_{n \text{ times }} \bar \otimes W \to W
    \end{equation} via $\mathcal{E}^{(n+1)} = \mathcal{E} \circ (\id_{M}\otimes \mathcal{E}^{(n)})$ and $\mathcal{E}^{(1)}:= \mathcal{E}$ is completely positive and normal for each $n$ and moreover 
    \[
        \psi(a_1   \otimes \cdots   \otimes a_n ) := \varrho( \mathcal{E}^{(n)}(a_1  \otimes \cdots   \otimes a_n   \otimes 1_W)),
    \] is positive and normal. Translation invariance of $\psi$ follows from the fact that $\mathcal{E}$ is a function of the observable and is independent of its index (translation does not change the operator itself, just its index). 
\end{proof}

\begin{rmk}
Assuming $M\neq \mathbb{C}$, the $W$ constructed in the above proof is necessarily infinite dimensional. However,
Fannes, Nachtergaele, and Werner considered conditions on $\psi$ that were necessary and sufficient to guarantee that $W$ could be chosen to be a finite dimensional algebra (see \cite[Propositions 2.1 and 2.3]{FannesNachtergaeleWerner}). $\hfill\blacksquare$
\end{rmk}

\subsection{Locally normal states with random generating map}
    
Throughout this section, let $(M,\tau)$ be a tracial von Neumann algebra with a separable predual, and let $(\Omega, \P)$ be a probability space equipped with ergodic $T\in \Aut(\Omega,\P)$. We will write $(W,\tau_W)$ for an auxiliary finite von Neumann algebra possessing a separable predual. When it is clear from context, we shall drop the subscript on $\tau_W$. 
    
Consider the von Neumann tensor product $M\bar \otimes W$ and let $\mathcal{E}_\omega:M\bar \otimes W \to M \bar \otimes W$ be a normal unital positive weak* random linear operator such that 
        \[
            \mathcal{E}_\omega(M\bar\otimes W) \subset \mathbb{C}\bar\otimes W \cong W
        \]
almost surely. Associated to such a weak* random linear operator, there is a family of normal weak* random linear operators $E_{\omega,a}:W\to W$ indexed by $a\in M$ and given by 
    \[ 
        W\ni x\mapsto E_{\omega, a}(x):= \mathcal{E}_\omega(a\otimes x).
    \] 
We denote $\phi_\omega:=E_{\omega,1}$, which is a normal unital positive weak* random linear operator on $W$, and we let $\Phi_{n,m}^{T^{-1}}$ be the associated quantum process on $W$. Similarly, we denote $\gamma_\omega:= (\phi_\omega)_*$, which is a bounded tracial (hence faithful) positive random linear operator on $L^1(W,\tau_W)$ by Lemma~\ref{lem:measurabilityTest}, and we let $\Gamma_{n,m}^T$ be the associated quantum process on $L^1(W,\tau_W)$. Recall that these two processes are dual to each other by Equation~(\ref{eqn:dual_quantum_process}).

Given integers $m<n$ and $\omega\in \Omega$, we define a map $\mathcal{E}_\omega^{[m,n]}\colon M_{[m,n]}\to W$ by
    \[
        \mathcal{E}^{[m,n]}_\omega(a):= \mathcal{E}_{T^m\omega}\circ (1\otimes \mathcal{E}_{T^{m+1}\omega})\circ\cdots \circ (1\otimes \cdots \otimes 1\otimes \mathcal{E}_{T^n\omega})(a\otimes 1_W),
    \]
which is almost surely normal unital and positive. Using our previous notation, for $a=a_m\otimes \cdots \otimes a_n\in M_{[m,n]}$ we have
    \[
         \mathcal{E}^{[m,n]}_\omega(a) = E_{T^m\omega, a_m}\circ \cdots \circ E_{T^n\omega, a_n}(1_W).
    \]
Consequently, for $m<k<\ell<n$ and $a\in M_{[k,\ell]}\subset M_{[m,n]}$ one has
    \begin{align}\label{eqn:hidden_quantum_process}
        \mathcal{E}^{[m,n]}_\omega(a) = \Phi_{-m,-k+1}^{T^{-1}}\circ \mathcal{E}_{\omega}^{[k,\ell]}(a)
    \end{align}
almost surely.

\begin{thm}[Thermodynamic Limit]\label{thm:thermo}
With the assumptions and notation as above, suppose that
    \[
        \P [\exists n,m\in \mathbb{Z}\colon n\geq m \text{ and }\eta \tau_{x_0} \leq \Phi_{m,n}^{T^{-1}} \leq \eta^{-1} \tau_{x_0} \text{ for some }\eta\in (0,1],\ x_0\in S] >0.    
    \]
Then there exists a map $\Psi\colon \Omega\times \A_{\mathbb{Z}}\to \C$ satisfying:
    \begin{enumerate}[label=(\arabic*)]
        \item $\Psi_\omega$ is a locally normal state on $\A_{\mathbb{Z}}$ for almost every $\omega\in \Omega$;

        \item $\Psi_\omega(a) \in L^\infty(\Omega,\P)$ for all $a\in \A_{\mathbb{Z}}$;

        \item $\Psi_\omega\circ \alpha_k= \Psi_{T^k\omega}$ almost surely for all $k\in \mathbb{Z}$;

        \item and for any local observable $a\in M_{\Lambda}$ one has
            \[
                \lim_{N\to\infty} \|\mathcal{E}^{[-N,N]}_\omega(a) - \Psi_\omega(a)\|_\infty =0        
            \]
        almost surely.
    \end{enumerate}
\end{thm}   
\begin{proof}
Let $Y_m\colon \Omega\to S(W)$, $m\in \mathbb{Z}$, be the family of random variables obtained by applying Theorem~\ref{thm:left_convergence_on_M} to the ergodic quantum process $\Phi_{m,n}^{T^{-1}}$ on $W$. Then the family $Z_m:= Y_{-m+1}$, $m\in \mathbb{Z}$, satisfies
    \[
        (\phi_{T^{m}\omega})_*(Z_m(\omega))= Z_{m+1}(\omega) \qquad \text{ and } \qquad  Z_m(T^{\pm 1}\omega)= Z_{m\pm 1}(\omega)
    \]
almost surely. Using Equation~(\ref{eqn:hidden_quantum_process}), we also have for $a\in M_{[m,n]}$ that
    \[
        \lim_{N\to\infty} \left\| \mathcal{E}^{[-N,N]}(a) - \tau_W(\mathcal{E}^{[m,n]}(a) Z_m) \right\|_\infty = \lim_{N\to\infty} \left\| \Phi_{N,-m+1}^{T^{-1}}(\mathcal{E}^{[m,n]}(a)) - \tau_W(\mathcal{E}^{[m,n]}(a) Y_{-m+1}) \right\|_\infty =0
    \]
almost surely. For $m<k<\ell<n$ and $a\in M_{[k,\ell]}\subset M_{[m,n]}$, using Equation~(\ref{eqn:hidden_quantum_process}) again and the above properties of $Z_m$ we have
    \begin{align*}
        \tau_W(\mathcal{E}_\omega^{[m,n]}(a) Z_m(\omega)) &= \tau_W( [\Phi_{-m,-k+1}^{T^{-1}}(\omega)](\mathcal{E}_{\omega}^{[k,\ell]}(a)) Z_m(\omega))\\
            &= \tau_W( \mathcal{E}_\omega^{[k,\ell]}(a) (\phi_{T^{k-1}\omega})_*\circ\cdots \circ (\phi_{T^m\omega})_*(Z_m(\omega))) = \tau_W( \mathcal{E}_\omega^{[k,\ell]}(a) Z_k(\omega))
    \end{align*}
almost surely. Noting that the $\omega$ where the above fails is independent of $a$, it follows that
    \[
        \Psi_\omega(a):= \tau_W(\mathcal{E}_\omega^{[-n,n]}(a) Z_{-n}(\omega)) \qquad a\in M_{[-n,n]}
    \]
almost surely gives a well-defined state on the $*$-subalgebra of local observables. As the local observables are norm dense in the quasi-local algebra, $\Psi_\omega$ almost surely admits a unique extension to a locally normal state on $\A_{\mathbb{Z}}$. Thus (1) holds and (4) follows from our limit computation above.

To see (2), recall that the separability of $L^1(W,\tau_W)$ implies $Z_m\colon \Omega\to S(W)$ is strongly measurable and can therefore be approximated by simple functions. Consequently, $\Psi_\omega(a)\in L^\infty(\Omega,\P)$ with $\|\Psi_\omega(a)\|_{L^\infty(\Omega,\P)} \leq \|a\|$ for local observables $a$ by definition of $\Psi_\omega$. This then holds for all elements of the quasi-local algebra through approximation by sequences of local observables.

Finally, towards showing (3) we first observe that for $a\in M_{[m,n]}$ and $k\in \mathbb{Z}$ one has
    \[
        \mathcal{E}_\omega^{[m+k,n+k]}(\alpha_k(a)) = \mathcal{E}_{T^k\omega}^{[m,n]}(a).
    \]
Combined with the properties of $Z_m$ above, we therefore have
    \[
        \Psi_\omega(\alpha_k(a)) = \tau_W( \mathcal{E}_\omega^{[m+k,n+k]}(a) Z_{m+k}(\omega)) = \tau_W(\mathcal{E}_{T^k\omega}^{[m,n]}(a) Z_m(T^k\omega)) = \Psi_{T^k \omega}(a). 
    \]
By density, this extends to $a\in \A_\mathbb{Z}$.
\end{proof}

\begin{thm}[{Theorem~\ref{thmx:C}}]\label{thm:Rclustering}
Let $\Psi\colon \Omega\times \A_{\mathbb{Z}}\to \C$ be as in Theorem~\ref{thm:thermo}. There exists $\kappa\in [0,1)$ and a family of random variables $E_k\colon \Omega\to [0,\infty)$, $k\in \mathbb{Z}$, satisfying
    \[
        E_k(T^\ell\omega) = E_{k+\ell}(\omega)
    \]
almost surely for all $\ell\in \mathbb{Z}$, and
    \[
        |\Psi_\omega(ab) - \Psi_\omega(a) \Psi_\omega(b)| \leq E_k(\omega) \kappa^{\text{dist}(\Lambda,\Pi)-1} \|a\|_\infty \|b\|_\infty \qquad \qquad a\in M_{\Lambda},\ b\in M_{\Pi}
    \]
almost surely for finite subsets of integers $\Lambda \subset (-\infty,k-1)$ and $\Pi\subset [k+1,+\infty)$.
\end{thm}
\begin{proof}
Let $C\in [0,1)$ be the constant obtained from applying Lemma~\ref{lem:randomconstproperties}.(1) to the dual quantum process $\Gamma^{T}_{n,m} = (\Phi_{-m,-n}^{T^{-1}})_*$. Set $\kappa$ to be any number in $(C,1)$. For $k\in \mathbb{Z}$,  we set
    \[
        E_k(\omega):= 8 D_{\bullet,k}(\omega) D_{k-1,\bullet}(\omega),
    \]
where $D_{\bullet,k}, D_{k-1,\bullet}$ are the random variables from Lemma~\ref{lem:randomconstproperties}.(2).
Now, for finite subsets of integers $\Lambda\subset (-\infty, k-1)$ and $\Pi\subset [k+1,+\infty)$, let $a\in M_\Lambda$ and $b\in M_\Pi$. Denote $c:=b - \Psi_\omega(b)\in M_\Pi$ so that
    \[
        \Psi_\omega(ab) - \Psi_\omega(a)\Psi_\omega(b)= \Psi_\omega(ac)
    \]
almost surely. Let $N\in \mathbb{N}$ be large enough so that $\Lambda\cup \Pi\subset [-N,N]$ and denote $m:=\max{\Lambda}$ and $n:=\min{\Pi}$. Note that $m<(m+1)\leq k-1<k \leq (n-1)< n$ and $\dist(\Lambda,\Pi)=n-m$. We have
    \begin{align*}
        |\Psi_\omega(ab) - \Psi_\omega(a)\Psi_\omega(b)| &= |\Psi_\omega(ac)|\\ &= \left| \tau_W(\mathcal{E}_\omega^{[-N,N]}(ac) Z_{-N}(\omega) ) \right|\\
            &= \left| \tau_W(\mathcal{E}_\omega^{[-N,m]}\left(a\otimes \Phi_{-(m+1),-(n-1)}^{T^{-1}}\circ \mathcal{E}_{\omega}^{[n,N]}(c)\right) Z_{-N}(\omega) ) \right|\\
            &\leq \|a\|_\infty \|\Phi_{-(m+1),-(n-1)}^{T^{-1}}\circ \mathcal{E}_{\omega}^{[n,N]}(c)\|_\infty \| (\mathcal{E}_\omega^{[-N,m]})_*(Z_{-N}(\omega)) \|_1\\
            &\leq \|a\|_\infty \|\Phi_{-(m+1),-(n-1)}^{T^{-1}}\circ \mathcal{E}_{\omega}^{[n,N]}(b) - \tau_W(\mathcal{E}_{\omega}^{[n,N]}(b) Z_n(\omega))\|_\infty,
    \end{align*}
where in the last inequality we have used that $\Phi_{-(m+1),-(n-1)}^{T^{-1}}\circ \mathcal{E}_{\omega}^{[n,N]}$ is unital and $(\mathcal{E}_\omega^{[-N,m]})_*$ is tracial, almost surely. By the proof of Theorem~\ref{thm:left_convergence_on_M}, we can estimate the second term in the last expression by
    \[
        \|\Phi_{-(m+1),-(n-1)}^{T^{-1}}\circ \mathcal{E}_{\omega}^{[n,N]}(b) - \tau_W(\mathcal{E}_{\omega}^{[n,N]}(b) Z_n(\omega))\|_\infty \leq 8\|b\|_\infty c(\Gamma_{ n-1,m+1}^T)
    \]
where we have used that $\mathcal{E}_{\omega}^{[n,N]}$ is unital and positive almost surely. Returning to our original estimate, the above and the properties of $D_{\bullet,k}, D_{k-1,\bullet}$ from Lemma~\ref{lem:randomconstproperties}.(2)  yield
    \begin{align*}
        |\Psi_\omega(ab) - \Psi_\omega(a)\Psi_\omega(b)| &\leq 8 c(\Gamma_{ n-1,m+1}^T) \|a\|_\infty \|b\|_\infty\\
        &\leq 8 c(\Gamma_{ n-1,k}^T)c(\Gamma_{ k-1,m+1}^T) \|a\|_\infty \|b\|_\infty\\
        &\leq 8 D_{\bullet,k}(\omega) \kappa^{(n-1)-k+1} D_{k-1,\bullet}(\omega) \kappa^{(k-1)-(m+1)+1} \|a\|_\infty \|b\|_\infty\\
        &= 8 D_{\bullet,k}(\omega) D_{k-1,\bullet}(\omega) \kappa^{n-m-1} \|a\|_\infty \|b\|_\infty = E_k(\omega) \kappa^{\text{dist}(\Lambda,\Pi) -1} \|a\|_\infty \|b\|_\infty,
    \end{align*}
almost surely.
\end{proof}

\begin{cor}
Let $\Psi\colon \Omega\times \A_{\mathbb{Z}}\to \C$ be as in Theorem~\ref{thm:thermo}. Then
    \[
        \bar\Psi(a):= \E[\Psi_\omega(a)]  
    \]
defines a locally normal translation invariant state on $\A_\mathbb{Z}$ such that
    \[
        \bar\Psi(a) = \lim_{N\to\infty} \frac{1}{2N+1} \sum_{n=-N}^N \Psi_{T^n\omega}(a) 
    \]
almost surely for all $a\in \A_{\mathbb{Z}}$.
\end{cor}
\begin{proof}
That $\bar\Psi$ is well-defined and a state follows from parts (1) and (2) of Theorem~\ref{thm:thermo}, while part (3) and the fact that $T$ is measure preserving gives that $\bar\Psi$ is translation invariant. The limit formula follows from Birkhoff's strong ergodic theorem.

It remains to show that $\bar\Psi$ is locally normal, and it suffices to show its restriction to $M_{[m,n]}$ is normal. Recall that for $a\in M_{[m,n]}$ that one has
    \[
        \Psi_\omega(a) = \tau_W( \mathcal{E}_\omega^{[m,n]}(a)Z_m(\omega)),
    \]
which is almost surely normal since $\mathcal{E}_\omega$ is almost surely normal and $Z_m(\omega)\in L^1(M,\tau_W)$. When this is the case, set $X_m(\omega)\in S(M)$ so that
    \[
        \tau(a X_m(\omega)) = \Psi_\omega(a),
    \]
and otherwise let $X_m(\omega)=1$. Part (2) of Theorem~\ref{thm:thermo} implies $X_m\colon \Omega\to S(M)$ is weakly measurable, and hence a random variable by Theorem~\ref{thm:Pettis}. Consequently, there exists a a sequence of simple functions $\phi_k\colon \Omega\to L^1(M,\tau)$ satisfying $\|X_m(\omega) - \phi_k(\omega)\|_1 \to 0$ as $k\to\infty$ almost surely. Note that we may assume $\phi_k(\omega)\in S(M)$ almost surely, and therefore the dominated convergence theorem implies
    \[
        \lim_{k\to\infty} \int_\Omega \|X_m(\omega) - \phi_k(\omega)\|_1\ d\P(\omega) =0.
    \]
It follows that $(\int_\Omega \phi_k\ d\P)_{k\in \mathbb{N}}$ is a Cauchy sequence in $L^1(M,\tau)$, and if we denote the limit by $\E[X_m]$ then
    \[
        \bar\Psi(a) = \int_\Omega \tau(a X_m)\ d\P = \lim_{k\to\infty} \int_\Omega \tau(a \phi_k)\ d\P =  \lim_{k\to\infty}  \tau(a \int_\Omega \phi_k \ d\P)=\tau( a \E[X_m]),
    \]
for all $a\in M_{[m,n]}$.
\end{proof}


\bibliographystyle{amsalpha}
\bibliography{bibliography}

\providecommand{\bysame}{\leavevmode\hbox to3em{\hrulefill}\thinspace}
\providecommand{\MR}{\relax\ifhmode\unskip\space\fi MR }
\providecommand{\MRhref}[2]{%
  \href{http://www.ams.org/mathscinet-getitem?mr=#1}{#2}
}
\providecommand{\href}[2]{#2}
\begin{thebibliography}{GGJN18}

\bibitem[AB]{AvilaBochi}
A.~Avila and J.~Bochi, \emph{On the subaddive ergodic theorem}.

\bibitem[ADP19]{ap}
C.~Anatharaman-Delaroche and S.~Popa, \emph{An introduction to {$II_1$}
  factors}, {https://www.math.ucla.edu/~popa/Books/IIun.pdf}, 2019.

\bibitem[Ara74]{araki}
H.~Araki, \emph{{Some properties of modular conjugation operator of von Neumann
  algebras and a non-commutative Radon-Nikodym theorem with a chain rule.}},
  Pacific Journal of Mathematics \textbf{50} (1974), no.~2, 309 -- 354.

\bibitem[BJP22]{BougronJoyePillet}
J-F. Bougron, A.~Joye, and C-A. Pillet, \emph{Markovian repeated interaction
  quantum systems}, Reviews in Mathematical Physics \textbf{34} (2022), no.~09,
  2250028.

\bibitem[BP07]{BreuerPetruccione}
H-P. Breuer and F.~Petruccione, \emph{{The Theory of Open Quantum Systems}},
  Oxford University Press, 01 2007.

\bibitem[BR72]{Bharucha-Reid}
A.~T. Bharucha-Reid, \emph{Chapter 2 operator-valued random variables}, Random
  Integral Equations, Mathematics in Science and Engineering, vol.~96,
  Elsevier, 1972, pp.~64--97.

\bibitem[BR81]{BratteliRobinson2}
O.~Bratteli and D.~W. Robinson, \emph{Operator algebras and quantum statistical
  mechanics ii}, 1 ed., Springer Berlin, Heidelberg, 1981.

\bibitem[BR87]{BratteliRobinson1}
\bysame, \emph{Operator algebras and quantum statistical mechanics 1}, 2 ed.,
  Springer Berlin, Heidelberg, 1987.

\bibitem[CL90]{CarLa}
R.~Carmona and J.~Lacroix, \emph{Products of random matrices}, pp.~175--240,
  Birkh{\"a}user Boston, Boston, MA, 1990.

\bibitem[DUJ77]{DiestelUhl}
J.~Diestel and J.~J. Uhl~Jr., \emph{Vector measures}, Mathematical Surveys,
  no.~15, American Mathematical Society, 1977.

\bibitem[Far96]{farenik}
D.~R. Farenick, \emph{Irreducible positive linear maps on operator algebras},
  Proceedings of the American Mathematical Society \textbf{124} (1996), no.~11,
  3381--3390.

\bibitem[FNW92]{FannesNachtergaeleWerner}
M.~Fannes, B.~Nachtergaele, and R.~F. Werner, \emph{Finitely correlated states
  on quantum spin chains}, Communications in Mathematical Physics \textbf{144}
  (1992), no.~3, 443--490.

\bibitem[GGJN18]{GGJN18}
C.~E. González-Guillén, M.~Junge, and I.~Nechita, \emph{On the spectral gap
  of random quantum channels}, 2018.

\bibitem[Hen97]{Hennion}
H.~Hennion, \emph{{Limit theorems for products of positive random matrices}},
  The Annals of Probability \textbf{25} (1997), no.~4, 1545 -- 1587.

\bibitem[Hia21]{hiai}
F.~Hiai, \emph{Quantum f-divergences in von {Neumann} algebras}, 1 ed.,
  Mathematical Physics Studies, Springer Singapore, 2021.

\bibitem[HM76]{HudsonMoody}
R.~L. Hudson and G.~R. Moody, \emph{Locally normal symmetric states and an
  analogue of {de Finetti's} theorem}, Zeitschrift f{\"u}r
  Wahrscheinlichkeitstheorie und Verwandte Gebiete \textbf{33} (1976), no.~4,
  343--351.

\bibitem[HOT81]{HiaiOhyaTsukuda}
F.~Hiai, M.~Ohya, and M.~Tsukuda, \emph{Sufficiency, {KMS} condition and
  relative entropy in von {N}eumann algebras}, Pacific J. Math. \textbf{96}
  (1981), no.~1, 99--109.

\bibitem[Kin73]{king}
J.~F.~C. Kingman, \emph{{Subadditive Ergodic Theory}}, The Annals of
  Probability \textbf{1} (1973), no.~6, 883 -- 899.

\bibitem[Kos86]{Kosaki}
H.~Kosaki, \emph{Relative entropy of states: a variational expression}, Journal
  of Operator Theory \textbf{16} (1986), no.~2, 335--348.

\bibitem[Lal19]{lalley}
S.~Lalley, \emph{Kingman's subadditive ergodic theorem}, 2019.

\bibitem[MS22]{jeff}
R.~Movassagh and J.~Schenker, \emph{An ergodic theorem for quantum processes
  with applications to matrix product states}, Communications in Mathematical
  Physics \textbf{395} (2022), no.~3, 1175--1196.

\bibitem[Ole12]{olesen}
K.K. Olesen, \emph{The {Connes Embedding Problem} {Sofic} groups and {QWEP}
  conjecture}, 2012.

\bibitem[Pau03]{paulsen}
V.~Paulsen, \emph{Completely bounded maps and operator algebras}, Cambridge
  Studies in Advanced Mathematics, Cambridge University Press, 2003.

\bibitem[Pet38]{pettis}
B.~J. Pettis, \emph{On integration in vector spaces}, Trans. Amer. Math. Soc.
  (1938), no.~44, 277--304.

\bibitem[PS23]{pathirana2023law}
L.~Pathirana and J.~Schenker, \emph{Law of large numbers and central limit
  theorem for ergodic quantum processes}, 2023.

\bibitem[Sko84]{Skorohod}
A.~V. Skorohod, \emph{Random linear operators}, Mathematics and Its
  Applications (\emph{Soviet Series}), D. Reidel Publishing Company, 1984.

\bibitem[Ste89]{steele}
J.~M. Steele, \emph{Kingman's subadditive ergodic theorem}, Annales de l'I.H.P.
  Probabilit\'es et statistiques \textbf{25} (1989), no.~1, 93--98 (en).
  \MR{995293}

\bibitem[Tak02]{TakesakiI}
M.~Takesaki, \emph{Theory of operator algebras. {I}}, Encyclopaedia of
  Mathematical Sciences, vol. 124, Springer-Verlag, Berlin, 2002, Reprint of
  the first (1979) edition, Operator Algebras and Non-commutative Geometry, 5.
  \MR{1873025}

\bibitem[Tak03]{TakesakiIII}
\bysame, \emph{Theory of operator algebras. {III}}, Encyclopaedia of
  Mathematical Sciences, vol. 127, Springer-Verlag, Berlin, 2003, Operator
  Algebras and Non-commutative Geometry, 8. \MR{1943007}

\bibitem[Yea77]{Yea77}
F.~J. Yeadon, \emph{Ergodic theorems for semifinite von {N}eumann algebras.
  {I}}, J. London Math. Soc. (2) \textbf{16} (1977), no.~2, 326--332.
  \MR{487482}

\bibitem[Yea80]{Yea80}
\bysame, \emph{Ergodic theorems for semifinite von {N}eumann algebras. {II}},
  Math. Proc. Cambridge Philos. Soc. \textbf{88} (1980), no.~1, 135--147.
  \MR{569639}

\end{thebibliography}
\end{document}